\documentclass[11pt]{article}
\usepackage[utf8]{inputenc}
\usepackage{setspace,mathtools,amsmath,amssymb,amsthm,fullpage,outline,centernot,slashed}
\usepackage{bbm}
\usepackage{xcolor}
\usepackage{tikz}
\newcommand{\Z}{\mathbb{Z}}

\newcommand{\R}{\mathbb{R}}
\newcommand{\N}{\mathbb{N}}

\DeclareMathOperator{\curl}{curl}
\DeclareMathOperator{\divr}{div}

\newtheorem{theorem}{Theorem}

\newtheorem{lemma}[theorem]{Lemma}

\newtheorem{proposition}[theorem]{Proposition}

\begin{document}

\title{Global stability for a nonlinear system of \\ anisotropic wave equations}
\author{John Anderson}
\date{\today}

\maketitle

\begin{abstract}
    In this paper, we initiate the study of global stability for anisotropic systems of quasilinear wave equations. Equations of this kind arise naturally in the study of crystal optics, and they exhibit birefringence. We introduce a physical space strategy based on bilinear energy estimates that allows us to prove decay for the nonlinear problem. This uses decay for the homogeneous wave equation as a black box. The proof also requires us to interface this strategy with the vector field method and take advantage of the scaling vector field. A careful analysis of the spacetime geometry of the interaction between waves is necessary in the proof.
\end{abstract}

\tableofcontents

\section{Introduction} \label{sec:introduction}
In this paper, we initiate the study of the global stability properties of anisotropic systems of wave equations. With fixed parameters $\lambda_1$ and $\lambda_2$, a good example to keep in mind is

\begin{equation} \label{eq:anisotropicintro}
\begin{aligned}
\Box \psi = -\partial_t^2 \psi + \partial_x^2 \psi + \partial_y^2 \psi = (\partial_t \psi)^2 (\partial_t \phi)
\\ \Box' \phi = - \partial_t^2 \phi + \lambda_1^{-2} \partial_x^2 \phi + \lambda_2^{-2} \partial_y^2 \phi = (\partial_t \phi)^2 (\partial_t \psi)
\end{aligned}
\end{equation}
in $2 + 1$ dimensions (see Theorem~\ref{thm:mainthm} and the discussion thereafter for a precise description of the kinds of admissible nonlinearities). We shall specifically be concerned with proving stability of the trivial solution. Because the problem is in $2 + 1$ dimensions and waves decay at a rate of $t^{-{1 \over 2}}$, this system is critical in terms of decay. We show that, nonetheless, the trivial solution is globally stable assuming that a kind of \emph{null condition} is present. This null condition will require that $\lambda_1 \ne 1$ and $\lambda_2 \ne 1$ in the above, guaranteeing that the light cones cones for the equations intersect transversally. Interactions involving at least one factor of each wave will then satisfy a kind of null condition (see the discussion after Theorem~\ref{thm:mainthm}).

This problem is motivated by trying to understand phenomena related to \emph{birefringence}, which is associated with different waves moving at different speeds in different directions. The study of anisotropic systems of wave equations was introduced to the author by Sergiu Klainerman in the larger context of studying hyperbolic equations whose characteristics have multiple sheets (see Sections~\ref{sec:anisotropicintro} and \ref{sec:multiplecharacteristics}).  The anisotropic system we study here is closely related to a system of equations that arises from applying a symmetry reduction to the equations governing the propagation of light in a \emph{biaxial crystal} (see Section~\ref{sec:multiplecharacteristics}).

In order to prove global nonlinear stability, we shall introduce a strategy for proving estimates based on bilinear energy estimates and a duality argument. This paper also aims to describe this strategy. One of the main difficulties in this problem is that the equations do not share many symmetries, and even though both equations are in reality the $\Box$ operator associated to the flat metric, the Lorentz operators associated to both equations are different. One can check that the only weighted commutator (see \eqref{eq:VFs}) from the vector field method that can effectively be used is the scaling vector field $S = t \partial_t + r \partial_r$. Commuting with other weighted vector fields creates bad terms because, for example, the commutator adapted to $\psi$ could hit $\phi$, introducing growing weights in $t$. We must still use the fact that $S$ can effectively be used as a commutator in a fundamental way.


We now briefly describe the role of billinear energy estimates. If $\psi$ and $f$ are such that $\Box \psi = \Box f = 0$, then $\int_{\Sigma_t} \partial_t \psi \partial_t f + \partial^i \psi \partial_i f d x$ is independent of $t$. Thus, solutions of the homogeneous wave equation have infinitely many conserved quantities because their inner products with other solutions are constant in $t$. A more geometric way of saying this is that pairs of solutions to the wave equation obey bilinear spacetime integration by parts formulas, and these formulas are in fact available on general globally hyperbolic Lorentzian manifolds (see Section~\ref{sec:bilinearestimates}). Our strategy is, then, to prove estimates for $\psi$ by using these bilinear estimates and making a good choice for the data for the other solution $f$ (see Section~\ref{sec:decayestimates} where this is described in more detail). We believe that this idea can be used in other settings, but we focus on using this to prove decay and stability statements in this paper. The class of data chosen for $f$ can be thought of as testing $\psi$ against smoothed out fundamental solutions.

We require three ingredients in order to apply this strategy to global stability. The first tool consists of bilinear integration by parts formulas like the bilinear energy estimates (see Section~\ref{sec:bilinearestimates}). The second is decay rates for solutions to the homogeneous wave equation with a sufficiently large class of data. The third is a duality argument. We now elaborate.

\begin{enumerate}
\item Everything is purely physical space, relying on bilinear energy estimates. Bilinear energy estimates are well known, but we shall use them to track how much energy of the solution is present in various scale $1$ balls (see Section~\ref{sec:decayintro}). These estimates and their counterparts in general globally hyperbolic Lorentzian manifolds are discussed in Section~\ref{sec:bilinearestimates}.
\item The method uses decay for solutions to the homogeneous wave equation as a black box. Decay for the homogeneous wave equation may be established using whatever method one prefers. The rates must only be sufficiently strong for the application at hand.
\item Pointwise decay follows from the bilinear energy estimates and pointwise decay for solutions to the homogeneous wave equation arising from a sufficiently large class of test functions. Indeed, controlling the integral of $\psi$ against a large class of test functions gives estimates on norms of $\psi$ using a duality argument. We will also use the fact that we can freely commute with translation vector fields.
\end{enumerate}
The proof of global stability for anisotropic systems of wave equations will require us to also use the scaling vector field $S = t \partial_t + r \partial_r$ in a fundamental way. This follows the philosophy introduced by Klainerman in \cite{Kl85} showing the importance of taking advantage of operators that commute with the equation (see Sections~\ref{sec:history}, \ref{sec:anisotropicintro}, and \ref{sec:multiplecharacteristics}).

In practice, most of the work goes into controlling the nonlinear errors (see Section~\ref{sec:decayestimates} to see how these terms appear). In order to better describe the method, we shall prove two easier stability results in before turning to study anisotropic systems of wave equations. We believe that these examples will be instructive in how bilinear energy estimates can be effectively used.

The rest of the introduction consists of a discussion of existing work and how it relates to this paper (Section~\ref{sec:history}), a more detailed description of the strategy used to prove decay (Section~\ref{sec:decayintro}), a description of the first simple applications (Section~\ref{sec:simpleintro}), a description of anisotropic systems of wave equations (Section~\ref{sec:anisotropicintro}), and a discussion of some related directions (Section~\ref{sec:relateddirections}). We end this introductory part with a discussion in Section~\ref{sec:multiplecharacteristics} specifically about how the study of anisotropic wave equations fits into a larger research program concerning hyperbolic equations with multiple characteristics. Then, Section~\ref{sec:decayestimates} uses the bilinear estimates we have described and a duality argument to prove the decay estimates we shall use in the rest of the paper. Section~\ref{sec:simpleapps} uses the decay estimates in two simple nonlinear settings. Finally, Section~\ref{sec:anisotropic} proves global stability of the trivial solution for an anisotropic system of wave equations satisfying the null condition. This final section is the most substantial, and contains the main Theorem in the paper (Theorem~\ref{thm:mainthm}).

\subsection{Nonlinear wave equations and related results} \label{sec:history}
We shall now provide a short history of stability results for nonlinear wave equations, and we shall also discuss some previous results using techniques which are similar in spirit to some of the strategies used in this paper. The reader may wish to look at this in conjunction with Section~\ref{sec:decayintro} in order to better understand the connections with previous work.

The study of global stability for nonlinear wave equations started with the work of Klainerman in \cite{Kla80}. The stability mechanism comes from the fact that solutions to the wave equation arising from localized initial data decay at quantitative rates as $t \rightarrow \infty$. This dispersive behavior allows one to control the contribution of the nonlinearity and treat it perturbatively. In this pioneering work, Klainerman used the dispersive estimate to take advantage of decay, and this allowed him to prove global stability for a large class of nonlinear wave equations. The proof was later simplified by Klaineramn-Ponce in \cite{KlaPon83}. Then, in the groundbreaking work \cite{Kl85}, Klainerman introduced the \emph{vector field method}, allowing him to prove decay for solutions to nonlinear equations by using a collection of weighted vector fields that satisfy good commutation properties with $\Box$. He was able to use this method to show global stability of the trivial solution for cubic nonlinearities in $3 + 1$ dimensions, and for quadratic nonlinearities in $n + 1$ dimensions for $n \ge 4$. More specifically, the family of weighted vector fields consisted of the Lorentz and scaling vector fields, which can be represented by
\begin{equation} \label{eq:VFs}
    \begin{aligned}
    S = t \partial_t + r \partial_r \hspace{5 mm} \text{and} \hspace{5 mm} \Omega_{\mu \nu} = x_\mu \partial_\nu - x_\nu \partial_\mu,
    \end{aligned}
\end{equation}
where $x^\mu$ represent coordinates on $n + 1$ dimensional Minkowski space. Using these weighted vector fields leads to a weighted Sobolev inequality called the \emph{Klainerman-Sobolev inequality}, which reads
\begin{equation} \label{eq:KlaSob}
    \begin{aligned}
    |f| (t,r,\omega) \le {C \over (1 + t + r)^{{n - 1 \over 2}} (1 + |t - r|^{{1 \over 2}})} \sum_{|\alpha| \le \lceil n / 2 \rceil} \Vert \Gamma^\alpha f \Vert_{L^2 (\Sigma_t)},
    \end{aligned}
\end{equation}
where $\Sigma_t$ denotes the constant $t$ hypersurface, and where $\Gamma^\alpha$ represents strings of the weighted vector fields introduced above along with the translation vector fields. As was described in Section~\ref{sec:introduction}, we will not have access to all of the weighted commutators in \eqref{eq:VFs}. Nonetheless, we do have access to the scaling vector field, and this plays a key role in the proof of the main Theorem, Theorem~\ref{thm:mainthm}, in Section~\ref{sec:anisotropic}.

This restriction in $3 + 1$ dimensions is fundamental, as was shown by John in \cite{Joh81}. In this work, John showed that general quadratic nonlinearities can lead to blow up in finite time, even when the data are arbitrarily small (but nonzero) and compactly supported. An example of an equation that experiences this blow up is
\[
\Box \phi = -(\partial_t \phi)^2.
\]
The work of Klainerman \cite{Kl85} was preceded by work of John-Klainerman in \cite{JohKla84} where they were able to establish almost global existence results for nonlinear wave equations having quadratic nonlinearities in $3 + 1$ dimensions. This work involved using the fundamental solution for spherically symmetric solutions to the wave equation. The error from being spherically symmetric was controlled using the rotation vector fields, and this is one of the first papers where the usefulness of weighted vector fields for global problems was apparent. As was described above, the class of data we allow for the auxiliary multipliers makes them behave like smoothed out fundamental solutions, so the error integrals we must control have similarities with the integrals encountered by John-Klainerman in this work.

Several physical systems in $3 + 1$ dimensions are modeled by hyperbolic equations having quadratic nonlinearities, meaning that any global stability result for these systems would have to have some mechanism that avoids the examples studied by John in \cite{Joh81}. These equations often have nonlinearities which exhibit some kind of \emph{null condition}, originally introduced by Klainerman in \cite{Kla82}. In the context of nonlinear wave equations in $3 + 1$ dimensions, quadratic nonlinearities obeying the null condition are better behaved and satisfy improved estimates. These improved estimates lead to global stability of the trivial solution (see \cite{Kla86} and \cite{Chr86}), as opposed to the examples studied by John in \cite{Joh81}. Then, in the monumental work \cite{ChrKl93}, Christodoulou-Klainerman were able to show that Minkowski space is globally nonlinearly stable with respect to sufficiently small and localized perturbations as a solution to the Einstein vacuum equations. Their work heavily used the vector field method, and it also required identifying a kind of null condition present in the Einstein vacuum equations.

These themes have continued to be the subject of intense study ever since. One prominent direction among much of this recent work is that it involves situations in which we do not have access to all of the vector fields in \eqref{eq:VFs} for one reason or another. The first such work was by Klainerman-Sideris in \cite{KlaSid96} in which they developed a version of the vector field method to study hyperbolic equations without the use of boosts, motivated by studying physical systems without this symmetry. We note that the scaling vector field played a fundamental role in that work, similar to this one. In addition, we mention the exciting recent work concerning black hole stability (see, for example, \cite{DafRodShl16}, \cite{DafHolRod19}, and \cite{KlaSze20} and the references therein). These works required extending the philosophy of using weighted vector fields as commutators and multipliers to nontrivial backgrounds.

New manifestations of some kind of null condition have also been found in various equations. In the context of nonlinear wave equations, Lindblad-Rodnianski identified a weaker form of the null condition called the \emph{weak null condition} which is present in the Einstein vacuum equations in wave gauge. This allowed them to provide a different proof of the stability of Minkowski space based on wave gauge in the works \cite{LinRod03}, \cite{LinRod05}, and \cite{LinRod10} (see also recent generalizations by Keir in \cite{Kei18}).

In addition to these purely physical space methods, there have been other approaches to study wave equations and other equations with a dispersive mechanism. One approach, called the method of \emph{spacetime resonances}, involves using the Duhamel formula for the propagator to study the nonlinear interactions in a very precise way. This method was introduced in \cite{GerMasSha09}, and it takes advantage of both oscillations in the nonlinearity and differences in group velocity. The main difficulty is then in understanding what happens in regions where the oscillations match with the oscillations of the linear flow (time resonances) and in regions where wave packets interact for a long time (space resonances). Among the vast literature concerning the spacetime resonances method, we mention in particular the hyperbolic applications found in the works \cite{PusSha13} and \cite{DenPus20}. The first concerns the long time dynamics of solutions to equations satisfying the null condition, while the second concerns the long time dynamics of solutions to equations satisfying the weak null condition.

We mention also the wave packet approach to studying wave equations. Using phase space decompositions such as wave packets has proven to be very useful in the study of wave equations (we note in particular the celebrated result \cite{SmiTat05} of Smith-Tataru which provides a sharp local well posedness result in $H^s$ spaces for general quasilinear wave equations). Wave packets are often constructed by tilings in phase space which saturate the uncertainty principle, and the tilings are chosen such that the wave equation approximately becomes a transport equation on each piece for some time scale. Other constructions of wave packets are given in the work \cite{KlaRodTao02} of Klainerman-Rodnianski-Tao. This paper provides novel proofs of bilinear estimates which were motivated by applications to low regularity results, and it also contains two wave packet constructions different from the one described above. One of these wave packet constructions involves ``two time cutoffs" in which a solution to the wave equation is truncated at both ends of a tube in order to construct a wave packet adapted to that tube. This construction provided inspiration for the strategy used in this paper.

The work of Ifrim-Tataru in \cite{IfrTat15} describes how one can use wave packets to study global problems for nonlinear dispersive equations. They derive effective ODEs for inner products with wave packets, and this allows them to prove modified scattering for the solution. This philosophy of testing against other functions is also used in this paper (see Section~\ref{sec:decay}).

Another use of testing against wave packets is found in the paper \cite{JeoOh19} by Jeong-Oh. Unlike the other examples described above, this paper proves ill posedness for certain regimes of the Hall and electron magnetohydrynamic equations, so this may seem out of place. However, the proof of ill posedness involves tracking inner products with wave packets solving the adjoint equation to a certain linearized equation, providing another example of the philosophy of studying an equation by testing its solutions against other functions (see Section~\ref{sec:decay}).

As will be described in Section~\ref{sec:anisotropicdescription}, a careful analysis of both the geometry of interaction (in particular the volume of interaction) and frame decompositions is important when studying anisotropic systems of wave equations (the geometry in general is important throughout the paper). We note that the author has previously worked on other problems where the same kinds of ideas have been useful. The work \cite{AndPas19} joint with Pasqualotto studied the stability of the trivial solution to nonlinear wave equations with respect to perturbations localized around several points, while the work \cite{AndZba20} joint with Zbarsky studied the stability and instability of plane wave solutions.

\subsection{Proving decay using bilinear energy estimates} \label{sec:decayintro}
Let us now describe the strategy we will use to study the solutions to the equations in question. We shall use the notation $\Sigma_t$ to denote the usual constant $t$ spacelike hypersurfaces, and we shall denote by $r$ the usual ``spatial" radial coordinate. For this schematic discussion, we shall assume that we want to study the solutions to the equation
\[
\Box \psi = F
\]
in $n + 1$ dimensions. Because we are interested in proving decay for localized data, we shall assume that the data for $\psi$ is smooth and compactly supported in the unit ball in $\Sigma_0$. Moreover, because $F$ is a nonlinearity involving $\psi$ in practice, we shall take $F$ to be an arbitrary smooth function supported where $t - r \ge -1$.

We shall use test functions and a duality argument (along with commuting the equation for $\psi$ with translation vector fields) in order to prove statements about the solution to the equation. The goal is to show that $\psi$ behaves like a solution to the homogeneous equation. Solutions to the homogeneous wave equation arising from sufficiently regular data supported in the unit ball (i.e., the case of $F = 0$ above) follow a specific profile. The solution is concentrated along the light cone $t = r$ in a spherically symmetric way that is consistent with conservation of $\partial_t$ energy, where we recall that the $\partial_t$ energy of $f$ through $\Sigma_t$ is given by
\[
{1 \over 2} \int_{\Sigma_t} (\partial_t \psi)^2 + \sum_{i = 1}^n (\partial_i \psi)^2 d x.
\]
An annular region $A_t \subset \Sigma_t$ of thickness $2$ and inner radius $t - 1$ has volume comparable to $t^{n - 1}$. This is precisely the region along the light cone $t = r$ in $\Sigma_t$ where we expect all of the energy of the solution to be concentrated if we apply the heuristic principle that waves propagate at the speed of light. In order to be consistent with energy conservation and being almost spherically symmetric, the solution must be of size $r^{-{n - 1 \over 2}} \approx t^{-{n - 1 \over 2}}$ along the light cone. A spacelike ball $B \subset \Sigma_t$ of radius comparable to $1$ whose center lies in $A_t$ should, therefore, contain energy comparable to $t^{-(n - 1)}$. This is because the ball has volume comparable to $1$, and it would take roughly $t^{n - 1}$ such balls to cover the annular region. Thus, because the energy in each ball should be roughly the same, each ball should have energy roughly comparable to $t^{-(n - 1)}$. See Figure~\ref{fig:energyannulus} for a picture of this configuration. The fact that each such ball has energy comparable to $t^{-(n - 1)}$ is usually a consequence of commuting the equation with rotation vector fields.

\begin{figure}
    \centering
    \begin{tikzpicture}
    \draw[very thick] (0,0) circle (4);
    \draw[very thick] (0,0) circle (3.7);
    \draw (3.8,0) circle (0.2);
    \draw (3.85,0.3) circle (0.2);
    \draw (3.8,0.55) circle (0.2);
    \draw (3.78,0.8) circle (0.2);
    
    \end{tikzpicture}
    \caption{This annulus depicts the region containing most of the energy of a solution to the homogeneous wave equation in $2 + 1$ dimensions arising from smooth data supported in the unit disk in $\Sigma_t$. This annulus is a subset of $\Sigma_t$, has an outer radius of $t + 1$, and has an inner radius of $t - 1$. Thus, the annulus has area roughly $t$. We have covered a small portion of the annulus with disks of radius comparable to the thickness of the annulus, which is comparable to $1$. It would take roughly $t$ such disks to cover this annulus.}
    \label{fig:energyannulus}
\end{figure}
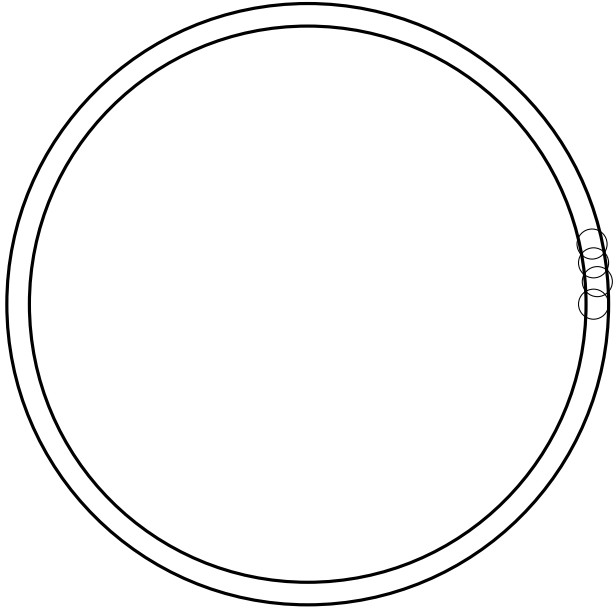

The motivation of the strategy used in this paper comes from trying to directly show that the energy will equidistribute among the balls in this way (every ball should roughly have energy $t^{n - 1}$). Because this ball is of radius comparable to $1$, a unit scale Sobolev embedding will then imply a pointwise decay rate of $t^{{n - 1 \over 2}}$, which is the expected rate. For applications to anisotropic systems of wave equations as in Section~\ref{sec:anisotropic}, this is a suitable way to show decay because we are still able to commute with translation vector fields. This procedure of showing that the energy equidistributes in this way among the balls can be thought of as replacing commuting with rotation vector fields.

In reality, the energy of the solution will not be located only along the light cone $t = r$, but it will also be present within the light cone. However, it is well known that the energy should decay away from the light cone. Thus, in addition to showing directly that the energy roughly equidistributes as if the solution was spherically symmetric, we shall also show that the energy in a unit sized ball decays as the center of the ball moves away from the light cone $t = r$. This replaces showing decay in $u$ by commuting with weighted vector fields (although we shall have to figure out how to combine this method with commuting with the scaling vector field to study anisotropic wave equations in Section~\ref{sec:anisotropic}).

We now describe how bilinear estimates allow us to show that the energy of the solution behaves in this way. Suppose we are given a scale $1$ ball $B \subset \Sigma_s$. We let $\tau$ denote the $u$ coordinate of the center of this ball. Thus, the spacetime coordinates of the center of the ball $B$ are $(s,x_0^1,\dots,x_0^n)$ where $s - \sqrt{(x_0^1)^2 + \dots + (x_0^n)^2} = \tau$. Now, suppose that we wish to show that the energy of the solution present in this ball is of the expected size. This means that we wish to show that
\begin{equation} \label{eq:introendec}
    \begin{aligned}
    \Vert \partial \psi \Vert_{L^2 (B)}^2 \approx {C \over (1 + s)^{n - 1} (1 + |\tau|)^{2 p}},
    \end{aligned}
\end{equation}
where $p$ measures the decay away from the light cone of the solution and depends on the application.

Now, let $f$ be an arbitrary solution to the homogeneous wave equation $\Box f = 0$. Because $\psi$ solves the equation $\Box \psi = F$, we note that $\Box \partial^\alpha \psi = \partial^\alpha F$. The estimates we shall use are of the form
\begin{equation} \label{eq:testest}
\begin{aligned}
\int_{\Sigma_s} (\partial_t \partial^\alpha \psi) (\partial_t f) + \sum_{i = 1}^n (\partial_i \partial^\alpha \psi) (\partial_i f) d x
\\ = \int_{\Sigma_0} (\partial_t \partial^\alpha \psi) (\partial_t f) + \sum_{i = 1}^n (\partial_i \partial^\alpha \psi) (\partial_i f) d x - \int_0^s \int_{\Sigma_t} (\partial^\alpha F) (\partial_t f) d x d t.
\end{aligned}
\end{equation}
These integration by parts identities are simply bilinear energy estimates, and we note that versions of this hold in general globally hyperbolic Lorentzian manifolds (see Section~\ref{sec:decay}). When $F = 0$, we note that this identity reduces to the fact that the inner product in $\dot{H}^1$ is a conserved quantity for pairs of solutions to the homogeneous wave equation.

This identity contains a lot of freedom. Because we are interested in getting estimates for $\psi$, the data for $\psi$ and $F$ are both rigid. However, we have not specified the data for $f$, and this identity holds regardless of how we pick data for $f$. Because we are interested in proving estimates for $\psi$ in $B$, it is natural to allow the data for $f$ to be supported in, say, $2 B$, which is the ball of twice the radius of $B$ that is centered at the same point. Let us denote the trace of $f$ in $\Sigma_s$ by $f_0$, and let us denote the trace of its time derivative by $f_1$. The functions $f_0$ and $f_1$ can be freely specified to give a unique solution $f$ to the homogeneous wave equation. We are thus restricting $f_0$ and $f_1$ to be supported in $2 B$. In this case, the above identity becomes
\begin{equation} \label{eq:introlocen}
\begin{aligned}
\int_{2 B} (\partial_t \partial^\alpha \psi) (f_1) + \sum_{i = 1}^n (\partial_i \partial^\alpha \psi) (\partial_i f_0) d x
\\ = \int_{\Sigma_0} (\partial_t \partial^\alpha \psi) (\partial_t f) + \sum_{i = 1}^n (\partial_i \partial^\alpha \psi) (\partial_i f) d x - \int_0^s \int_{\Sigma_t} (\partial^\alpha F) (\partial_t f) d x d t.
\end{aligned}
\end{equation}
By this choice of data for the auxiliary multiplier $f$, we note that can use a bilinear energy estimate to bound averages of derivatives of $\psi$ against the functions $f_0$ and $f_1$. If we allow $f_0$ and $f_1$ to vary among a suitable class of functions which are supported in $2 B$, a duality argument will result in estimates for $\psi$ and its derivatives localized to $B$ (see Section~\ref{sec:dualityargument}). These localized estimates for $\psi$ are what we set out to prove.

We now briefly discuss one way in which \eqref{eq:introlocen} can be used to show \eqref{eq:introendec}. If we allow for $f_0$ and $f_1$ to vary among all $C^k$ and $C^{k - 1}$ functions, respectively, which are compactly supported in $2 B$, then sufficient commutation with translation vector fields and a duality argument can lead to
\[
\Vert \partial \psi \Vert_{L^2 (B)} \le C \sup_{f_0, f_1, |\alpha| \le K(k)} \int_{2 B} (\partial_t \partial^\alpha \psi) f_1 + \sum_{i = 1}^n (\partial_i \partial^\alpha \psi) (\partial_i f_0) d x.
\]
Thus, in order to show \eqref{eq:introendec}, it suffices to show that the right hand side of \eqref{eq:introlocen} is bounded by ${C \over (1 + s)^{{n - 1} \over 2} (1 + |\tau|)^p}$ for all admissible $f_0$, $f_1$, and $\alpha$.

Let us now assume that $k$ is chosen sufficiently large so that we have pointwise decay rates for solutions to the homogeneous wave equation in that class. More precisely, we choose $k$ so large such that we have that
\[
|\partial f| \le {C \over (1 + s - t)^{{n - 1 \over 2}}} {1 \over 1 + |s - t - \sqrt{(x^1 - x_0^1)^2 + \dots + (x - x_0^n)^2}|^p}
\]
for $t \le s - t$. Such a rate can be proven in several different ways, and we recall that the sharp value of $p$ depends on the dimension. For example, the Klainerman-Sobolev Inequality implies this with $p = {1 \over 2}$ in arbitrary dimensions. As another example, in the case of $n = 2$, the sharp value of $p$ is ${3 \over 2}$, as can be seen using the fundamental solution.

After choosing $k$ such that this is true, we can now examine the right hand side of \eqref{eq:introlocen}. The integral over $\Sigma_0$ is controlled in terms of $f$, which we have information on, and the data for $\psi$. Because the data for $\psi$ is compactly supported in the unit ball, an examination of the geometry involved tells us that
\[
{1 \over (1 + s - t)^{{n - 1} \over 2} (1 + |s - t - \sqrt{(x - x_0^1)^2 + \dots + (x - x_0^n)^2}|)^p}
\]
is comparable to
\[
{1 \over (1 + s)^{{n - 1} \over 2} (1 + |\tau|)^p}
\]
in the support of the data for $\psi$ (see Figure~\ref{fig:lightconesaux} in Section~\ref{sec:decay}). Thus, we see that
\[
\int_{\Sigma_0} (\partial_t \partial^\alpha \psi) (\partial_t f) + \sum_{i = 1}^n (\partial_i \partial^\alpha \psi) (\partial_i f) d x \le {C \over (1 + s)^{{n - 1} \over 2} (1 + |\tau|)^p} \Vert \psi \Vert_{C^{|\alpha| + 1} (\Sigma_0)}.
\]
Thus, assuming that we can write
\begin{equation} \label{eq:errorreqboundintro}
    \begin{aligned}
    \int_0^s \int_{\Sigma_t} (\partial^\alpha F) (\partial_t f) d x d t \le {C \over (1 + s)^{{n - 1} \over 2} (1 + |\tau|)^p} \Vert \psi \Vert_{C^{|\alpha| + 1}, (\Sigma_0)}
    \end{aligned}
\end{equation}
we will have shown \eqref{eq:introendec}, which is the decay estimate we set out to prove.

These considerations are all carried out in detail in Section~\ref{sec:decay}. In practice, the most difficult thing to do is of course to control the error integrals involving $F$. This is done in a few simple examples in Section~\ref{sec:simpleapps}, and it is done for a class of anisotropic systems of wave equations in Section~\ref{sec:anisotropic}. We note that several aspects of this strategy can be modified as needed (see Section~\ref{sec:relateddirections} for a discussion of some of these potential modifications).

\subsection{Simple applications} \label{sec:simpleintro}
Before applying thee ideas to anisotropic systems of wave equations, we shall discuss two simple applications which we believe to be instructive. The first consists of the following problem. Let $\chi$ be a smooth cutoff function equal to $1$ for $|x| \le 1$ and equal to $0$ for $|x| \ge 2$. Then, working in $3 + 1$ dimensions, we take $\gamma$ to be the null curve $(t,t,0,0)$ in the $t$, $x$ plane, and we define $r_\gamma^2 (t,x,y,z) = (x - t)^2 + y^2 + z^2$. We then take the semilinear wave equation
\begin{equation} \label{eq:simpappintro1}
    \begin{aligned}
    \Box \phi = -\chi (r_\gamma) (\partial_t \phi)^2.
    \end{aligned}
\end{equation}
We have, thus, taken one of the examples studied by John in \cite{Joh81} (see Section~\ref{sec:history} for more of a discussion on this) and have multiplied the nonlinearity by a function which localizes it to a tubular neighborhood of the null geodesic $\gamma$. The asymptotic system (see \cite{Hor97}) for this equation blows up in finite time, and yet we shall show in Section~\ref{sec:simpleapps} that the trivial solution of this equation is globally nonlinearly stable. In addition to showing how the error integrals that arise when using this method can be controlled, this example shows the effect that angular localization can have when studying wave equations, a phenomenon which is present in anisotropic systems of wave equations (see Section~\ref{sec:anisotropicintro}). Indeed, we note that the cutoff function means that the nonlinearity only takes effect in a single angular direction from the origin. In fact, an analysis of the proof in Section~\ref{sec:simpleapps} shows that we still have global stability for the equation
\[
\Box \phi = -\chi \left (r_\gamma \left (x,{y \over (1 + t)^\omega},{z \over (1 + t)^\omega},t \right ) \right ) (\partial_t \phi)^2
\]
for any $\omega < {1 \over 2}$.

The next simple application will be proving global stability for a cubic nonlinear wave equation in $3 + 1$ dimensions. The proof will actually give global stability and decay rates for an equation of the form
\begin{equation} \label{eq:simpappintro2}
    \begin{aligned}
    \Box \phi = h (\partial_t \phi)^3,
    \end{aligned}
\end{equation}
where $h$ is any smooth function with bounded $C^k$ norm for $k$ sufficiently large. This example will show how an analysis of null geometry is necessary for controlling the error integrals in a simpler setting.

In both cases, the proofs use a bootstrap argument. The bootstrap argument involves decay estimates and energy estimates. The decay estimates are shown using the bilinear energy estimates as described in Section~\ref{sec:decayintro}. There is an energy level corresponding to $N$ derivatives in $L^2$, and we prove pointwise decay for only, say, the first ${3 N \over 4}$ derivatives in order to close. The energy estimates are interpolated with the pointwise decay estimates in order to close the pointwise estimates. Closing the pointwise estimates requires an analysis of the geometry of the interaction between the nonlinearity and the auxiliary multipliers $f$ described in Section~\ref{sec:decayintro}. This is carried out in Section~\ref{sec:simpleapps}.

\subsection{Applications to anisotropic systems of wave equations} \label{sec:anisotropicintro}
Anisotropic systems of wave equations arise as the simplest model problem before studying more complicated equations having characteristics with multiple sheets. Notable physical examples of such hyperbolic equations include the equations of crystal optics (see \cite{CouHil62}), compressible magnetohydrodynamics (see \cite{CouHil62}), and crystalline materials (see \cite{Chr98}). Characteristics with multiple sheets are related to the physical phenomenon of birefringence. If we wish to understand the solutions of these equations, it is beneficial to try to study the effects of having multiple characteristics in a setting which does not require us to understand the other novel phenomena they exhibit (such as \emph{conical refraction} in the case of a biaxial crystal in crystal optics). Thus, we study anisotropic wave equations in order to try to isolate the effects birefringence in these equations from the other phenomena they exhibit (see Section~\ref{sec:multiplecharacteristics} for more discussion on the role of anisotropic wave equations in the larger context of studying hyperbolic equations with multiple characteristics). In fact, anisotropic systems of wave equations arise naturally in the study of \emph{uniaxial crystals}, and also in the study of biaxial crystals under a certain symmetry reduction. We postpone a more thorough discussion of these issues and more generally of hyperbolic equations having multiply sheeted characteristics to Section~\ref{sec:multiplecharacteristics}.

For a nonlinear equation such as \eqref{eq:anisotropicintro}, one mathematical difficulty arises from the fact that the two wave equations only share the scaling symmetry. The proof of global stability will combine the ideas described in Section~\ref{sec:decayintro} with the vector field method. We shall only have access to the scaling vector field $S$. A rough version of the main Theorem is given below (see Theorem~\ref{thm:mainthm} for the precise statement):
\begin{theorem}[rough version of the main Theorem] The trivial solution to nonlinear wave equations of the form \eqref{eq:anisotropicintro} is asymptotically stable as long as $\lambda_1 \ne 1$ and $\lambda_2 \ne 1$.
\end{theorem}
The main Theorem will actually cover equations more general than \eqref{eq:anisotropicintro}, and will allow for higher order nonlinearities, quasilinear terms, and other cubic nonlinearities (see Section~\ref{sec:anisotropic}). The null condition can be summarized as saying that cubic interactions do not have the same wave for all three factors, and that $\lambda_1 \ne 1$ and $\lambda_2 \ne 1$. This restriction on $\lambda_1$ and $\lambda_2$ guarantees that the light cones intersect transversally, and it will allow us to control the nonlinear interaction.

We shall now provide only a very brief description of the proof (see Section~\ref{sec:anisotropicdescription} for a more thorough description). The proof will once again follow from using a bootstrap argument. Pointwise estimates and $L^2$ based energy estimates will be propagated. There are two key ideas that must be used.
\begin{enumerate}
    \item Controlling the error integrals will require a fine analysis of the geometry. This requires understanding the geometry of the intersections of various cones. The condition that $\lambda_1 \ne 1$ and $\lambda_2 \ne 1$ will be required to prove estimates on the geometry which allow us to effectively control the nonlinear interactions.
    \item The scaling vector field $S$ must be used. It will help in getting good powers of $\tau$ when controlling the error integrals (see \eqref{eq:errorreqboundintro}). The use of this operator is fundamental to making the proof work, and this follows in the philosophy introduced by Klainerman in \cite{Kl85}.
\end{enumerate}
The details of all of these considerations are carried out in Section~\ref{sec:anisotropic}.

\subsection{Related directions} \label{sec:relateddirections}
We now briefly discuss some potential directions for further study related to this paper.

For the anisotropic system of wave equations studied in Section~\ref{sec:anisotropic}, there are several questions and improvements one may wish to pursue.

\begin{enumerate}
    \item One natural improvement is to remove compact support and to instead have data decaying away from the origin. The estimates within the light cone from this paper should still be applicable in this case, but they would have to be supplemented by an analysis of the geometry of the region outside of the light cones. Specifically, it seems as though the main new ingredient would be bounds on the behavior of the geometry of the scaling vector field and how it interacts with the geometry of the various cones in question in this region. We believe that the analysis in Section~\ref{sec:SGeometry} could be useful in understanding the behavior of the scaling vector field in this region. We discuss this further in the remarks following the statement of Theorem~\ref{thm:mainthm}.
    \item Another natural direction is to investigate the problem in $3 + 1$ dimensions with a quadratic nonlinearity satisfying the analogous null condition. The analysis of the geometry becomes more complicated because of the increase in dimension. However, we believe that the techniques in this paper could be applicable, and we hope that this will be the topic of future work.
    \item It is also natural to consider the problem consisting of a system of more than two wave equations with more than two different light cones. When the analogous null condition is satisfied, we believe that the techniques in this paper would be applicable. In fact, having a product of 3 waves with different light cones seems to be even better behaved, particularly when the three light cones do not intersect in the same four lines.
    \item The fact that the null condition we consider suppresses most parallel interactions could lead to improved low regularity local well posedness results. This would be analogous to results concerning nonlinear wave equations satisfying the classical null condition (see, for example, \cite{KlaMac93} for the improved estimates that null forms satisfy in a low regularity setting, and for a description of how this can lead to improved local well posedness results).
\end{enumerate}

The strategy of controlling the equation using bilinear estimates also has further flexibility that has not been used in this paper. An example of this flexibility is that we may control integrals over different regions using a similar strategy. Performing a bilinear energy estimate between data on $\Sigma_0$ and, say, a null hypersurface is one instance of this. The data for the auxiliary multiplier can be posed on this null hypersurface in an appropriate way. We also believe that this could be useful for understanding solutions to wave equations in other contexts. Also, this kind of strategy could be used for other equations having some kind of bilinear energy structure.


Another possible use is that we can use the bilinear estimates in Section~\ref{sec:bilinearestimates} in a more refined way in order to provide a finer analysis of the phase space behavior of solutions to wave equations (in order to show, for example, improved decay of good derivatives). We can also further interface these ideas with existing techniques, such as more fully using weighted vector fields. We hope that exploiting this finer phase space analysis will be the topic of forthcoming work with Samuel Zbarsky.

\subsection{Hyperbolic equations with characteristics having multiple sheets} \label{sec:multiplecharacteristics}
The problem of studying the stability of systems of anisotropic wave equations was posed to the author by Sergiu Klainerman in the larger context of studying hyperbolic equations with multiple characteristics. In \cite{AndKla21}, Klainerman and the author discuss several interesting directions related to this program. In addition to the equations of crystal optics mentioned above, other physically relevant examples include compressible magnetohydrodynamics and the equations governing elasticity (see the work of Christodoulou \cite{Chr98} where he presents a general framework for these kinds of systems). In this introduction, we focus mostly on the role of anisotropic systems of wave equations in the context of studying crystal optics, but we encourage the reader to look at \cite{AndKla21} for a much more thorough discussion.

In terms of stability, we run into obstacles when applying the vector field method in the usual way to hyperbolic systems of equations with multiple characteristics because different components in the system can have different natural operators associated to them. However, there are still some operators present for anisotropic systems of wave equations, namely the translation vector fields and the scaling vector field. The strategy of using bilinear energy estimates can interface with commuting operators from the vector field method, and we in fact have to use the scaling vector field in a fundamental way in the proof of the main Theorem in Section~\ref{sec:anisotropic}. This more broadly follows the philosophy introduced by Klainerman in \cite{Kl85} which emphasizes the importance of being able to take advantage of good commutators (and in particular weighted commutators for decay). Moreover, as is discussed in \cite{AndKla21}, physically significant hyperbolic systems with multiple characteristics (such as crystal optics) still often exhibit a compatible scaling symmetry. Thus, the fact that anisotropic systems of wave equations still have access to the scaling vector field is representative of interesting physical situations.

As is also discussed in \cite{AndKla21}, it is natural to look for other operators that satisfy good commutation properties with equations exhibiting this birefringent behavior. This is an important part of the larger goal of robustly extending the vector field method to these kinds of equations. Similar to vector fields, we believe that these operators could be very valuable in this more exotic context. These operators require much study themselves, and are a key part of providing a fine tuned analysis of the linear equations of, say, crystal optics. These operators may no longer be vector fields, but may instead be, for example, second order operators. We note that second order operators were used by Andersson--Blue in \cite{AndBlu15} in the study of the wave equation on black hole spacetimes.

The system \eqref{eq:anisotropicintro} admits such an operator. We now describe Klainerman's procedure in this particular case. If we look at the symbols of $\Box$ and $\Box'$ by taking the spacetime Fourier transform, they are given by $\tau^2 - |\xi|^2$ and $\tau^2 - \lambda_1^{-2} \xi_1^2 - \lambda_2^{-2} \xi_2^2$. Level sets of these two functions are hyperboloids in frequency space. These hyperboloids are two dimensional in a three dimensional space, and they intersect transversally along families of curves. Because the symbol acts by multiplication, vector fields in frequency space which are tangent to level sets of the symbol annihilate the symbol, and thus commute with the equation because they annihilate the symbol. We can thus find an operator that commutes with both equations by looking at the vector field generated by the family of curves formed by intersecting these level sets. This vector field in frequency space can essentially be found by taking the cross product between the normal vector fields of the two families of hyperboloids. One can check that this results in the operator
\[
-\Lambda_1 \tau \xi_1 \partial_{\xi_2} - \Lambda_2 \tau \xi_2 \partial_{\xi_1} - \Lambda_3 \xi_1 \xi_2 \partial_\tau,
\]
subject to the constraints that $\Lambda_1 + \Lambda_2 - \Lambda_3 = 0$ and $\lambda_2^{-2} \Lambda_1 + \lambda_1^{-2} \Lambda_2 - \Lambda_3 = 0$ (of course, multiplying this operator by any constant still commutes with both equations as well). This vector field in frequency space can be seen to correspond to the second order differential operator
\[
\Lambda_1 y \partial_t \partial_x + \Lambda_2 x \partial_t \partial_y + \Lambda_3 t \partial_x \partial_y
\]
in physical space. We can then hope to use operators of this kind to further study solutions to anisotropic systems of wave equations, and we can look for similar kinds of operators in other systems. For example, the weights in these operators lead to weighted Sobolev inequalities, similar to how the weights in the Lorentz and scaling vector fields lead to the weighted Klainerman-Sobolev inequality (see \eqref{eq:VFs} and \eqref{eq:KlaSob}). We refer the reader to \cite{AndKla21} for a much more thorough discussion of this circle of ideas. For now, we just note that we hope to further explore uses of these operators together with Klainerman.

We end this Section with a more precise description of the relation between anisotropic systems of wave equations and crystal optics. These are a simpler counterpart to biaxial crystals. While they still exhibit birefringent behavior, they do not display other complicated phenomenon, such as conical refraction.

In the general case, it may be that the electrical and magnetic properties of the media themselves change as a function of the electric and magnetic fields. This corresponds to having the \emph{dielectric tensor} $\epsilon$ and \emph{magnetic permeability} $\mu$ can depend on the electric and magnetic fields. This leads to a nonlinear system of equations. For this discussion, we shall greatly simplify the situation by considering a constant, diagonal matrix $\epsilon$ and a constant $\mu$. More precisely, we take the dielectric tensor $\epsilon$ to be given by
\[
\epsilon = \begin{pmatrix} \epsilon_1 & 0 & 0 \\ 0 & \epsilon_2 & 0 \\ 0 & 0 & \epsilon_3 \end{pmatrix}.
\]
The three constants $\epsilon_i$ in this matrix are called the \emph{dielectric constants}. When they are all equal, we are left with the usual Maxwell equations governing the propagation of light in an isotropic media. When two of the constants are equal to each other but distinct from the third, the resulting equations then govern the propagation of light in a uniaxial crystal. Finally, when all three of the constants are distinct from each other, the resulting equations govern the propagation of light in a biaxial crystal. The magnetic permeability $\mu$ is taken to be a multiple of the identity, corresponding to a magnetically isotropic material. This can be thought to be $1$ for simplicity (see \cite{CouHil62}, \cite{Lie91}, \cite{Lie89}, and \cite{AndKla21} for a more thorough discussion).

Let us now write down the equations in the most general form. We shall denote by $E$ the electric field and $B$ the magnetic field. Thus, in $3 + 1$ dimensions, both the electric and magnetic fields have $3$ components. The Maxwell equations are then given by
\begin{equation} \label{eq:dielectricMaxwell}
    \begin{aligned}
    -\partial_t (\epsilon E) + \curl(B) = 0
    \\ -\partial_t (\mu B) - \curl(E) = 0
    \\ \divr(\epsilon E) = \divr(\mu B) = 0.
    \end{aligned}
\end{equation}
It is easy to see that the last set of equation can be thought of as constraints because they propagate in the evolutionary equations as long as they hold initially.

Commuting the evolutionary equations in \eqref{eq:dielectricMaxwell} with $\partial_t$, expanding them out in components, and using the equations to get rid of the magnetic field $B$ gives us the system of equations
\begin{equation}
    \begin{aligned}
    \epsilon_1 \partial_t^2 E^1 = \partial_y^2 E^1 + \partial_z^2 E^2 - \partial_x \partial_y E^2 - \partial_x \partial_z E^3
    \\ \epsilon_2 \partial_t^2 E^2 = \partial_z^2 E^2 + \partial_x^2 E^2 - \partial_y \partial_z E^3 - \partial_y \partial_x E^1
    \\ \epsilon_3 \partial_t^2 E^3 = \partial_x^2 E^3 + \partial_y^2 E^3 - \partial_z \partial_x E^1 - \partial_z \partial_y E^2
    \\ \epsilon_1 \partial_x E^1 + \epsilon_2 \partial_y E^2 + \epsilon_3 \partial_z E^3 = 0.
    \end{aligned}
\end{equation}
Let us now make the symmetry assumption that the electric field does not depend on $z$, meaning that all $\partial_z$ derivatives vanish. This results in the system of equations
\begin{equation}
    \begin{aligned}
    \epsilon_1 \partial_t^2 E^1 = \partial_y^2 E^1 - \partial_x \partial_y E^2
    \\ \epsilon_2 \partial_t^2 E^2 = \partial_x^2 E^2 - \partial_y \partial_x E^1
    \\ \epsilon_3 \partial_t^2 E^3 = \partial_x^2 E^3 + \partial_y^2 E^3
    \\ \epsilon_1 \partial_x E^1 + \epsilon_2 \partial_y E^2 = 0.
    \end{aligned}
\end{equation}
We can now use the divergence free condition, which we know propagates in the evolution as long as it holds initially, to write the equations as
\begin{equation}
    \begin{aligned}
    -\partial_t^2 E^1 + \epsilon_2^{-1} \partial_x^2 E^1 + \epsilon_1^{-1} \partial_y^2 E^1 = 0
    \\ -\partial_t^2 E^2 + \epsilon_2^{-1} \partial_x^2 E^2 + \epsilon_1^{-1} \partial_y^2 E^2 = 0
    \\ -\partial_t^2 E^3 + \epsilon_3^{-1} \partial_x^2 E^3 + \epsilon_3^{-1} \partial_y^2 E^3 = 0
    \\ \epsilon_1 \partial_x E^1 + \epsilon_2 \partial_y E^2 = 0.
    \end{aligned}
\end{equation}
The evolutionary equations listed here can be compared with \eqref{eq:anisotropicintro}, and we see that they are of a similar form. Indeed, we have a system of three equations. One equation (the one for $E^3$) has circular light cones while the other two equations (the ones for $E^1$ and $E^2$) have elliptical light cones. Thus, the equations we study in Section~\ref{sec:anisotropic} are related to this symmetry reduction for biaxial crystals. Moreover, we can easily see that the null condition that $\lambda_1 \ne 1$ and $\lambda_2 \ne 1$ is analogous to the condition that the $\epsilon_i$ all be distinct, meaning that the crystal is biaxial.\footnote{This is the first half of the null condition we need to prove global stability. It would be interesting to see if the second half, that the same wave not be cubed, corresponds to some physical property of the material.} When the crystal is uniaxial, we see that there are two possibilities. If $\epsilon_1 = \epsilon_2$, the equations become an isotropic system of wave equations with multiple speeds. These kinds of equations can be treated using the techniques introduced by Klainerman--Sideris in \cite{KlaSid96} (see also \cite{Sogge08}). If $\epsilon_1 = \epsilon_3$ or $\epsilon_2 = \epsilon_3$, the equations become anisotropic, but it is like the situation in which one of $\lambda_1$ or $\lambda_2$ is equal to $1$.

Although the results in Section~\ref{sec:anisotropic} thus seem to be relevant for the study of biaxial crystals under this symmetry reduction, we emphasize that the problem outside of symmetry is much more difficult and requires a much finer analysis. This symmetry reduction has made the equations for each component of the electric field decouple, which does not occur in the general case of a biaxial crystal. This decoupling removes many of the interesting difficulties present in biaxial crystals, such as conical refraction. The first step in analyzing these more complicated phenomena was undertaken by Liess in \cite{Lie91}, resulting in a global stability statement in \cite{Lie89} arising from using machinery similar to that introduced by Klainerman--Ponce in \cite{KlaPon83}. The linear analysis provided by Liess in \cite{Lie91} shows that the situation is much more difficult for these kinds of systems than for the usual wave equation. The stability result in \cite{Lie89} does not cover the class of nonlinearities considered in this paper, and in particular, it does not exploit any kind of null condition (it treats fourth order and higher nonlinearities). To our knowledge, this is the only global stability result for a biaxial crystal outside of symmetry assumptions. Providing sharper global stability statements requires a finer analysis. We believe this to be an extremely interesting direction for future work, and we refer the reader once again to \cite{AndKla21} for a more thorough discussion.

\subsection{Acknowledgements}
The author is extremely indebted to his advisor, Sergiu Klainerman, for making the author aware of this problem. The author is grateful to him both for his extremely valuable suggestions, especially concerning the importance of scaling, and for discussions with him about the role of this problem in a larger research program for studying hyperbolic equations with multiple characteristics. The author is also very thankful for several useful conversations with Mihalis Dafermos, Ross Granowski, Christoph Kehle, Jonathan Luk, Sung-Jin Oh, Federico Pasqualotto, Igor Rodnianski, Yakov Shlapentokh-Rothman, and Samuel Zbarsky.

\section{Decay estimates} \label{sec:decayestimates}
In this section, we shall prove the estimates that will lead to decay. Decay comes from bilinear integration by parts formulas which show that certain quantities are conserved at the linear level. The energy estimate is a special case of these estimates. These bilinear estimates are simple, but they are quite general, with versions being true on general Lorentzian manifolds. We record these geometric estimates in Section~\ref{sec:bilinearestimates}. In this paper, when we are interested in controlling $\psi$ where $\Box \psi = F$, the estimates are applied by taking an auxiliary solution of the homogeneous wave equation $\Box f = 0$ and using this as a multiplier for $\psi$.


To go from these bilinear estimates to pointwise decay statements, we use a duality argument. Controlling integral averages of a function $\psi$ against a sufficiently large class of test functions results in estimates on various norms for $\psi$. Thus, by allowing the data to vary appropriately for $f$, we prove that $\psi$ decays pointwise through a duality argument. The duality argument is described in Section~\ref{sec:dualityargument}, and the duality argument and bilinear estimates are combined to prove pointwise estimates in Section~\ref{sec:decay}. The decay rates this gives depend on decay rates for the homogeneous wave equation. Decay rates for solutions to the homogeneous wave equation can be taken as a black box, and can be proven in whatever way one prefers.

\subsection{Basic bilinear estimates} \label{sec:bilinearestimates}
We begin with the estimates on Minkowski space $\R^{n + 1}$. The first is a bilinear energy estimate and the second comes from integrating $\phi \Box \psi$ by parts. In this paper, we shall only use estimates between two time slices $\Sigma_{t_1}$ and $\Sigma_{t_2}$, but we note that we are of course free to do estimates in regions with different geometry (such as, for example, null hypersurfaces). This is expressed in the more geometric versions of these integration by parts formulas below which simply show that we can integrate appropriate divergence and use Stokes' Theorem.

\begin{lemma} \label{lem:bilinenM}
Let $\psi$ and $\phi$ be smooth functions decaying sufficiently rapidly at infinity with $\Box \psi = F$ and with $\Box \phi = G$. Then, we have that
\begin{equation}
\begin{aligned}
\int_{\Sigma_{t_2}} \partial_t \psi \partial_t \phi + \partial^i \psi \partial_i \phi d x + \int_{t_1}^{t_2} \int_{\Sigma_s} F \partial_t \phi d x d s
\\ = \int_{\Sigma_{t_1}} \partial_t \psi \partial_t \phi + \partial^i \psi \partial_i \phi d x - \int_{t_1}^{t_2} \int_{\Sigma_s} G \partial_t \psi d x d s.
\end{aligned}
\end{equation}
More generally, we have that
\begin{equation}
\begin{aligned}
\int_{\Sigma_{t_2}} \partial_t \partial^\alpha \psi \partial_t \phi + \partial^i \partial^\alpha \psi \partial_i \phi d x + \int_{t_1}^{t_2} \int_{\Sigma_s} \partial^\alpha F \partial_t \phi d x d s
\\ = \int_{\Sigma_{t_1}} \partial_t \partial^\alpha \psi \partial_t \phi + \partial^i \partial^\alpha \psi \partial_i \phi d x - \int_{t_1}^{t_2} \int_{\Sigma_s} G \partial^\alpha \partial_t \psi d x d s.
\end{aligned}
\end{equation}
\end{lemma}
\begin{proof}
We compute the following:
\begin{equation}
\begin{aligned}
\Box \psi \partial_t \phi = - \partial_t^2 \psi \partial_t \phi + \partial_i \partial^i \psi \partial_t \phi = - \partial_t (\partial_t \psi \partial_t \phi) + \partial_t \psi \partial_t^2 \phi + \partial_i (\partial^i \psi \partial_t \phi) - \partial^i \psi \partial_i \partial_t \phi
\\ = \partial_t (-\partial_t \psi \partial_t \phi) + \partial_t \psi \partial_t^2 \phi + \partial_i (\partial^i \psi \partial_t \phi) - \partial_t (\partial^i \psi \partial_i \phi) + \partial_t \partial^i \psi \partial_i \phi
\\ = \partial_t (-\partial_t \psi \partial_t \phi) + \partial_t \psi \partial_t^2 \phi + \partial_i (\partial^i \psi \partial_t \phi) - \partial_t (\partial^i \psi \partial_i \phi) + \partial^i (\partial_t \psi \partial_i \phi) - \partial_t \psi \partial^i \partial_i \phi
\\ = \partial_t (-\partial_t \psi \partial_t \phi) + \partial_t (-\partial^i \psi \partial_i \phi) + \partial_i (\partial^i \psi \partial_t \phi) + \partial^i (\partial_t \psi \partial_i \phi) - \partial_t \psi \Box \phi.
\end{aligned}
\end{equation}
Integrating between $\Sigma_{t_1}$ and $\Sigma_{t_2}$ gives us that
\begin{equation}
\begin{aligned}
\int_{t_1}^{t_2} \int_{\Sigma_s} \Box \psi \partial_t \phi d x d s
\\ = \int_{\Sigma_{t_2}} - \partial_t \psi \partial_t \phi - \partial^i \psi \partial_i \psi d x - \int_{\Sigma_{t_1}} - \partial_t \psi \partial_t \phi - \partial^i \psi \partial_i \psi d x - \int_{t_1}^{t_2} \int_{\Sigma_s} \partial_t \psi \Box \phi d x d s.
\end{aligned}
\end{equation}
Using the fact that $\Box \psi = F$ and that $\Box \phi = G$ gives us the first estimate. The second estimate comes from commuting the equation for $\psi$ with $\partial^\alpha$.
\end{proof}

\begin{lemma} \label{lem:bilinintM}
Let $\psi$ and $\phi$ be smooth functions decaying sufficiently rapidly at infinity with $\Box \psi = \Box \phi = 0$. Then, we have that
\begin{equation}
\begin{aligned}
\int_{\Sigma_{t_2}} \partial_t \psi \phi - \psi \partial_t \phi d x + \int_{t_1}^{t_2} F \phi d x d s
\\ = \int_{\Sigma_{t_1}} \partial_t \psi \phi - \psi \partial_t \phi d x + \int_{t_1}^{t_2} G \psi d x d s.
\end{aligned}
\end{equation}
More generally, we have that
\begin{equation}
\begin{aligned}
\int_{\Sigma_{t_2}} \partial_t \partial^\alpha \psi \phi - \partial^\alpha \psi \partial_t \phi d x + \int_{t_1}^{t_2} \partial^\alpha F \phi d x d s
\\ = \int_{\Sigma_{t_1}} \partial_t \partial^\alpha \psi \phi - \partial^\alpha \psi \partial_t \phi d x + \int_{t_1}^{t_2} G \partial^\alpha \psi d x d s.
\end{aligned}
\end{equation}
\end{lemma}
\begin{proof}
We compute the following:
\begin{equation}
\begin{aligned}
\Box \psi \phi = - \partial_t^2 \psi \phi + \partial_i \partial^i \psi \phi = - \partial_t (\partial_t \psi \phi) + \partial_t \psi \partial_t \phi + \partial_i (\partial^i \psi \phi) - \partial^i \psi \partial_i \phi
\\ = - \partial_t (\partial_t \psi \phi) + \partial_t (\psi \partial_t \phi) - \psi \partial_t^2 \phi + \partial_i (\partial^i \psi \phi) - \partial^i (\psi \partial_i \phi) + \psi \partial^i \partial_i \phi
\\ = - \partial_t (\partial_t \psi \phi) + \partial_t (\psi \partial_t \phi) + \psi \Box \phi + \partial_i (\partial^i \psi \phi) - \partial^i (\psi \partial_i \phi).
\end{aligned}
\end{equation}
Integrating between $\Sigma_{t_1}$ and $\Sigma_{t_2}$ gives us that
\begin{equation}
\begin{aligned}
\int_{t_1}^{t_2} \int_{\Sigma_s} \Box \psi \phi d x d s = \int_{\Sigma_{t_2}} - \partial_t \psi \phi + \psi \partial_t \phi d x - \int_{\Sigma_{t_1}} - \partial_t \psi \phi + \psi \partial_t \phi d x + \int_{t_1}^{t_2} \int_{\Sigma_s} \psi \Box \phi d x d s.
\end{aligned}
\end{equation}
Using the fact that $\Box \psi = F$ and $\Box \phi = G$ gives us the desired result.

\end{proof}

In order to more accurately show the geometric character of these identities, we provide versions of these formulas which are available on arbitrary globally hyperbolic Lorentzian manifolds $(M,g)$.

\begin{lemma} \label{lem:bilinint}
Let $D$ be a domain in $M$. Moreover, let $\Box_g \psi = F$ and let $\Box_g \phi = G$. We have that
\[
\int_D div(\nabla \psi \phi) - div(\nabla \phi \psi) d vol(g) = \int_D \phi F - \psi G d vol(g).
\]
\end{lemma}
\begin{proof}
We can write the wave equation $\Box_g \psi = F$ as $div(\nabla \psi) = tr (\nabla^2 \psi) = F$, where $\nabla$ is the covariant derivative operator associated with the Lorentzian metric $g$. Now, we have that
\begin{equation}
\begin{aligned}
div(\nabla \psi \phi) = \phi div (\nabla \psi) + g(d \psi,d \phi) = \phi \Box_g \psi + g(d \psi,d \phi).
\end{aligned}
\end{equation}

Similarly, we have that
\begin{equation}
\begin{aligned}
div(\nabla \phi \psi) = \psi div (\nabla \phi) + g(d \psi,d \phi) = \psi \Box_g \phi + g(d \psi,d \phi)
\end{aligned}
\end{equation}

Subtracting the second equation from the first gives us that
\begin{equation}
\begin{aligned}
\phi \Box_g \psi - \psi \Box_g \phi = div(\nabla \psi \phi) - div(\nabla \phi \psi).
\end{aligned}
\end{equation}
This gives us the desired result.
\end{proof}
Integrating this identity and using the divergence theorem can give us fluxes analogous to Lemma~\ref{lem:bilinintM} on a general Lorentzian manifold.

\begin{lemma}
Let $T [\psi,\phi]$ denote the bilinear energy momentum tensor given by
\begin{equation}
    \begin{aligned}
    T [\psi,\phi] = d \psi \otimes d \phi + d \phi \otimes d \psi - g (d \psi,d \phi) g.
    \end{aligned}
\end{equation}
Then, we have that
\[
div(T [\psi,\phi]) = (\Box \psi) d \phi + (\Box \phi) d \psi.
\]
\end{lemma}
\begin{proof}
We begin by noting that $T [\psi,\phi]$ is symmetric, so the divergence is well defined. This can be written as
\begin{equation}
    \begin{aligned}
    T_{\mu \nu} [\psi,\phi] = \partial_\mu \psi \partial_\nu \phi + \partial_\nu \psi \partial_\nu \phi - g_{\alpha \beta} \partial^\alpha \psi \partial^\beta \phi g_{\mu \nu}.
    \end{aligned}
\end{equation}
Computing the divergence means contracting with $\nabla^\mu$, where $\nabla$ is the connection associated with $g$. We have that
\begin{equation}
    \begin{aligned}
    \nabla^\mu T_{\mu \nu} [\psi,\phi] = (\nabla^\mu \partial_\mu \psi) \partial_\nu \phi + \partial_\mu \psi \nabla^\mu \partial_\nu \phi + \nabla^\mu \partial_\nu \psi \partial_\mu \phi + \partial_\nu \psi (\nabla^\mu \partial_\mu \phi)
    \\ - g_{\alpha \beta} \nabla^\mu \partial^\alpha \psi \partial^\beta \phi g_{\mu \nu} - g_{\alpha \beta} \partial^\alpha \psi \nabla^\mu \partial^\beta \phi g_{\mu \nu}.
    \end{aligned}
\end{equation}

Now, we note that $\partial_\mu \psi \nabla^\mu \partial_\nu \phi = g_{\alpha \beta} \partial^\alpha \psi \nabla^\mu \partial^\beta \phi g_{\mu \nu} = \nabla^2_{\mu \nu} \phi \partial^\mu \psi = (\nabla^2 \phi) (\nabla \psi,\cdot)$, where we have used the fact that the Hessian of a function is symmetric (we have also used the musical isomorphism induced by the metric between covariant and contravariant tensors). We have an analogous equality for the other terms with $\psi$ and $\phi$ interchanged. Putting everything together gives us that
\[
\nabla^\mu T_{\mu \nu} [\psi,\phi] = (\nabla^\mu \partial_\mu \psi) \partial_\nu \phi + (\nabla^\mu \partial_\mu \phi) \partial_\nu \psi.
\]
This is exactly the statement that
\[
div(T [\psi,\phi]) = (\Box \psi) d \phi + (\Box \phi) d \psi,
\]
as desired.
\end{proof}

Because the divergence of $T [\psi,\phi]$ can be written in terms of $\Box_g \psi$ and $\Box_g \phi$, it can be contracted with appropriate vector fields to find useful identities after integrating over a spacetime domain and applying Stokes' Theorem, just as is the case with the usual energy momentum tensor ${1 \over 2} T [\psi,\psi]$. The resulting identities will then also depend on the deformation tensors of the vector fields. For example, the identities in Lemma~\ref{lem:bilinenM} which are specific to Minkowski space correspond to contracting this bilinear energy momentum tensor with the vector field $\partial_t$, which is a Killing field (i.e., it has a vanishing deformation tensor). We refer the reader to, for example, \cite{Ali10} for more thorough description of integrating the contraction of energy momentum tensors with vector fields in order to derive useful quantities.

\subsection{Duality argument} \label{sec:dualityargument}
The estimates from Section~\ref{sec:bilinearestimates} lead us to think that we should be able to control the averages of $\psi$ and its derivatives weighted by functions $f$ as long as we can control the error integrals involving $F$ in \eqref{eq:testest}. Moreover, because we can commute with unit derivatives in practice, the same is true of $\partial^\alpha \psi$. If $f$ is allowed to vary in a sufficiently large class and if $|\alpha|$ is allowed to be sufficiently large, it is well known that control over these averages will imply bounds for $\psi$ and some of its derivatives. Because we require pointwise decay, we have chosen to immediately prove pointwise estimates using these averages. Of course, these averages can be used to prove estimates for other $L^p$ spaces, including the energy space. Thus, we could modify the following argument to first prove an energy bound, and then prove pointwise decay using the usual Sobolev embedding.


We shall now provide a proof of this duality argument for completeness. Although this result is far from sharp, it suffices for our applications. We assume that we are on some $\Sigma_t$, and that we have that $\left |\int_{\Sigma_t} \partial^\alpha \psi f \right | d x \le M$ where $M$ is some explicit constant. Moreover, we assume that $\psi$ is some smooth function, and that $f$ is allowed to be any function with $C^k$ norm bounded by another constant $D$. Our goal is to show that we, in fact, have control over $\Vert \psi \Vert_{L^\infty}$ explicitly in terms of the constants $M$ and $D$ as long as we can take all $\omega$ with $|\omega|$ sufficiently large depending on $k$.

\begin{proposition}
Let $\psi$ be some smooth function on $\R^n$ such that, for any smooth function $f$ supported in any ball of radius $1$ having $C^k$ norm at most $D$, we have that
\begin{equation}
\begin{aligned}
\left |\int \partial^\alpha \psi(x) f(x) d x \right | \le M
\end{aligned}
\end{equation}
for all $|\alpha| \le k + n + 1$. Then, there exists some universal constant $C$ depending only on $k$ and $n$ such that
\begin{equation}
\begin{aligned}
\Vert \psi \Vert_{L^\infty} \le {C M \over D}.
\end{aligned}
\end{equation}
\end{proposition}
\begin{proof}
It suffices to get control over $\psi$ near the origin, as control over any other location follows in the same way. We fix a smooth, radially symmetric cutoff function $\chi$ equal to $1$ in the ball of radius ${1 \over 2}$, decaying monotonically to $0$. The function is identically equal to $0$ outside of the ball of radius $1$. Let $\Vert \chi \Vert_{C^{2 k + n + 1}} = E$. Then, we have that ${1 \over E} \chi$ has $C^{2 k + n + 1}$ norm controlled by $1$.

Now, we take the cube centered at the origin with side lengths equal to $10$. We note that this cube contains the support of $\chi$. By identifying the faces of this cube in the usual way, we can think of the function $\chi \psi$ as living on the flat torus. We take the functions $e^{{1 \over 5} \pi l \cdot x}$ for $l \in \Z^n$. Integrating against these functions give us the Fourier coefficients of a given function.

The bounds on integrals against test functions are not directly useful when integrating against the functions $e^{{1 \over 5} \pi l \cdot x}$ because they have growing $C^k$ norm as $\Vert l \Vert_{\ell^\infty}$ increases. Thus, we divide by $\Vert l \Vert_{\ell^\infty}^k$, giving us the functions ${1 \over \Vert l \Vert_{\ell^\infty}^k} e^{{1 \over 5} \pi l \cdot x}$. These functions have $C^k$ norm controlled by some constant $A$ independent of $l$. Thus, multiplying by $D$ and dividing by $A$, we get that the functions $f_l (x) := {D \over A} {1 \over \Vert l \Vert_{\ell^\infty}^k} e^{{1 \over 5} \pi l \cdot x}$ have $C^k$ norm controlled by $D$. We now consider the Fourier coefficients of the function $\chi \psi$. We shall use the regularity of this function to show that its Fourier coefficients decay by integrating against the functions $f_l$.

By the assumptions given above, we have that
\begin{equation}
\begin{aligned}
{D \over E A \Vert l \Vert_{\ell^\infty}^k} |c_l| = \left |\int {1 \over E} \chi(x) \psi(x) f_l (x) d x \right | \le M C_1,
\end{aligned}
\end{equation}
where $c_l$ is the Fourier coefficient of $\chi \psi$ associated to $l$, and where $C_1$ depends on $k$ (this constant arises from controlling the $C^k$ norm of ${1 \over E} \chi f_l$ in terms of the $C^k$ norm of $\chi$ and the $C^k$ norm of $f_l$).

Now, we have that
\begin{equation}
\begin{aligned}
\left ({\pi \over 5} \right )^{|\alpha|} {\Vert l \Vert_{\ell^\infty}^{|\alpha|} D \over E A \Vert l \Vert_{\ell^\infty}^k} |c_l| = \left |\int {1 \over E} \partial^\alpha (\chi(x) \psi(x)) f_l (x) d x \right | \le M C_2,
\end{aligned}
\end{equation}
where $C_2$ depends on $|\alpha|$ and $k$, and where the partial derivatives in $\partial^\alpha$ are all taken in the direction corresponding to the largest component of $l$. Taking $|\alpha| = k + n + 1$ gives us that

\begin{equation}
\begin{aligned}
\left ({\pi \over 5} \right )^{|\alpha|} \Vert l \Vert_{\ell^\infty}^{n + 1} {D \over E A} |c_l| = \left |\int {1 \over E} \partial^\alpha (\chi(x) \psi(x)) f_l (x) d x \right | \le M C''.
\end{aligned}
\end{equation}

This gives us that
\begin{equation}
\begin{aligned}
|c_l| \le C {M E A \over D \Vert l \Vert_{\ell^\infty}^{n + 1}},
\end{aligned}
\end{equation}
where the constant $C$ depends on the $C^k$ norms of the $f_l$ (we recall that this is $A$, which is independent of $l$), and the numbers $k$ and $|\alpha| = k + n + 1$. Now, we have that
\begin{equation}
\begin{aligned}
\chi(x) \psi(x) = \sum_{l \in \Z^n} c_l e^{{1 \over 5} \pi l \cdot x}.
\end{aligned}
\end{equation}
Plugging in the above estimates gives us the quantitative estimate
\begin{equation}
\begin{aligned}
|\chi(x) \psi(x)| \le \sum_{l \in \Z^n} |c_l| \le {C M \over D},
\end{aligned}
\end{equation}
where $C$ is some absolute constant depending only on $n$, $k$, and the functions $\chi$ and $f_l$. This is the desired result.
\end{proof}

\subsection{Decay estimates} \label{sec:decay}
We now proceed to one specific application of the above considerations. We shall use the bilinear estimates in Minkowski space from Section~\ref{sec:bilinearestimates} and the duality argument from Section~\ref{sec:dualityargument} to prove pointwise decay for solutions to wave equations. We shall include inhomogeneities in the estimates. In our applications, these inhomogeneities are nonlinearities.

Let $\Box \psi = F$, and suppose that $\psi$ has compactly supported data in the unit ball. Suppose that we want to show pointwise decay for $\psi$. We then take auxiliary solutions of the homogeneous wave equation $\Box f = 0$, and decay will be a result of applying the bilinear estimates from Section~\ref{sec:bilinearestimates} to $\psi$ and $f$, allowing the data for $f$ to vary in an appropriate class, and then applying the duality argument from Section~\ref{sec:dualityargument}.

The data for $f$ is chosen as follows. Suppose we want to show decay for $\psi$ at some point $(s,x_0^i)$. We then consider the ball $B$ of radius $1$ centered at $x_0^i$ in $\Sigma_s$, and we allow the data $f_0 (x) = f(s,x)$ and $f_1 (x) = -\partial_t f(s,x)$ for $f$ to vary among all smooth functions whose $C^k$ norm (where $k$ will be specified shortly) is of size at most $10$ and which are supported in $B \subset \Sigma_s$. More specifically, we have chosen the nonlinearities to only involve $\partial_t$ derivatives for simplicity, so we shall always set $f_0 = 0$. If other derivatives are involved instead, this argument can be modified to treat those cases.

The estimate from Lemma~\ref{lem:bilinenM} says that

\[
\int_{\Sigma_s} \partial_t \psi f_1 d x = \int_{\Sigma_0} \partial_t \psi \partial_t f + \partial^i \psi \partial_i f d x - \int_0^s \int_{\Sigma_t} F \partial_t f d x d t.
\]

Thus, we have that
\[
\left |\int_{\Sigma_s} \partial_t \psi f_1 d x  \right | \le \left |\int_{\Sigma_0} \partial_t \psi \partial_t f + \partial^i \psi \partial_i f d x \right | + \left |\int_0^s \int_{\Sigma_t} F \partial_t f d x d t \right |
\]

We denote by $\tau$ the $u$ coordinate of $(s,x_0^i)$ (i.e., $\tau = s - \sqrt{\sum_{i = 1}^n (x_0^i)^2}$). In practice, we know that solutions to the homogeneous wave equation with compactly supported and sufficiently regular data (like $f$) decay in a certain way. Let us thus assume that $k$ is chosen so large such that
\[
|\partial f| \le {C \over (1 + s - t)^{{n - 1 \over 2}}} {1 \over 1 + |s - t - \sqrt{(x^1 - x_0^1)^2 + \dots + (x - x_0^n)^2}|^p}
\]
for $0 \le t \le s$. For example, when $n = 2$, we can take $k = 2$ and $p = {3 \over 2}$ using the fundamental solution (this is of course not sharp).

As a consequence of this, we have that
\begin{equation} \label{eq:auxdec}
    \begin{aligned}
    \left |\int_{\Sigma_0} \partial_t \psi \partial_t f + \partial^i \psi \partial_i f d x \right | \le {C \over (1 + s)^{{n - 1 \over 2}} (1 + |\tau|^p)} \Vert \partial \psi \Vert_{L^1 (\Sigma_0)},
    \end{aligned}
\end{equation}
where $p > 0$ measures the decay away from the light cone. Here, we are using the fact that the data for $\psi$ is compactly supported in the unit ball in $\Sigma_0$. Indeed, the symmetry of the configuration guarantees that the distance from the center of the unit ball in $\Sigma_0$ to the light cone associated with $f$ is comparable to $\tau$, and clearly the height of this cone is given by $s$. This can be seen in Figure~\ref{fig:lightconesaux}.

\begin{figure}
    \centering
    \begin{tikzpicture}
    \draw[dashed] (0,0) -- (4,4);
    \draw[ultra thick] (4,4) -- (5,5);
    \draw[dashed] (0,0) -- (-1,1);
    \draw[ultra thick] (-1,1) -- (-5,5);
    \draw[ultra thick] (3,5) -- (2.7,4.7);
    \draw[dashed] (2.7,4.7) -- (-1,1);
    \draw[ultra thick] (-1,1) -- (-2,0);
    \draw[ultra thick] (3,5) -- (3.27,4.73);
    \draw[ultra thick] (4,4) -- (8,0);
    \draw[dashed] (3.27,4.73) -- (4,4);
    \draw[dotted] (3,5) -- (5,5);
    \node (tau) at (4,5.2) {$\tau$};
    \draw[dotted] (3,0) -- (3,5);
    \node (s) at (3.2,2) {$s$};
    \draw[dotted] (0,0) -- (3,0);
    \node (a) at (1.5,0.2) {$s - \tau$};
    \draw[thick,domain=-180:-85] plot ({1.75*cos(\x)},{1.75+0.3*sin(\x)});
    \draw[dashed,domain=-85:0] plot ({1.75*cos(\x)},{1.75+0.3*sin(\x)});
    \draw[dashed,domain=0:180] plot ({1.75*cos(\x)},{1.75+0.3*sin(\x)});
    \draw[dashed,domain=-180:-105] plot ({3+1.75*cos(\x)},{3.25+0.3*sin(\x)});
    \draw[thick,domain=-105:0] plot ({3+1.75*cos(\x)},{3.25+0.3*sin(\x))});
    \draw[dashed,domain=0:180] plot ({3+1.75*cos(\x)},{3.25+0.3*sin(\x))});
    \draw[thick] (0,5) ellipse (5 and 0.35);
    \draw[thick] (-2,0) arc(-180:0:5 and 0.35);
    \draw[dashed] (-2,0) arc(-180:-360:5 and 0.35);
    
    \draw[color=red,rotate around={32:(1.5,2.5)}] (-1.36,2.5) arc(-180:0:2.88 and 0.2);
    \draw[dashed,color=red,rotate around={32:(1.5,2.5)}] (-1.36,2.5) arc(-180:-360:2.88 and 0.2);
    \end{tikzpicture}
    \caption{A configuration of light cones for $\psi$ and $f$. The light cone for $\psi$ is the upward opening cone, and the light cone for $f$ is the downward opening cone. The vertical distance between the tips of the cones is $s$, and the horizontal distance from the tip of the downward opening cone and the upward opening cone is $\tau$. The intersection of the two cones is a tilted spacelike ellipse whose major axis is comparable to $s$ and whose minor axis is comparable to $\sqrt{\tau} \sqrt{s}$ (this follows from the results in Sections~\ref{sec:geometry} and \ref{sec:SGeometry}). This ellipse appears in red in the above figure. We also refer the reader to Figure~\ref{fig:psifSigma_t} in Section~\ref{sec:geometry} to see what this looks like in each $\Sigma_t$.}
    \label{fig:lightconesaux}
\end{figure}
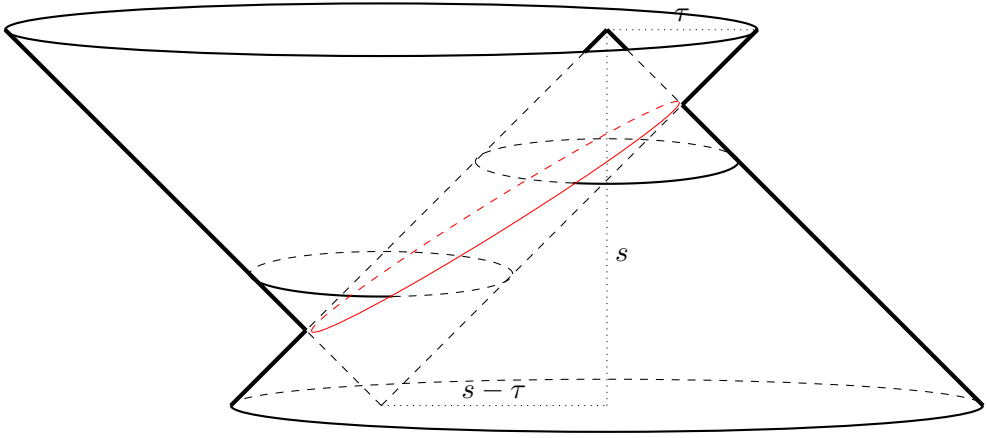

Thus, we have the estimate
\begin{equation} \label{eq:decayest}
    \begin{aligned}
    \left |\int_{\Sigma_s} \partial_t \psi f_1 d x  \right | \le {C \over (1 + s)^{{n - 1 \over 2}} (1 + |\tau|^p)} \Vert \partial \psi \Vert_{L^1 (\Sigma_0)} + \left |\int_0^s \int_{\Sigma_t} F \partial_t f d x d t \right |
    \end{aligned}
\end{equation}
This is the form of the estimate we shall use. When combined with commutation with translation vector fields and the duality argument from Section~\ref{sec:dualityargument}, we get pointwise decay for $\partial_t \psi$ assuming that
\begin{equation} \label{eq:reqerrorbounds}
    \begin{aligned}
    \left |\int_0^s \int_{\Sigma_t} F \partial_t f d x d t \right | \le {C \over (1 + s)^{{n - 1} \over 2} (1 + |\tau|^p)}.
    \end{aligned}
\end{equation}
In this paper, we note that $\psi$ is of size $\epsilon$ and $F$ is schematically of size $\epsilon^q$ for some $q > 1$ in practice. We will be able to recover these estimates in a bootstrap argument.

The decay estimate resulting from \eqref{eq:decayest} and the duality argument in Section~\ref{sec:dualityargument} is recorded in the following proposition.

\begin{proposition} \label{prop:decay}
Let
\[M [\psi] (s,x_0^i) = \sup_{|\alpha| \le k + n + 1} \sup_{f_1} \left ({C \over (1 + s)^{{n - 1 \over 2}} (1 + |\tau|^p)} \Vert \partial \partial^\alpha \psi \Vert_{L^1 (\Sigma_0)} + \left |\int_0^s \int_{\Sigma_t} (\partial^\alpha F) \partial_t f d x d t \right | \right ),
\]
where $(s,x_0^i)$ denotes the center of the unit ball in $\Sigma_s$ where $f_1$ is supported, where $\tau = s - \sqrt{\sum_{i = 1}^n (x_0^i)^2}$, and where the $C^k$ norm of $f_1$ is bounded by $10$. Then, we have that
\[
|\partial_t \psi| (s,x_0^i) \le C M [\psi] (s,x_0^i).
\]
\end{proposition}
The value of $p$ will be specified in each application.

We end this Section by remarking that while Proposition~\ref{prop:decay} is based on using duality and the bilinear energy estimate from Lemma~\ref{lem:bilinenM}, there is a similar estimate coming from Lemma~\ref{lem:bilinint}. Proposition~\ref{prop:decay} is used in Section~\ref{sec:anisotropic} to study anisotropic wave equations. In Section~\ref{sec:simpleapps}, we shall actually use a version of Proposition~\ref{prop:decay} adapted to Lemma~\ref{lem:bilinint} instead. The statement and proof are analogous.

\section{Simple applications} \label{sec:simpleapps}
In this Section, we shall use the strategy described in Section~\ref{sec:decayestimates} in order to prove global stability of the trivial solution $\psi = 0$ in a few specific cases. The problems are artificial, but they show how to use Section~\ref{sec:decayestimates} in simplified settings, and they help us examine some of the behavior of solutions to nonlinear wave equations. Because the main purpose of these applications is to give an idea of how to handle anisotropic systems of wave equations, less detail will be given than is found in Section~\ref{sec:anisotropic}. The first result of this Section, Proposition~\ref{prop:localizednonlinearity}, deals with a localized nonlinearity. The second result, Proposition~\ref{prop:cubic}, deals with a general cubic nonlinearity in $3 + 1$ dimensions. Both of these problems were described in Section~\ref{sec:simpleintro}.

Both of the following results are in $3 + 1$ dimensions. We shall use Proposition~\ref{prop:decay} for the application in Section~\ref{sec:localizednonlinearity}, and we shall use the version of Propisition~\ref{sec:decay} adapted to Lemma~\ref{lem:bilinint} in Section~\ref{sec:cubic} (this is described after Proposition~\ref{sec:decay}). In order to do so, we need to assume pointwise decay of solutions to the homogeneous wave equation for a suitable class of initial data. We shall use the rates given by the fundamental solution, meaning that we shall assume that solutions arising from, say, $C^2$ initial data supported on the ball of radius $R$ are supported where $|u| \le R$ and decay like ${1 \over (1 + t)}$. This is like taking $n = 3$ and, schematically, $p = \infty$ in \eqref{eq:auxdec}. We do not require decay this strong in $u$ for solutions to the homogeneous equation, but we shall assume it for simplicity.

Before turning to these applications, we note that these examples require similar analysis to that of anisotropic systems of wave equations, which are considered in Section~\ref{sec:anisotropic}. Both examples require an analysis of the geometry involved, which is the most technically involved part of Section~\ref{sec:anisotropic}. More precisely, the example in Section~\ref{sec:localizednonlinearity} exhibits how angular localization can help in control nonlinear effects, and the example in Section~\ref{sec:cubic} exhibits how it can be important to understand the geometry of intersecting cones. There are analogous considerations that go into analyzing anisotropic systems of wave equations in Section~\ref{sec:anisotropic}.

\subsection{Global stability with a localized nonlinearity} \label{sec:localizednonlinearity}

We recall the setting of \eqref{eq:simpappintro1}. The function $r_\gamma^2 (t,x,y,z) = (x - t)^2 + y^2 + z^2$ denotes a radial distance from the null geodesic in the $t$, $x$ plane emanating from the spacetime origin. Then, for $\chi$ a smooth cutoff function equal to $1$ for $|x| \le 1$ and equal to $0$ for $|x| \ge 2$, we are interested in solving \eqref{eq:simpappintro1}, which we recall is given by
\begin{equation} \label{eq:beameq}
    \begin{aligned}
    \Box \psi = -\chi(r_\gamma) (\partial_t \psi)^2
    \end{aligned}
\end{equation}
for small data compactly supported in the unit ball. We have the following Proposition.

\begin{proposition} \label{prop:localizednonlinearity}
There exists $\epsilon_0 > 0$ and $N \in \N$ such that, for all $\epsilon \le \epsilon_0$, the trivial solution to \eqref{eq:beameq} is globally nonlinearly stable with respect to perturbations $\psi(0,x) = \psi_0 (x)$ and $\partial_t \psi(0,x) = \psi_1 (x)$ which are compactly supported in the unit ball and have that $\Vert \partial \psi_0 \Vert_{H^N} + \Vert \psi_1 \Vert_{H^N} \le \epsilon$. Moreover, the solutions to this equation satisfy the decay estimate
\[
|\partial_t \partial^i \psi| \le {C \epsilon \over 1 + t}
\]
for all $|i| \le {3 N \over 4}$.
\end{proposition}

\begin{proof}
We shall take as bootstrap assumptions an energy bound on $\psi$ as well as a pointwise bound. We let $T$ be the maximal real number such that
\begin{enumerate}
    \item We have that
    \begin{equation} \label{eq:bootstrapenergylocalized}
        \begin{aligned}
        \sup_{0 \le t \le T} \Vert \partial \psi \Vert_{H^N (\Sigma_t)} \le C \epsilon^{{3 \over 4}} (1 + t)^{\sqrt{\epsilon}}.
        \end{aligned}
    \end{equation}
    \item We have that
    \begin{equation} \label{eq:bootstrappointwiselocalized}
        \begin{aligned}
        |\partial_t \partial^i \psi| (t,x) \le {C \epsilon^{{3 \over 4}} \over 1 + t}
        \end{aligned}
    \end{equation}
    for all $0 \le t \le T$ and $|i| \le {3 N \over 4}$.
\end{enumerate}

We shall now recover both of these bootstrap assumptions. The energy will be recovered using an energy estimate, Gronwall's inequality, and the pointwise bootstrap assumptions, while the pointwise bootstrap assumptions will be recovered using Proposition~\ref{prop:decay}.

We first recover the bootstrap assumptions for the energy. Let $0 \le s \le T$. We commute the equation \eqref{eq:beameq} with $\partial^i$ for $|i| \le N$. Then, a $\partial_t$ energy estimate gives us that
\begin{equation} \label{eq:enestbeam}
    \begin{aligned}
    \Vert \partial \partial^i \psi \Vert_{L^2 (\Sigma_s)}^2 = \Vert \partial \partial^i \psi \Vert_{L^2 (\Sigma_0)}^2 + 2 \int_0^s \int_{\Sigma_t} -\partial^i (\chi (r_\gamma) (\partial_t \psi)^2) (\partial_t \partial^i \psi) d x d t.
    \end{aligned}
\end{equation}
We note that at least one factor of $\partial_t \psi$ will have fewer than ${3 N \over 4}$ derivatives on it after applying the product rule to $\partial^\alpha$. Thus, using the pointwise bootstrap assumptions in \eqref{eq:enestbeam} gives us that
\[
E(s) \le E(0) + C \epsilon^{{3 \over 4}} \int_0^s {1 \over 1 + t} E(t) d t,
\]
where
\[
E(t) = \sum_{j \le i} \int_{\Sigma_t} (\partial_t \partial^j \psi)^2 + (\partial_x \partial^j \psi)^2 + (\partial_y \partial^j \psi)^2 + (\partial_z \partial^j \psi)^2 d x.
\]
An application of Gronwall's inequality then gives us that $E(s) \le 2 \epsilon^2 (1 + t)^{\sqrt{\epsilon}}$ for $\epsilon$ sufficiently small, as desired.

We shall now recover the pointwise bootstrap assumptions. By Proposition~\ref{prop:decay}, it suffices to show that
\[
M [\partial^i \psi] (s,x) \le {C \epsilon \over 1 + s}
\]
for all $0 \le s \le T$ and $|i| \le {3 N \over 4}$, where $M$ is as in Propsition~\ref{prop:decay}. Now, because we are assuming \eqref{eq:bootstrappointwiselocalized} for the pointwise decay of the auxiliary functions $f$, we see that it suffices to control
\[
\int_0^s \int_{\Sigma_t} -\partial^j \partial_i (\chi (r_\gamma) (\partial_t \psi)^2) (\partial_t f) d x d t
\]
because the terms involving initial data in $M$ in Proposition~\ref{prop:decay} are controlled by ${C \epsilon \over 1 + t}$.

Now, after interpolating between the pointwise bootstrap assumptions and the energy bootstrap assumptions, we have that
\[
|\partial^j \partial^i (\chi (r_\gamma) (\partial_t \psi)^2)| (t,x) \le {C \epsilon^{3 \over 2} \over (1 + t)^{2 - \delta}}
\]
for any $\delta > 0$ as long as $N$ is taken sufficiently large as a function of $\delta$. Thus, taking $\tilde{\chi} (x)$ to be a smooth, positive function equal to $1$ for $|x| \le 5$ and equal to $0$ for $|x| \ge 10$, we have that
\begin{equation} \label{eq:errorintlocalized}
    \begin{aligned}
    \left |\int_0^s \int_{\Sigma_t} -\partial^j \partial_i (\chi (r_\gamma) (\partial_t \psi)^2) (\partial_t f) d x d t \right | \le C \epsilon^{{3 \over 2}} \int_0^s \int_{\Sigma_t} \tilde{\chi} (r_\gamma) {1 \over (1 + t)^{2 - \delta}} {1 \over 1 + s - t} d t
    \\ \le C \epsilon^{{3 \over 2}} \Vert \tilde{\chi} \Vert_{L^1} \int_0^s {1 \over (1 + t)^{2 - \delta}} {1 \over 1 + s - t} d t \le C \epsilon^{{3 \over 2}} {1 \over 1 + s},
    \end{aligned}
\end{equation}
as desired.

\end{proof}

Examining this proof, we see that the localization given by $\chi$ is fundamental. It is interesting to note that we can allow the support of $\chi$ to expand in $t$ in the angular directions as long as the rate is smaller than $t^{{1 \over 2}}$, which is the wave packet scaling. An examination of these terms in the absence of the cutoff shows that the wave packet scaling also appears naturally when analyzing the example
\[
\Box \phi = -(\partial_t \phi)^2
\]
studied by John in \cite{Joh81} which leads to blow up. This gain in localization is similar to what will happen for anisotropic systems of wave equations in Section~\ref{sec:anisotropic}. We recall that this corresponds to an equation of the form
\[
\Box \phi = -\chi \left (r_\gamma \left (x,{y \over (1 + t)^\omega},{z \over (1 + t)^\omega},t \right ) \right ) (\partial_t \phi)^2,
\]
where $\omega$ denotes the rate at which the support of $\chi$ expands.

If we allow the support of $\chi$ to expand at a rate of $t^{{1 \over 2}}$ in the angular directions (i.e., if we take $\omega = {1 \over 2}$), it still appears as though there is a substantial gain in the volume of interaction. Indeed, the total volume of worst interaction (i.e., the volume of intersection of the sphere with the cutoff) will be $t$ instead of $t^2$, which is what it usually is. However, wave packets that propagate along the null geodesic at the center of the cutoff will still experience amplification from interactions on a set of maximal measure. This can be seen by examining the integrals in \eqref{eq:errorintlocalized} that come from picking auxiliary multipliers whose data are posed along this null geodesic. Thus, because these wave packets behave as if the nonlinearity was just $-(\partial_t \phi)^2$, the example that John studied, we cannot close the argument, and we believe that this may in fact still lead to blow up.

We note that we have allowed the support of $\chi$ to expand only in the angular directions in the sense that the $y$ and $z$ directions are roughly angular in the region of the support of $\chi$. If we allow the support to expand in the radial direction as well (this corresponds roughly to $x$), the analysis must be slightly different, as the solution should decay faster away from the light cone.

\subsection{Global stability with a general cubic nonlinearity} \label{sec:cubic}

We now turn to the second of the simple applications. We have the following proposition.


\begin{proposition} \label{prop:cubic}
Let $\psi$ satisfy the nonlinear wave equation
\begin{equation}
\begin{aligned}
\Box \psi = (\partial_t \psi)^3
\end{aligned}
\end{equation}
in $\R^{3 + 1}$ with initial data given by $\psi(0,x) = \psi_0 (x)$ and $\partial_t \psi(0,x) = \psi_1 (x)$. We assume that $\psi_0 (x)$ and $\psi_1 (x)$ are smooth functions supported in the unit ball in $\Sigma_0$ with $H^{N + 1}$ norm of size $\epsilon$, where the minimum size of $N$ is determined by the proof. As long as $\epsilon$ is sufficiently small, the above perturbations converge back to the trivial solution pointwise at a rate of ${1 \over 1 + t}$.
\end{proposition}
We note that the same proof works for a nonlniearity of $h (\partial_t \psi)^3$ where $h$ is an arbitrary function whose $C^N$ norm is bounded.
\begin{proof}
We shall take as bootstrap assumptions energy and pointwise bounds on $\psi$. Let $\delta > 0$ be some sufficiently small real number. We let $T$ be the maximal real number such that, for all $0 \le t \le T$,
\begin{enumerate}
\item we have that
\begin{equation} \label{eq:cubicenbootstrap}
    \begin{aligned}
    |\partial^\alpha \psi| (t,x) \le {C \epsilon^{{3 \over 4}} \over (1 + t) (1 + |u|^{1 - \delta})}
    \end{aligned}
\end{equation}
for $|\alpha| \le {N \over 2} + 1$ and where $\partial^\alpha$ contains at most a single $t$ derivative
\item we have that
\begin{equation} \label{eq:cubicpointwisebootstrap}
    \begin{aligned}
    \Vert \partial^{\alpha} \partial \psi \Vert_{L^2 (\Sigma_t)} \le C \epsilon^{{3 \over 4}}
    \end{aligned}
\end{equation}
for all $|\alpha| \le N$.
\end{enumerate}
We shall recover that, for all $t$ for which the above estimates hold true, we in fact have the stronger estimates coming from replacing $\epsilon^{{3 \over 4}}$ in the bootstrap assumptions with $C \epsilon$. This will close the bootstrap argument, which will complete the proof of Proposition~\ref{prop:cubic}.

In order to effectively analyze this problem, it will be useful to introduce some notation. We shall denote by $(t,r,\theta)$ usual polar coordinates on $\R^{3 + 1}$ where $t$ is the usual time coordinate, $r^2 = x^2 + y^2 + z^2$, and $\theta$ denotes coordinates on $S^2$. There are of course null coordinates $(v,u,\theta)$ which are defined by $v = t + r$ and $u = t - r$.

Because we will use auxiliary multipliers as is described in Section~\ref{sec:decay}, it will also be useful to introduce coordinates adapted to these auxiliary multipliers. We take some point $(s,x_0,y_0,z_0)$ as the center of the ball in $\Sigma_s$ where the data for the auxiliary multiplier will be supported. Following the notation in Section~\ref{sec:decayestimates}, we denote by $\tau$ the $u$ coordinate of this point, meaning that $\tau = s - \sqrt{x_0^2 + y_0^2 + z_0^2}$. We shall denote by $(t,r',\theta')$ polar coordinates adapted to this point (meaning that $(r')^2 = (x - x_0)^2 + (y - y_0)^2 + (z - z_0)^2$), and we shall denote by $(v',u',\theta')$ null coordinates adapted to this point. Because we are thinking of solving for $f$ backwards in time, we are actually setting $v' = s - t + r'$ and $u' = s - t - r'$. Thus, for $0 \le t \le s$, we have that $f$ is compactly supported in $u'$ if we use the strong Huygens principle.

Because $\psi$ decays in $u$ and $f$ decays in $u'$, it is additionally useful to introduce coordinates which are well adapted to how both of these functions decay. Thus, it makes sense to use coordinates given by $(t,r,r',\phi)$ and $(t,u,u',\phi)$. Given level sets of $r$ and $r'$, we have that the intersection of these two spheres is either empty, a circle, or a point. When it is a circle, $\phi$ denotes an angular coordinate on this circle. The coordinates degenerate when the intersection is a point, but this is a set of measure $0$. The volume form of $\R^{3 + 1}$ is comparable to
\begin{equation} \label{eq:dtdrdr'dphi}
    \begin{aligned}
    {r r' \over R} d t \wedge d r \wedge d r' \wedge d \phi,
    \end{aligned}
\end{equation}
and also to
\begin{equation} \label{eq:dtdudu'dphi}
    \begin{aligned}
    {r r' \over R} d t \wedge d u \wedge d u' \wedge d \phi,
    \end{aligned}
\end{equation}
where $R^2 = (x_0)^2 + (y_0)^2 + (z_0)^2$. We refer the reader to \cite{AndPas19} where this fact is proved and where these coordinates are described in more detail.

With these considerations in hand, we can now turn to recovering the bootstrap assumptions.

It is clear that assumed bootstrap assumptions \eqref{eq:cubicenbootstrap} and \eqref{eq:cubicpointwisebootstrap} directly recover the energy bootstrap assumptions. Indeed, if we commute the equation $N$ times with unit partial derivatives in order to propagate $\Vert \partial^\alpha \partial \psi \Vert_{L^2 (\Sigma_t)}$ for all $|\alpha| \le N$, in the nonlinearity, there can be at most ${N \over 2} + 1$ derivatives falling on the term with the second most derivatives. We can, thus, put the term with the most derivatives in $L^2$, and the pointwise decay for the two remaining terms is enough to give us a convergent integral. We must now recover the pointwise bootstrap assumptions.

We will now recover the pointwise bounds on some $\Sigma_s$ with $0 \le s \le T$. We shall use a version of Proposition~\ref{prop:decay} in order to do so. In particular, we shall use one adapted to Lemma~\ref{lem:bilinintM} instead of Lemma~\ref{lem:bilinenM}. This results in pointwise decay for $\psi$ itself and not just $\partial_t \psi$. The necessary result analogous to Proposition~\ref{prop:decay} and the duality argument this requires follow in similar fashions to the ones proved in Section~\ref{sec:decay}. We are, thus, using $f$ as a multiplier for the wave equation as opposed to $\partial_t f$.

In order to implement this strategy, and following the notation above, we take some point $(s,x_0,y_0,z_0)$ in $\Sigma_s$. We recall that $\tau$ is the $u$ coordinate of this point. We consider initial data for an auxiliary solution to the wave equation $f$ having initial data whose support is contained in a unit ball centered at this point in $\Sigma_s$. Now, we must recover pointwise bounds up to order ${s \over 2} + 1$. More precisely, we must show that
\begin{equation}
\begin{aligned}
|\partial^\alpha \psi| \le {C \epsilon \over (1 + s) (1 + |\tau|^{1 - \delta})},
\end{aligned}
\end{equation}
where $|\alpha| \le {N \over 2} + 1$. We must control error integrals of the kind found in \eqref{eq:reqerrorbounds} with $f$ instead of $\partial_t f$. This provides pointwise bounds for $\partial^\alpha \psi$ after commuting with $\partial^\alpha$, and then further commuting by $\partial^\beta$ for all $|\beta| \le k + n + 1 = 6$. Thus, in order to show decay for $|\partial^\alpha \psi|$, we must commute the equation by $\partial^\beta \partial^\alpha$ for all $|\beta| \le 6$ where $\beta$ consists of only spatial derivatives.

We consider the worst term in this commutation, which is when all of the derivatives fall on the same term in the nonlinearity. We have that
\begin{equation}
\begin{aligned}
\Box \partial^\beta \partial^\alpha \psi = N + (\partial_t \psi)^2 (\partial^\beta \partial^\alpha \partial_t \psi),
\end{aligned}
\end{equation}
where the other terms in $N$ can be treated in either the same way as we shall handle this remaining term, or in even easier ways.

Now, we know that we can write
\begin{equation}
\begin{aligned}
\partial^\beta \partial^\alpha \partial_t \psi = \partial^\beta \partial^{\alpha'} \partial_i \partial_t \psi,
\end{aligned}
\end{equation}
where $\partial_i$ is some spatial derivative. If $\partial^{\alpha'}$ contains a time derivative (of which it can have at most one), we can rewrite this as
\begin{equation}
\begin{aligned}
\partial^\beta \partial_i \partial^{\alpha''} \psi + N,
\end{aligned}
\end{equation}
where $N$ consists of higher order nonlinearities after using the equation. These terms are all better, so we disregard them. If it does not contain a time derivative, we can incorporate the $\partial_t$, giving us
\begin{equation}
\begin{aligned}
\partial^\beta \partial_i \partial^{\alpha''} \psi.
\end{aligned}
\end{equation}
In either case, we note that $\partial^{\alpha''} \psi$ is something we can assume pointwise bounds for using the bootstrap assumptions. This means, of course, that taking any ball $B$ of radius $2$ in some $\Sigma_t$ where $t \le s$ having center with some $u$ coordinate $u$, we have that
\begin{equation} \label{eq:cubicenball1}
    \begin{aligned}
    \Vert \partial^{\alpha''} \psi \Vert_{L^2 (B)} \le {C \epsilon^{{3 \over 4}} \over (1 + t) (1 + |u|^{1 - \delta})}.
    \end{aligned}
\end{equation}
Similarly, using the bootstrap assumptions on the energy, we have that
\begin{equation} \label{eq:cubicenball2}
    \begin{aligned}
    \Vert \partial^\gamma \partial^\beta \partial_i \partial^{\alpha''} \psi \Vert_{L^2 (B)} \le C \epsilon^{{3 \over 4}},
    \end{aligned}
\end{equation}
for all $\gamma$ such that $|\gamma| + |\beta| + 1 + |\alpha''| \le N$, where $\partial^\gamma$ consists of only spatial derivatives. Now, using Sobolev embedding in $\R^3$, we know that
\begin{equation}
\begin{aligned}
\Vert \partial^\beta \partial_i (\chi \partial^{\alpha''} \psi) \Vert_{L^\infty} \le \Vert \partial^\beta \partial_i (\chi \partial^{\alpha''} \psi) \Vert_{H^2 (\Sigma_t)},
\end{aligned}
\end{equation}
where $\chi$ is a smooth cutoff function equal to $1$ in the ball of radius $1$ concentric with $B$, decaying monotonically to $0$, and equal to $0$ outside of $B$. Moreover, we have that

\begin{equation}
\begin{aligned}
\Vert \partial^\beta \partial_i (\chi \partial^{\alpha''} \psi) \Vert_{H^2 (\Sigma_t)} \le \Vert \chi \partial^{\alpha''} \psi \Vert_{H^9 (\Sigma_t)},
\end{aligned}
\end{equation}
where we are using the fact that $|\beta| \le 6$. Now, we have that $\Vert \chi \partial^{\alpha''} \psi \Vert_{L^2 (\Sigma_t)} \le {C \epsilon^{{3 \over 4}} \over (1 + t) \left (1 + |u|^{1 - \delta} \right )}$ by \eqref{eq:cubicenball1} above, and we additionally have that $\Vert \chi \partial^{\alpha''} \psi \Vert_{H^{{N \over 2} - 1}} \le C \epsilon^{{3 \over 4}}$ by \eqref{eq:cubicenball2} above. Interpolating and with $N$ sufficiently large in terms of $\delta$ gives us that
\begin{equation}
\begin{aligned}
\Vert \chi \partial^{\alpha''} \psi \Vert_{H^9 (\Sigma_t)} \le {C \epsilon \over (1 + t)^{1 - \delta} \left (1 + |u|^{(1 - \delta)^2} \right )},
\end{aligned}
\end{equation}
meaning that we have that
\begin{equation}
\begin{aligned}
|\partial^\beta \partial_i (\chi \partial^{\alpha''} \psi)| \le {C \epsilon \over (1 + t)^{1 - \delta} \left (1 + |u|^{(1 - \delta)^2} \right )}.
\end{aligned}
\end{equation}
Now, because $\chi = 1$ in a ball of radius $1$ concentric with $B$, this implies that
\begin{equation}
\begin{aligned}
|\partial^\beta \partial_i \partial^{\alpha''} \psi| \le {C \epsilon \over (1 + t)^{1 - \delta} \left (1 + |u|^{(1 - \delta)^2} \right )}
\end{aligned}
\end{equation}
on that same ball. Because this can be done for any such ball, we see that we can use the pointwise estimates for $\partial^\beta \partial^\alpha \psi$ up to a loss of $t^\delta$ and $u^\delta$.

Thus, after commuting the equation by $\partial^\beta \partial^\alpha$ in order to apply Proposition~\ref{prop:decay}, we see that the error integral we must control is of the form
\begin{equation}
\begin{aligned}
\left |\int_0^s \int_{\Sigma_t} N' \phi d x d t \right | + \left |\int_0^s \int_{\Sigma_t} (\partial_t \psi)^2 (\partial^\beta \partial^\alpha \partial_t \psi) f d x d t \right |,
\end{aligned}
\end{equation}
where, once again, the term $N'$ can be controlled in the same way as the other. We can now plug in the pointwise bounds for $\partial_t \psi$ and $\partial^\beta \partial^\alpha \psi = \partial^\beta \partial_i \partial^{\alpha''} \psi$ from the above to get the desired result. Now, we note that $|f| \le {C \over 1 + s - t}$ by the asymptotics we have for the linear wave equation, and we note that it is additionally compactly supported in the $u'$ coordinate. Thus, going into $(t,r,r',\phi)$ coordinates, we have that the error integral is controlled by
\begin{equation}
\begin{aligned}
C \epsilon^{{9 \over 4}} \int_{{\tau \over 10}}^{s - {\tau \over 10}} \int_0^{s + 2} \int_{s - t - 5}^{s - t + 5} \int_0^{2 \pi} {1 \over (1 + t)^{3 - \delta} (1 + |u|^{{5 \over 2}})} {1 \over 1 + s - t} {(1 + s - t) (1 + t) \over s + 1} d \phi d r d r' d t
\end{aligned}
\end{equation}
as long as $\tau \le {9 s \over 10}$ (we are using that, when this is satisfied, we have that $R$ is comparable to $s$ in \ref{eq:dtdrdr'dphi}). We are also assuming that $\delta$ is sufficiently small in order to bound the losses in decay in $u$, which depend on $\delta$, by ${1 \over 2}$. The integral bounds on $r'$ come from the fact that $f$ is compactly supported with respect to $u'$. The fact that the lower bound on the $t$ integral is ${\tau \over 10}$ and the upper bound is $s - {\tau \over 10}$ deserves explanation. We are trying to show that $|\partial^\alpha \psi|$ is controlled on $\Sigma_s$ at a point whose $u$ coordinate is given by this $\tau$. Thus, the initial data for $f$ is prescribed in the ball of radius $1$ centered at this point. Because of this, and using the strong Huygens principle for $\phi$ and a domain of dependence argument for $\psi$, the support of $\phi$ and the support of $\psi$ do not intersect in $\Sigma_t$ for $t \le {\tau \over 10}$. These considerations hold only when $\tau \ge 20$. Otherwise, we simply take $0$ as the lower bound for this integral. These considerations will be worked out more rigorously in an analogous setting in Lemma~\ref{lem:rrbar}.

Using the fact that
\begin{equation}
\begin{aligned}
\int_0^{s + 2} {1 \over (1 + |u|^{{5 \over 2}})} d r \le C,
\end{aligned}
\end{equation}
we can control the above integral by
\begin{equation}
\begin{aligned}
{C \epsilon^{{9 \over 4}} \over 1 + t} \int_{{\tau \over 10}}^{s - {\tau \over 10}} {1 \over (1 + t)^{2 - \delta}} d s \le {C \epsilon^{{9 \over 4}} \over (1 + s) (1 + |\tau|^{1 - \delta})}.
\end{aligned}
\end{equation}

We now consider the region $\tau \ge {9 t \over 10}$. For this, we further decompose into two regions. Without loss of generality, we have that the data for $f$ is supported in a ball whose center has $y$ and $z$ coordinate equal to $0$, and we have that $(r')^2 = (x - a)^2 + y^2 + z^2$. We first consider the case where $a \ge 100$. Using the $(t,u,u',\phi)$ coordinate system, we have that the error integral is controlled by
\begin{equation}
\begin{aligned}
C \epsilon^{{9 \over 4}} \int \int \int \int {1 \over (1 + t)^{3 - \delta} (1 + |u|^{{5 \over 2}})} {1 \over 1 + s - t} {(1 + s - t) (1 + t) \over a} d \phi d u d u' d t,
\end{aligned}
\end{equation}
where we shall establish the bounds of integration shortly. We first note that the integral in $u'$ is can always be taken from $-5$ to $5$ because $f$ is compactly supported in $u'$. Then, we shall break up the integral in $u$ into pieces of length $10$. More specifically, we can control the above integral by
\begin{equation}
\begin{aligned}
C \epsilon^{{9 \over 4}} \sum_{i = 0}^\infty \int_{5 (i - 1)}^{5 (i + 1)} \int \int \int_{-5}^5 {1 \over (s + 1)^{3 - \delta} (1 + |5 i|^{{5 \over 2}})} {1 \over 1 + s - t} {(1 + s - t) (s + 1) \over \overline{u}} d u_2 d s d \theta d u_1.
\end{aligned}
\end{equation}

Now, we note that the bounds for $\phi$ are of course given by $0$ and $2 \pi$. The final range we must establish is that of $t$. We take the line $t = 5 (i + 1) + x$ and $t = s + a - x + 5$ in the plane where $y = z = 0$, and we calculate the $s$ coordinate of the intersection. This is given by $s = {t + a + 5 (i + 2) \over 2}$. Similarly, we take the line $s = 5 (i - 1) - x$ and the line $s = t - a + x - 5$ and calculate the $s$ coordinate of the intersection. This is given by $t = {s - a + 5 (i - 2) \over 2}$. These give us the bounds of integration, as we know that the lowest and highest $t$ coordinates must occur on the plane where $y = z = 0$ (this is considered more carefully in an analogous situation in Lemma~\ref{lem:rtr's-t}). Thus, the integral is controlled by
\begin{equation}
\begin{aligned}
C \epsilon^{{9 \over 4}} \sum_{i = 0}^\infty \int_{5 (i - 1)}^{5 (i + 1)} \int_{{s - a + 5 (i - 2) \over 2}}^{{s + a + 5 (i + 1) \over 2}} \int_0^{2 \pi} \int_{-5}^5 {1 \over (1 + t)^{3 - \delta} \left (1 + |5 i|^{{5 \over 2}} \right )} {1 \over 1 + s - t} {(1 + s - t) (1 + t) \over a} d u' d t d \phi d u.
\end{aligned}
\end{equation}

Now, we note that the length of $t$ integration is always comparable to $a$ by the above bounds (note that $a \ge 100$). Moreover, in this region, both $1 + t$ and $1 + s - t$ are comparable to $s$. Thus, we have that we can control the above integral by
\begin{equation}
\begin{aligned}
C \epsilon^{{9 \over 4}} \sum_{i = 0}^\infty {1 \over 1 + |5 i|^{{5 \over 2}}} {1 \over (1 + s)^{2 - \delta}},
\end{aligned}
\end{equation}
giving us the desired result. The final region where $\overline{u} \le 100$ can be controlled by comparing with a downward pointing light cone of thickness $200$ where the tip lies along the $t$ axis.
\end{proof}

\section{Anisotropic systems of wave equations} \label{sec:anisotropic}
We now precisely describe the setting of the main application for the techniques introduced in this paper. We recall that
\[
\Box = -\partial_t^2 + \partial_x^2 + \partial_y^2,
\]
and we introduce the notation
\[
\Box' = -\partial_t^2 + \lambda_1^{-2} \partial_x^2 + \lambda_2^{-2} \partial_y^2.
\]
We assume that ${1 \over 4} \le \lambda_1^2 \le {1 \over 2}$ and $2 \le \lambda_2^2 \le 4$.

We shall consider the following anisotropic system of wave equations in $2 + 1$ dimensions:
\begin{equation} \label{eq:anisotropic}
\begin{aligned}
\Box \psi + (\partial_t \phi)^2 (\partial_t \psi) (\partial_t^2 \psi) = (\partial_t \phi) (\partial_t \psi)^2 + (\partial_t \psi) (\partial_t \phi)^2 + (\partial_t \psi)^4
\\ \Box' \phi + (\partial_t \psi)^2 (\partial_t \phi) (\partial_t^2 \phi) = (\partial_t \psi) (\partial_t \phi)^2 + (\partial_t \phi) (\partial_t \psi)^2 + (\partial_t \phi)^4.
\end{aligned}
\end{equation}
We take smooth data for these equations $\psi(0,x) = \psi_0 (x)$, $\partial_t \psi(0,x) = \psi_1 (x)$, $\phi(0,x) = \phi_0 (x)$, and $\partial_t \phi(0,x) = \phi_1 (x)$ supported in the ball of radius ${1 \over 10}$. Moreover, we assume that the data are arbitrary functions with $\Vert \psi_0 \Vert_{H^{N + 1}} + \Vert \psi_1 \Vert_{H^N} + \Vert \phi_0 \Vert_{H^{N + 1}} + \Vert \phi_1 \Vert_{H^N} \le \epsilon$ for some $\epsilon$. With this, we are ready to state the main Theorem of this paper.

\begin{theorem} \label{thm:mainthm}
There exist $\epsilon_0 > 0$ sufficiently small and $N_0$ sufficiently large such that the trivial solution of \eqref{eq:anisotropic} is globally nonlinearly asymptotically stable with respect to the perturbations described above whenever $\epsilon \le \epsilon_0$ and $N \ge N_0$.
\end{theorem}

Before proceeding to the proof of this result, we shall make some brief remarks.
\begin{enumerate}
\item \textbf{The fundamental restriction on $\lambda_1$ and $\lambda_2$ is that $\lambda_1 \ne 1$ and $\lambda_2 \ne 1$}. Indeed, the semilinear terms are cubic, which is critical in $2 + 1$ dimensions because solutions to the wave equation only decay like $t^{-{1 \over 2}}$. The condition that $\lambda_1 \ne 1$ and $\lambda_2 \ne 1$ produces a kind of null condition on interactions between these waves. When $\lambda_1 = 1$ or $\lambda_2 = 1$, the light cones are tangent instead of intersecting transversally, and parallel interactions in the tangential direction are not effectively suppressed. We are unsure what happens in this case, but would not be surprised if this situation would lead to blow up. In our setting, however, the null condition is thus that $\lambda_1 \ne 1$, $\lambda_2 \ne 1$, and that the same wave not appear cubed.

\item The other restrictions on $\lambda_1$ and $\lambda_2$ are purely for convenience. We restricted to the case $\lambda_1 < 1 < \lambda_2$ to force the light cones to intersect. In case both $\lambda_1$ and $\lambda_2$ are either smaller than $1$ or greater than $1$, the argument seems to still work, with the integrals appearing to be easier to control. We further restricted to the case of ${1 \over 4} \le \lambda_1^2 \le {1 \over 2}$ and $2 \le \lambda_2^2 \le 4$ simply to make the equations symmetric under the change of coordinates described in Section~\ref{sec:coordinates}. Indeed, an appropriate rescaling of the coordinates makes $\Box'$ become $\Box$ (see Section~\ref{sec:coordinates}), and this restriction ensures that the form of the equation is preserved under this change of coordinates. We do, however, note that the estimates are not uniform in $\lambda_1$ and $\lambda_2$, and a variety of constants may degenerate as $\lambda_1$ or $\lambda_2$ approach $0$, $1$, or $\infty$. When $\lambda_1$ and $\lambda_2$ are uniformly away from these quantities, we believe the estimates can be made uniform.

\item The quasilinear terms in \eqref{eq:anisotropic} are present to show that the method is applicable to quasilinear equations. They can be more general.

\item We have included the fourth order nonlinearity just to show that it can be handled. We can, of course, introduce higher order nonlinearities and higher order mixed nonlinearities and deal with them in the same way. We note, however, that the equation for $\psi$ does not have a nonlinear term depending only on $\phi$, and similarly for the equation for $\phi$. While we believe the techniques used in this paper are applicable to study this problem, they would require a new analysis. In fact, because the light cones for $\psi$ and $\phi$ have different causal structures, introducing these kinds of nonlinearities is related to removing the assumption of compact support, which we discuss below. Also, the derivatives in the nonlinearity are seen to all be $\partial_t$, and Proposition~\ref{prop:decay} is specifically adapted to showing decay for this derivative. Treating other derivatives in the nonlinearity requires relatively minor modifications to the considerations in Section~\ref{sec:decay}.

\item We have not tried to optimize the results in this paper, and they are far from sharp in terms of, for example, regularity. Also, it seems as though the same strategy can work after weakening the assumptions. For example, it seems as though we can assume less decay for solutions to the homogeneous wave equation and a very similar argument will, at least until we assume less than integrable decay in $u$ for first derivatives. With this assumption, the analysis may have to be changed. A very natural way to try to strengthen the result is to remove the assumption that the data are compactly supported (and to replace this with the assumption that the data are localized), and which we now discuss.

\item The proof uses that the data are compactly supported in the unit ball. However, we believe that the strategy followed in the proof of Theorem~\ref{thm:mainthm} can be used in the more general setting of data that decay sufficiently rapidly away from the origin. This would, however, require an analysis of the different geometry arising from this situation (see Section~\ref{sec:closingpointwise} to see where the geometry is necessary, and see Sections~\ref{sec:geometry} and \ref{sec:SGeometry} to see where the geometry involved in the proof of Theorem~\ref{thm:mainthm} is studied). More precisely, we note that it seems as though all of the nonlinear terms except for $(\partial_t \psi) (\partial_t \phi)^2$ in the equation for $\psi$ (and the analogous term in the equation for $\phi$) can be treated using the considerations below along with the observation that, if $\tau < -100$, we in fact have that $|u|$ is comparable to $\tau$ in the entire support of the error integral (see Sections~\ref{sec:decayestimates} and \ref{sec:coordinates} where $u$ and $\tau$ are described). For the remaining term, it seems as though an analysis of how the geometry of the scaling vector field interacts with the light cone of $f$ when $\tau < 0$ is necessary.

\item The estimates on the error integrals involving cubic interactions below are far from optimal. In Section~\ref{sec:anisotropicdescription}, some of the additional gains from studying the geometry more carefully will be described.
\end{enumerate}

\subsection{Description of the proof} \label{sec:anisotropicdescription}
The proof follows the same strategy as the proofs in Section~\ref{sec:simpleapps}. Indeed, we use the strategy outlined in Section~\ref{sec:decayestimates} to prove decay. Thus, we have to control error integrals of the form
\[
\int_0^s \int_{\Sigma_t} F(\psi,\phi) \partial_t f d x d t,
\]
where $f$ is the auxiliary solution to the wave equation (see Proposition~\ref{prop:decay}). Once again, we shall interpolate between the energy and pointwise estimates in order to effectively take advantage of the transversal intersection of the cones. The main differences arise from the fact that we shall combine this method with commuting with weighted vector fields from the vector field method. We shall commute with the scaling vector field $S = t \partial_t + r \partial_r$. Because we need a norm that allows us to interpolate with scaling as well, we must use a Morawetz estimate. To avoid discussing this in detail, we shall actually use a simpler estimate that is very similar to a Morawetz estimate (see Lemma~\ref{lem:spacetimeL2}). The important fact is that this estimate, like the Morawetz estimate, is sharp in terms of decay (up to a loss of $t^\delta$). Controlling the error integrals will require more geometric information than was needed for the simple applications in Section~\ref{sec:simpleapps}. However, as was noted in the remarks after Theorem~\ref{thm:mainthm}, we are still not bounding at least some of the terms optimally. One example of this stems from the fact that tubes formed by considering $u$ and $\overline{u}$ neighborhoods of thickness comparable to $1$ intersect the downward pointing light cone of $f$ in a set of small measure in $\tau$ (see Section~\ref{sec:coordinates} for a description of $\overline{u}$). Facts such as these, which can be seen by examining the intersections of the three cones more carefully, will not be needed in this argument. However, it may be useful to keep in mind for other applications that it seems to be possible to get stronger estimates.

We shall now describe the proof in a bit more detail. We shall first focus on how the scaling vector field is used in the context of the strategy described in Section~\ref{sec:decay}. In order to apply Proposition~\ref{prop:decay} in a bootstrap argument, we must show that the nonlinear error integral has good powers of $\tau$. The vector field $S$ will help us in getting good powers of $\tau$. The main observation is that the downward opening light cone associated to the auxiliary multiplier $f$ is everywhere pierced by $S$ (see Figure~\ref{fig:lightconesaux}). This means that good derivatives of $f$ and $S$ span the tangent space of $\R^{2 + 1}$. We may, thus, write any vector field in terms of the frame consisting of $S$ and good derivatives of $f$. In particular, we may schematically write that $\partial_t f = \overline{\partial}_f f + \gamma S f$ where $\overline{\partial}_f$ denotes good derivatives of $f$. Thus, we may write that
\[
\int_0^s \int_{\Sigma_t} F(\psi,\phi) \partial_t f d x d t = \int_0^s \int_{\Sigma_t} F(\psi,\phi) (\overline{\partial}_f f + \gamma S (f)) d x d t.
\]
The term involving good derivatives of $f$ is better than the one involving $S$, so to a first approximation, we can assume that this term satisfies improved estimates. For the term involving $S$, we can integrate by parts in $S$, giving us that, schematically,
\[
\int_0^s \int_{\Sigma_t} F(\psi,\phi) \gamma S(f) d x d t = \int_0^s \int_{\Sigma_t} S(F(\psi,\phi)) \gamma f + S(\gamma) F(\psi,\phi) f + \gamma F(\psi,\phi) f  d x d t.
\]
Now, because we may commute the equations for $\psi$ and $\phi$ with $S$, we note that $S(F(\psi,\phi)) \approx F(\psi,\phi)$ in the sense that any estimates that $F(\psi,\phi)$ satisfies are, to a first approximation, satisfied by $S(F(\psi,\phi))$. Thus, the gain comes from noting that $\gamma$ and $S(\gamma)$ satisfy good estimates in certain regions. More specifically, they can be shown to be of size ${1 \over \tau}$. Thus, we can use $S$ to gain a power of ${1 \over \tau}$, which helps in showing that the nonlinear error integrals are of the necessary size. Proving these estimates on $\gamma$ and $S(\gamma)$ involves studying the geometry of $S$ in relation to the various cones in question. This is carried out in Section~\ref{sec:SGeometry}.

We now describe how we estimate the geometry involved. To control quartic and higher order nonlinearities, we must estimate the geometry coming from the intersection of a forward openining and downward opening cone. This is because the error integrals are of the form
\[
\int_0^s \int_{\Sigma_t} (\partial_t \psi)^4 (\partial_t f) d x d t,
\]
and because $\psi$ decays away from a forward opening cone while $f$ decays away from a downward opening cone.

To a first approximation, $\psi$ is supported in a forward pointing cone of thickness $1$ and $f$ is supported in a downward pointing cone of thickness $1$. If we take the intersection of this configuration with $\Sigma_t$, we are left with two annuli of thickness $1$, and it is natural to compute the volume of intersection order to control the interaction. Of course, these kinds of estimates depend on the parameters $s$ and $\tau$. These estimates are one example of the kind of analysis that must be done in order to control the nonlinearity. Such estimates are carried out in Section~\ref{sec:geometry}.

The interactions between $\psi$ and $\phi$ can be controlled using the condition that $\lambda_1 \ne 1$ and $\lambda_2 \ne 1$. Geometrically, this means that either one cone is contained within the other, or that the two cones intersect transversally. This paper focuses on the case where the two light cones intersect transversally as this case is more interesting, although we note that a similar proof can be used in the other case. Thus, we restrict ourselves to this case.

There are now two forward opening cones to consider, the one associated to $\psi$ and the one associated to $\phi$. If we take two such cones of thickness $1$, if we look in $\Sigma_t$, we are left with two annuli of thickness $1$. One annulus is circular and the other is elliptical. When the cones intersect transversally, we note that the area of intersection of these annulur regions is comparable to $1$. Thus, decay away from the light cone gives localization to a small set, and the mechanism of improvement becomes similar to the case of the localized nonlinearity which was studied in Section~\ref{sec:localizednonlinearity} above. Indeed, the total area of the set of interaction in $\Sigma_t$ becomes $1$ instead of $t$, which is what it is between waves having the same light cone. Figure~\ref{fig:crosssections} shows the cross sections of the cones for different values of $\lambda_1$ and $\lambda_2$. The left configuration is representative of the case considered in this paper (i.e., $\lambda_1 < 1$ and $\lambda_2 > 1$). The second corresponds to $\lambda_1 = 1$, which we have just discussed. The third is a configuration where $\lambda_1 > 1$ and $\lambda_2 > 1$, and although it is not treated in this paper, we believe that similar ideas can be used in this case.

These considerations also reveal why having $\lambda_1 \ne 1$ and $\lambda_2 \ne 1$ is important. If we consider the case where, say, $\lambda_1 = 1$, we can see that the area of intersection of the two annuli in $\Sigma_t$ is comparable to $\sqrt{t}$. Given that cubic nonlinearities are critical in $2 + 1$ dimensions, one may ask why the gain over $t$ is not enough. However, in this case, there exist wave packets which still experience amplification caused by interactions on sets with maximal measure. Indeed, when $\lambda_1 = 1$, we note that the two light cones intersect in two null lines. Wave packets propagating along those null lines will interact in essentially the same way as if the light cones were the same. This configuration looks very much like a $2 + 1$ dimensional version of the localized nonlinearity studied in Section~\ref{sec:localizednonlinearity} above, where we take $\omega = {1 \over 2}$. Because of this behavior, we believe that this configuration is likely to lead to blow up in finite time.
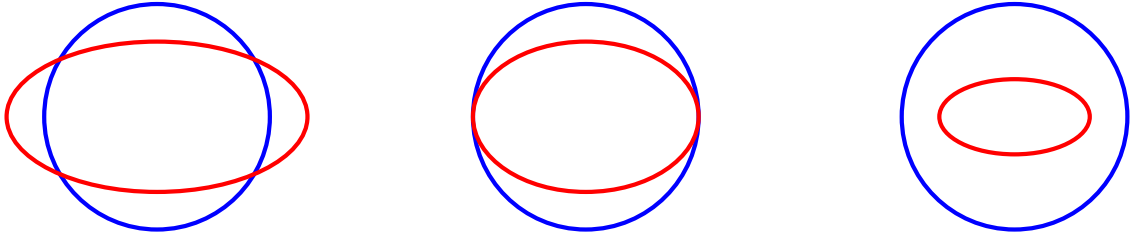
\begin{figure}
    \centering
    \begin{tikzpicture}
    \draw[color=blue,ultra thick] (-2,0) circle (1.5);
    \draw[color=red,ultra thick] (-2,0) ellipse (2 and 1);
    \draw[color=blue,ultra thick] (3.7,0) circle (1.5);
    \draw[color=red,ultra thick] (3.7,0) ellipse(1.5 and 1);
    \draw[color=blue,ultra thick] (9.4,0) circle (1.5);
    \draw[color=red,ultra thick] (9.4,0) ellipse(1 and 0.5);
    \end{tikzpicture}
    \caption{Cross sections of the light cones for $\phi$ and $\psi$ for different $\lambda_1$ and $\lambda_2$. More specifically, the blue circles represent intersections of the light cone of $\psi$ with $\Sigma_t$, and the red ellipses represent intersections of the light cone of $\phi$ with $\Sigma_t$.}
    \label{fig:crosssections}
\end{figure}

We hope it is clear from the above considerations that the geometry of the cones must be effectively analyzed. This was already present in the example considered in Section~\ref{sec:cubic}, and we in fact use some of the same observations that are found therein. However, the geometry requires more substantial analysis in the case of anisotropic systems of wave equations which we are now considering.

\subsection{Notation, coordinates, and the parameters} \label{sec:coordinates}
We take $\R^{2 + 1}$ with the usual $(t,x,y)$ coordinates. In these coordinates, we recall that the equations in question (see \eqref{eq:anisotropic}) are of the form
\begin{equation}
    \begin{aligned}
    \Box \psi = -\partial_t^2 \psi + \partial_x^2 \psi + \partial_y^2 \psi = F(\partial_t \psi,\partial_t \phi,\partial_t^2 \psi),
    \\ \Box' \phi = -\partial_t^2 \psi + \lambda_1^{-2} \partial_x^2 \psi + \lambda_2^{-2} \partial_y^2 \psi = F(\partial_t \psi,\partial_t \phi,\partial_t^2 \psi).
    \end{aligned}
\end{equation}
We shall also use the coordinates $(t,r,\theta)$ and $(t,\overline{r},\overline{\theta})$. These coordinates are defined as follows. The functions $r$ and $\theta$ are the usual polar coordinates on $\R^2$, meaning that $r = \sqrt{x^2 + y^2}$ and $\theta = \arctan \left ({y \over x} \right )$. We now consider the coordinates $(t,\overline{x},\overline{y})$ given by $\overline{x} = \lambda_1 x$ and $\overline{y} = \lambda_2 y$. We then have that $\partial_{\overline{x}} = \lambda_1^{-1} \partial_x$ and that $\partial_{\overline{y}} = \lambda_2^{-1} \partial_y$. Thus, in the $(t,\overline{x},\overline{y})$ coordinates, the equations can be written in the form

\begin{equation}
\begin{aligned}
\Box \psi = -\partial_t^2 \psi + \lambda_1^2 \partial_{\overline{x}}^2 \psi + \lambda_2^2 \partial_{\overline{y}}^2 \psi &= F(\partial_t \psi,\partial_t \phi,\partial_t^2 \psi),
\\ \Box' \phi = -\partial_t^2 \phi + \partial_{\overline{x}}^2 \phi + \partial_{\overline{y}}^2 \phi &= G(\partial_t \psi,\partial_t \phi,\partial_t^2 \phi).
\end{aligned}
\end{equation}

We define $\overline{r} = \sqrt{\overline{x}^2 + \overline{y}^2} = \sqrt{\lambda_1^2 x^2 + \lambda_2^2 y^2}$, and we define $\overline{\theta} = \arctan \left ({\overline{y} \over \overline{x}} \right ) = \arctan \left ({\lambda_2 y \over \lambda_1 x} \right )$. We also introduce the functions $u = t - r$ and $\overline{u} = t - \overline{r}$. We have that
\[
S = t \partial_t + x \partial_x + y \partial_y = t \partial_t + r \partial_r = t \partial_t + \overline{r} \partial \overline{r}.
\]
This is a consequence of the fact that
\[
r \partial_r = x \partial_x + y \partial_y = \lambda_1 x \lambda_1^{-1} \partial_x + \lambda_2 y \lambda_2^{-1} \partial_y = \overline{x} \partial_{\overline{x}} + \overline{y} \partial_{\overline{y}} = \overline{r} \partial_{\overline{r}}.
\]

More geometrically, there is a metric $g$ adapted to $\Box$ and a metric $\overline{g}$ adapted to $\Box'$. In $(t,x,y)$ coordinates, the metric $g$ is given by the constant coefficient diagonal matrix whose entries are $-1$, $1$, and $1$ along the diagonal. The metric $\overline{g}$ is also constant coefficient and diagonal, but its entries are given by $-1$, $\lambda_1^2$, and $\lambda_2^2$. In $(t,\overline{x},\overline{y})$ coordinates, the metric $\overline{g}$ is constant coefficient and diagonal with entries $-1$, $1$, and $1$, while the metric $g$ is constant coefficient and diagonal with entries $-1$, $\lambda_1^{-2}$, and $\lambda_2^{-2}$.

Because it will be necessary to consider the the geometry involving the level sets of $u$ and $\overline{u}$ (these are cones opening to the future with circular and elliptical cross sections, respectively), we note that the level set $u = 0$ and the level set $\overline{u} = 0$ intersect in four straight lines emanating from the spacetime origin. These four straight lines each have a constant $\theta$ coordinate depending only on $\lambda_1$ and $\lambda_2$. These coordinates can be found by solving for $\tan(\theta) = {y \over x}$ in the equations
\[
t^2 = x^2 + y^2 = \lambda_1^2 x^2 + \lambda_2^2 y^2.
\]
We shall denote the $\theta$ coordinates of these lines by $\Theta_1$, $\Theta_2$, $\Theta_3$, and $\Theta_4$, taking $0 < \Theta_1 < {\pi \over 2}$, ${\pi \over 2} < \Theta_2 < \pi$, $\pi < \Theta_3 < {3 \pi \over 4}$, and ${3 \pi \over 4} < \Theta_4 < 2 \pi$. We note that all four of these values are bounded uniformly away from half integer multiples of $\pi$ given the restrictions we have placed on $\lambda_1$ and $\lambda_2$. We similarly use $\overline{\Theta}$ and $\overline{\Theta}_i$ for $1 \le i \le 4$ to denote the $\overline{\theta}$ coordinates of these quantities. These quantities are related to the $\Theta_i$ because
\[
\tan(\overline{\theta}) = {\overline{y} \over \overline{x}} = {\lambda_2 y \over \lambda_1 x} = {\lambda_2 \over \lambda_1} \tan(\theta).
\]

Once again, the function $f$ will always denote an auxiliary solution to the homogeneous wave equation $\Box f = 0$ or $\Box' f = 0$. Which equation $f$ solves will be clear from the context. Because $f$ is the test function we use as a multiplier, it is natural to also introduce coordinates adapted to $f$. We introduce $r'$, $\theta'$, and $u'$ for $f$ when $f$ solves the equation $\Box f = 0$. When it solves $\Box' f = 0$, we introduce $\overline{r}'$, $\overline{\theta}'$, and $\overline{u}'$ instead, and these quantities are defined analogously (see below). Now, after rescaling $x$ and $y$ to $\overline{x}$ and $\overline{y}$, we note that we can always assume that $f$ solves $\Box f = 0$. Thus, in the following discussion, we shall make this assumption.

We denote by $s$ the time slice in which data for $f$ is posed. We then denote by $\tau$ the $u$ coordinate of the center of the ball in $\Sigma_s$ where the data for $f$ is supported, and by $\Theta$ the $\theta$ coordinate of this point. If the center of this ball is given by $(s,a,b)$, we denote by $r'$ the radial coordinate away from this point, that is, $(r')^2 = (x - a)^2 + (y - b)^2$, and we denote by $\theta'$ the angular coordinate for this point normalized such that $\theta' = 0$ along the line $\theta = \Theta$, meaning that $\theta' = \arctan \left ({y - b \over x - a} \right ) - \Theta$. When proving geometric estimates that depend on $\tau$ and $s$, it will often be possible to assume that $b = 0$. When $b \ne 0$, we may of course redefine coordinates to make $b = 0$. We must be cautious, however, about circumstances where the geometry of the other light cone must be taken into account (rotating an ellipse does not produce the same ellipse because it changes the axes). Because we shall often not have to take the geometry of all three cones into account at the same time, this is still a useful observation. The only time we must consider the geometry of all three light cones for a geometric estimate (i.e., the geometry of the light cones for $\psi$, $\phi$, and $f$) is in Lemma~\ref{lem:dr'drbar}. Assuming that $b = 0$ is the same as assuming that $\Theta = 0$. The reader may also wish to consult Figure~\ref{fig:psifSigma_t} where a diagram of some of these quantities is given.

We shall also introduce null coordinates adapted to $\psi$, $\phi$, and $f$. For $\psi$, the null coordinates are given by $(u,v,\theta)$ where $u = t - r$, $v = t + r$, and $\theta = \theta$. The null coordinate systems $(\overline{u},\overline{v},\overline{\theta})$ and $(u',v',\theta')$ are defined analogously. More precisely, we note that $\overline{v} = t + \overline{r}$, $\overline{u} = t - \overline{r}$, $v' = s - t + r'$, and $u' = s - t - r'$. We note that
\[
S = t \partial_t + x \partial_x + y \partial_y = v \partial_v + u \partial_u = \overline{v} \partial_{\overline{v}} + \overline{u} \partial_{\overline{u}}.
\]

Because we often have to consider the functions $\psi$, $\phi$, and $f$, we shall use $\overline{\partial}_\psi$, $\overline{\partial}_\phi$, and $\overline{\partial}_f$ to denote good derivatives for $\psi$, $\phi$, and $f$, respectively. These are precisely the derivatives tangent to the light cones adapted to these functions. We shall also introduce the null frames $L = \partial_t + \partial_r$, $\underline{L} = \partial_t - \partial_r$, and $\partial_\theta = \partial_\theta$. The null frames $\overline{\underline{L}}$, $\overline{L}$, $\partial_{\overline{\theta}}$ and $\underline{L}'$, $L'$, $\partial_{\theta'}$ are defined analogously. We note that these are simply rescaled versions of the null coordinate vector fields (for example, $L$ is a constant rescaling of $\partial_v$ and $\underline{L}$ is a constant rescaling of $\partial_u$). The good derivatives $\overline{\partial}_\psi$ of $\psi$ consist of $L$ and $\partial_\theta$, and similarly for the good derivatives of $\phi$ and $f$.

When doing estimates on the geometry in Sections~\ref{sec:geometry} and \ref{sec:SGeometry}, we can freely assume that $s$ is larger than some fixed constant. In the proof of Theorem~\ref{thm:mainthm}, we use $\epsilon$ to absorb all nonlinear contributions for $t \le {1 \over \delta_0^{10}}$. Thus, when analyzing the geometry in those sections in order to recover the pointwise estimates, we can assume that $s \ge {1 \over \delta_0^{10}}$.

We shall sometimes use $O$ notation, in which $f$ is said to be $O(g)$ if $f \le C g$ for some $C$. There are some instances in which we will have expressions which are $O(\delta_0)$ (see below where the parameter $\delta_0$ is described). In these cases, the constant $C$ will be implied to not depend on $\delta_0$.

We shall have to commute with several vector fields. In the simple applications studied in Section~\ref{sec:simpleapps}, we only commuted with translation vector fields. However, we shall also have to use the scaling vector field. Thus, in this section, the admissible commutation fields will be strings of translation vector fields and scaling vector fields. We shall denote by $\Gamma^\alpha$ such a string of scaling vector fields and translation vector fields. This collection of vector fields is closed under Lie brackets because the commutator of two commutation fields is another commutation field. Indeed, translation vector fields commute with each other and the scaling vector field commutes with itself. We also have that the commutator between $S$ and a translation vector field is the negative of that translation vector field. Thus, for example, we have that
\[
[S,\partial_t] = - \partial_t.
\]
This shows that the commutator of any two of our commutation fields is another commutation field. Thus, by induction, we have that
\[
\Gamma^\alpha (\partial_t h) = (\partial_t \Gamma^\alpha h) + \sum_\gamma c_\gamma (\partial_t \Gamma^\gamma h)
\]
for an appropriate collection of $\gamma$ and where $c_\gamma = \pm 1$, and we moreover have that $|\gamma| \le |\alpha| - 1$. These facts will be used in the proof.

Now, given data supported in the ball of radius $1$ whose $C^2$ norm is bounded by $10$, we note that the solution to the linear wave equation $\Box f = 0$ has that
\begin{equation} \label{eq:assumeddecay0}
    \begin{aligned}
    |f| \le {C \over \sqrt{1 + s - t} (1 + |u'|)^{{1 \over 2}}}.
    \end{aligned}
\end{equation}
This follows from the fundamental solution. Then, by commuting the homogeneous equation with Lorentz fields, we have that
\begin{equation} \label{eq:assumeddecay1}
    \begin{aligned}
    |\partial f| \le {C \over \sqrt{1 + s - t} (1 + |u'|)^{{3 \over 2}}},
    \end{aligned}
\end{equation}
and we have that
\begin{equation} \label{eq:assumeddecay2}
    \begin{aligned}
    |\overline{\partial}_f f| \le {C \over (1 + s - t)^{{3 \over 2}} (1 + |u|^{{1 \over 2}})}.
    \end{aligned}
\end{equation}
Thus, we shall take $n = 2$ and $p = {3 \over 2}$ in Proposition~\ref{prop:decay}. The improved decay of good derivatives is important because we shall decompose $\partial_t$ in terms of $S$ and $\overline{\partial}_f$, as was described in Section~\ref{sec:anisotropicdescription}.

Similarly, we note that the solution to the linear wave equation $\Box' f = 0$ with the same kind of data has that
\begin{equation} \label{eq:assumeddecay3}
    \begin{aligned}
    |\partial_t f| \le {C \over \sqrt{1 + s - t} (1 + |\overline{u}'|)^{{3 \over 2}}},
    \end{aligned}
\end{equation}
and that
\begin{equation} \label{eq:assumeddecay4}
    \begin{aligned}
    |\overline{\partial}_f f| \le {C \over (1 + s - t)^{{3 \over 2}} (1 + |\overline{u}'|^{{1 \over 2}})}.
    \end{aligned}
\end{equation}
We shall use these decay rates, but an examination of the proof below reveals that less decay would be enough. The main result that would require modification is Lemma~\ref{lem:annuliu'dec}, but even there, we believe the modification is not substantial. Moreover, we emphasize once again that it suffices to take decay rates as a black box. Indeed, all that is required is a proof of decay for the homogeneous equation with a sufficiently large class of test functions as data. The test functions we are taking are all smooth functions supported in a ball of radius $1$ with $C^2$ norm bounded by $10$.

We finally note that we make use of three smallness parameters in the following proof. They are $\epsilon$, $\delta$, and $\delta_0$. The parameter $\epsilon$ can be chosen to depend on $\delta$, $\delta_0$, and the constants involved in the problem, and it controls the size of the initial data. It must be chosen sufficiently small for the bootstrap assumptions to be recovered. Because we are allowing $\epsilon$ to depend on $\delta$ and $\delta_0$, we also allow the expressions $C$ and $c$ to depend on $\delta$ and $\delta_0$ as well. The expressions $C$ and $c$ represent positive numbers which may be very large or small depending on some universal constants and on $\delta$ and $\delta_0$. More specifically, $C$ will represent something that may be very large (such as, for example, ${1 \over \delta_0}$), while $c$ will represent something that may be very small (such as $\delta_0$). There are a few places where it will be important to have constants that do not depend on $\delta_0$ so that we can get smallness relative to them by making $\delta_0$ smaller (see, for example, Lemma~\ref{lem:dr'drbar}). Whenever this is necessary, we shall explicitly say so.

Now, the parameters $\delta \le {1 \over 2}$ and $\delta_0 \le {1 \over 2}$ must be chosen sufficiently small to make the following proof work. The parameter $\delta$ controls a loss in decay that comes from interpolation (see Lemma~\ref{lem:interpolation}), and it is also the parameter we have used to denote the (necessary unless restricting to a dyadic region) loss in decay from a modified Morawetz estimate (see Lemma~\ref{lem:spacetimeL2}). Meanwhile, the parameter $\delta_0$ is our universal smallness parameter for geometric estimates. For example, it will be beneficial to consider two regimes, one where $\tau$ is small compared to $s$ and another where it is comparable to $s$. We shall thus consider the cases $\tau \le \delta_0 s$ and $\tau \ge \delta_0 s$. There will be several other cases where we must compare quantities, and $\delta_0$ is the parameter we shall use to make this comparison. There exist bounds for $\delta$ and $\delta_0$ in the form of universal constants that make the following proof work, but we have not computed them. They can, however, be computed in principle.

We finally note the regularity parameter $N$. This denotes the $L^2$ Sobolev space for the initial data. This parameter is allowed to depend on $\delta$, and we note that $N \rightarrow \infty$ as $\delta \rightarrow 0$. This is required in order to effectively interpolate (see Lemma~\ref{lem:interpolation}). However, for every fixed $\delta > 0$, we note that $N$ is a finite positive integer. The dependence of $N$ on $\delta$ can be explicitly computed (meaning that a lower bound for the regularity we require on the initial data can be computed), but we have not done so.

\subsection{Additional linear estimate} \label{sec:Morawetzreplacement}
We only require one additional linear estimate for this problem. In order to get a sharper result, it would be beneficial to use a Morawetz estimate, which schematically says that $\Vert r^{-{1 \over 2}} \partial \psi \Vert_{L_t^2 L_x^2}$ is controlled by the energy (the real estimate either requires a slight loss in the power of $r$ or a restriction to a dyadic region). We shall not discuss this estimate here, but will rather use a very easy estimate which suffices for this problem.

\begin{lemma} \label{lem:spacetimeL2}
We have that
\begin{equation}
\begin{aligned}
\int_{\Sigma_s} (1 + t)^{-{\delta \over 2}} \left [(\partial_t \psi)^2 + (\partial_x \psi)^2 + (\partial_y \psi)^2 \right ] d x + \int_0^s \int_{\Sigma_t} \delta (1 + t)^{-1 - \delta} \left [(\partial_t \psi)^2 + (\partial_y \psi)^2 + (\partial_x \psi)^2 \right ] d x d t
\\ = \int_{\Sigma_0} (1 + t)^{-{\delta \over 2}} \left [(\partial_t \psi)^2 + (\partial_x \psi)^2 + (\partial_y \psi)^2 \right ] d x + \int_0^s \int_{\Sigma_t} (\Box \psi) (1 + t)^{-{\delta \over 2}} (\partial_t \psi) d x d t.
\end{aligned}
\end{equation}
\end{lemma}
\begin{proof}
This follows from simply using $(1 + t)^{-{\delta \over 2}} \partial_t \psi$ as a multiplier for the wave equation.
\end{proof}

This estimate is useful because it is sharp in terms of decay up to a power of $t^{{\delta \over 2}}$ (the Morawetz estimate on dyadic regions is sharp). Moreover, it is a spacetime $L^2$ norm, meaning that it allows for interpolation with the scaling vector field $S$.

\subsection{Estimates on geometry} \label{sec:geometry}
We first record several observations which shall be used in Section~\ref{sec:closingpointwise}. These geometric estimates shall be needed when using $\partial_t f$ as a multiplier for the equation for $\psi$. In order to apply Proposition~\ref{prop:decay}, we must show that the error integrals are sufficiently small in terms of $s$ and $\tau$. These estimates will be used to control these quantities. Thus, several of the following estimates depend on the parameters $s$ and $\tau$. We recall that we take data for $f$ supported in a unit disk whose center has coordinates $(s,a,b)$. The data consists of the trace of $f$ and $\partial_t f$ in $\Sigma_s$. Moreover, we recall that $\tau = s - \sqrt{a^2 + b^2}$. Except in Lemma~\ref{lem:dr'drbar}, we shall assume that $b = 0$, meaning that $a + \tau = s$. Lemma~\ref{lem:dr'drbar} is the only result which requires us to simultaneously compare the geometry of light cones adapted to $\psi$, $\phi$, and $f$ at the same time. In order to translate these results to the general case where $b$ is potentially nonzero, we can simply replace instances of $a$ with instances of $s - \tau$.

Many of the following lemmas will deal with a configuration arising from looking at the geometry of the light cone for $\psi$ and the light cone for $f$ in a single $\Sigma_t$ slice. Figure~\ref{fig:psifSigma_t} shows what this looks like. The reader may wish to keep this figure and also Figure~\ref{fig:lightconesaux} in mind throughout this Section, as they both provide valuable intuition for why the results are true.

Throughout this Section, we recall that we can take $s \ge {1 \over \delta_0^{10}}$.

\begin{figure}
    \centering
    \begin{tikzpicture}
    \draw[ultra thick] (0,0) -- (7.4,0);
    \draw[very thick] (0,0) circle (3);
    \draw[very thick] (7.4,0) circle (5);
    \draw (0,0) -- (2.62,1.46);
    \node (r) at (1.21,0.83) {$r$};
    \draw (7.4,0) -- (2.62,1.46);
    \node (r') at (5.06,0.95) {$r'$};
    \node (theta) at (0.75,0.23) {$\theta$};
    \node (theta') at (5.8,0.23) {$\vartheta'$};
    
    \end{tikzpicture}
    \caption{The circles denote the intersection of light cones adapted to $\psi$ and $f$ in some fixed $\Sigma_t$. The line segment connects the centers of the two resulting circles. The length of this line segment is $a = s - \tau$, where we note that we are assuming that $b = 0$. When $b \ne 0$, there is an appropriate modification for this formula. We note that $r = t - u$ and $r' = s - t - u'$. We also recall that $\vartheta = \pi - \theta$ and that $\vartheta' = \pi - \theta'$. For the spacetime picture, we refer the reader to Figure~\ref{fig:lightconesaux}. The two points at which the circles intersect correspond to the two points of intersection of the red ellipse with some $\Sigma_t$. Of course, it is possible for this intersection to be a single point when the circles are tangent to each other, or to be empty when the light cones are no longer intersecting.}
    \label{fig:psifSigma_t}
\end{figure}
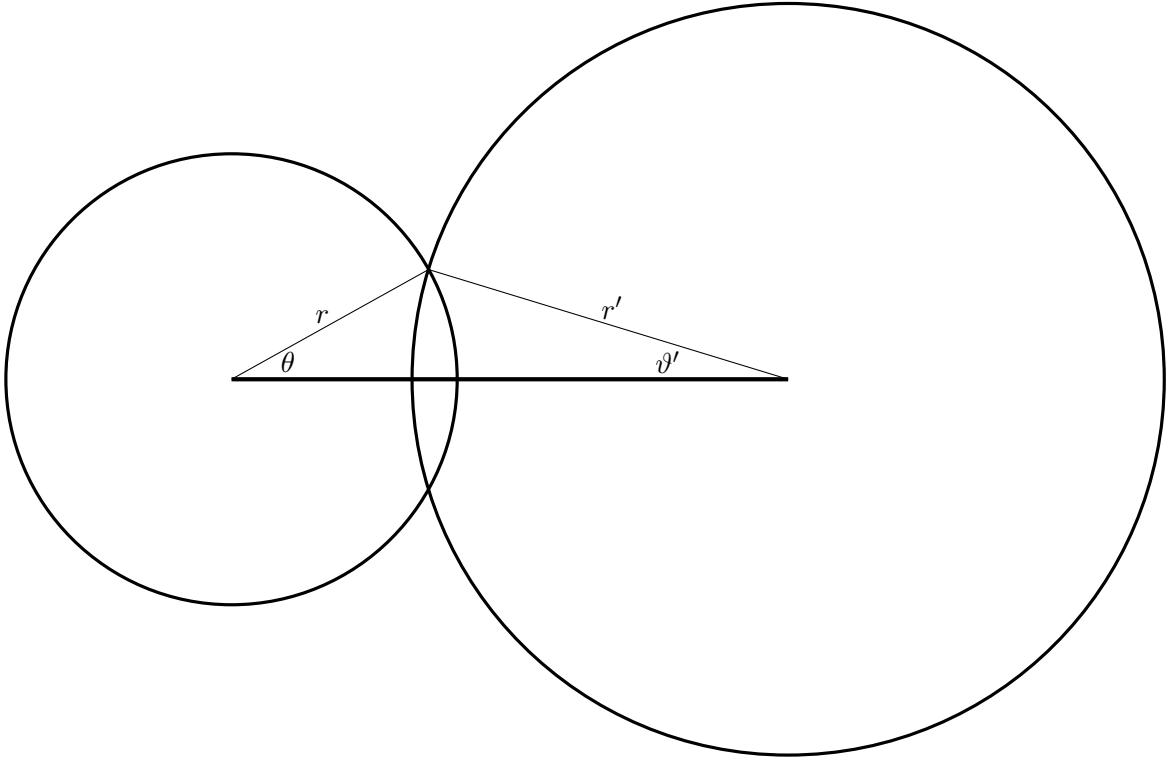

The first lemma says that $r$ is comparable to $t$ and $r'$ is comparable to $s - t$ when we are close to the light cones for $\psi$ and $f$.

\begin{lemma} \label{lem:rtr's-t}
Let $\tau \ge 100$.
\begin{enumerate}
    \item We have that $r + r' - a \le 2 r$, and similarly that $r + r' - a \le 2 r'$.
    \item We have that $r \ge {\tau - u - u' \over 2} \ge 0$, and similarly that $r' \ge {\tau - u - u' \over 2} \ge 0$.
    \item In the region where $u \le \delta_0 \tau$ and $u' \le \delta_0 \tau$, we have that $r' \ge {99 \over 100} (s - t) \ge {\tau \over 10}$ and that $r \ge {99 \over 100} t \ge {\tau \over 10}$.
    \item In the region where $u' \le \delta_0 \tau$ and $r \le t + 1$, we have that $t \ge {\tau \over 10}$. Similarly, in the region where $u \le \delta_0 \tau$ and $r' \le s - t + 1$, we have that $s - t \ge {\tau \over 10}$.
\end{enumerate}
Moreover, in the region where $-1 \le u \le \delta_0 \tau$ and $-1 \le u' \le \delta_0 \tau$, we have that $r' \ge {99 \over 100} (s - t) \ge {\tau \over 10}$ and that $r \ge {99 \over 100} t \ge {\tau \over 10}$.
\end{lemma}
\begin{proof}
We begin by using the triangle inequality to note that
\[
a - r' \le r \le a + r'.
\]
Thus, we have that
\[
0 \le r + r' - a \le 2 r'.
\]
Because the argument is symmetric in $r$ and $r'$, this proves the first assertions.

Now, we note that
\begin{equation} \label{eq:tauuu'rr'ar}
    \begin{aligned}
    \tau - u - u' = \tau - (t - r) - (s - t - r') = \tau - s + r + r' = r + r' - a \le 2 r.
    \end{aligned}
\end{equation}
Similarly, we have that
\begin{equation} \label{eq:tauuu'rr'ar'}
    \begin{aligned}
    \tau - u - u' = \tau - (t - r) - (s - t - r') = \tau - s + r + r' = r + r' - a \le 2 r'.
    \end{aligned}
\end{equation}
This implies the lower bounds on $r$ and $r'$ in terms of $\tau - u - u'$.

We now assume the further restrictions on $u$ and $u'$, which imply that
\[
{\tau \over 2} \le \tau - u - u' = r + r' - a \le 2 r.
\]
From this, it follows that
\[
r \ge {\tau \over 4}.
\]
Now, because $u \le \delta_0 \tau$, we have that
\[
{r \over t} = {r \over r + u} \ge {r \over r + \delta_0 \tau} = {1 \over 1 + {\delta_0 \tau \over r}} \ge {99 \over 100}
\]
for $\delta_0$ sufficiently small. This gives us the desired result for $r$ and $t$, and by symmetry, the same argument gives us the desired result for $r'$ and $s - t$.

The fourth and final assertion follows from the second assertion, noting that $t = r + u$, and using that $u' \le \delta_0 \tau$. This implies the lower bound on $t$, and the lower bound on $s - t$ follows in an analogous way way.


\end{proof}

The next lemma allows us to effectively use the fact that the light cones intersect transversally. It is here that the assumption that $\lambda_1 \ne 1$ and $\lambda_2 \ne 1$ is fundamental. In the region along both light cones, we expect that the Jacobian of the $(r,\overline{r})$ coordinate system in $\Sigma_t$ is comparable to $1$ because the light cones intersect transversally (see Figure~\ref{fig:crosssections}). This will allow us to effectively take advantage of decay in $u$ and $\overline{u}$. The necessary properties of the $(r,\overline{r})$ coordinate system within each $\Sigma_t$ along both light cones is established in the following lemma.

\begin{lemma} \label{lem:rrbar}
In the region where $t \ge 100$, $r \le t + 1$, $\overline{r} \le t + 1$, $|u| \le {1 \over 100} t$, and $|u'| \le {1 \over 100} t$, we have that $(r,\overline{r})$ is a well defined coordinate system on each $\Sigma_t$ with uniformly bounded Jacobian. More precisely, we have that
\begin{equation} \label{eq:rrbar}
    \begin{aligned}
    d x \wedge d y = {r \overline{r} \over (\lambda_2^2 - \lambda_1^2) x y} d r \wedge d \overline{r},
    \end{aligned}
\end{equation}
and we have that
\[
\left |{r \overline{r} \over (\lambda_2^2 - \lambda_1^2) x y} \right | \le C,
\]
where $C$ depends only on $\lambda_1$ and $\lambda_2$, and is uniform given our restrictions on $\lambda_1$ and $\lambda_2$.
\end{lemma}

\begin{proof}
We begin by noting that $r$ and $\overline{r}$ are comparable to $t$ in this region. Thus, we have that $r$ and $\overline{r}$ are comparable to each other. More precisely, we have that
\begin{equation} \label{eq:r/rbar}
    \begin{aligned}
    {9 \over 10} \le {r \over \overline{r}} \le {10 \over 9}
    \end{aligned}
\end{equation}
because the conditions on $u$ and $u'$ imply that
\[
{99 \over 100} \le {r \over t} \le {101 \over 100},
\]
and similarly for ${\overline{r} \over t}$.

We shall now establish the formula \eqref{eq:rrbar}. We shall then show that $x$ and $y$ are also comparable to $r$ and $\overline{r}$ in this region from which the desired result will follow.

We have that
\[
r d r = x d x + y d y,
\]
and we have that
\[
\overline{r} d \overline{r} = \lambda_1^2 x d x + \lambda_2^2 y d y.
\]
Thus, we have that
\[
r \overline{r} d r \wedge d \overline{r} = (\lambda_2^2 - \lambda_1^2) x y d x \wedge d y.
\]
This implies \eqref{eq:rrbar}, as desired.

Now, in order to show that $x$ and $y$ are comparable to $r$ and $\overline{r}$ in this region, we go into polar coordinates $(r,\theta)$. We shall show that $\theta$ is bounded uniformly away from $0$, ${\pi \over 2}$, $\pi$, and ${3 \pi \over 2}$. This will imply the desired result.

Writing $\overline{r}$ in polar coordinates, we have that
\[
\overline{r}^2 = \lambda_1^2 x^2 + \lambda_2^2 y^2 = \lambda_1^2 r^2 + (\lambda_2^2 - \lambda_1^2) y^2 = \lambda_1^2 r^2 + (\lambda_2^2 - \lambda_1^2) r^2 \sin^2 (\theta).
\]
Now, for $\theta$ sufficiently close to $0$ or to $\pi$, we have that
\[
\overline{r}^2 \le {11 \over 10} \lambda_1^2 r^2.
\]
Thus, using the bounds on $\lambda_1^2$, we have that
\[
{\overline{r} \over r} \le {3 \over 4}.
\]
This contradicts the bounds \eqref{eq:r/rbar}. Similarly, for $\theta$ sufficiently close to ${\pi \over 2}$ or ${3 \pi \over 2}$, we have that
\[
{\overline{r} \over r} \ge {3 \over 2},
\]
once again contradicting \eqref{eq:r/rbar}.

\end{proof}

The next lemma writes $x$, $x - a$, and $y$ in terms of $r$ and $r'$ when $\tau$ is small in the region along the light cones for $\psi$ and $f$. The fact that $\tau$ is small and that we are along the light cones for $\psi$ and $f$ gives us smallness in ${1 \over a}$ and ${\tau \over a}$, and it lets us treat $u$ and $u'$ as errors.

\begin{lemma} \label{lem:xyrr'}
Let $\tau \le \delta_0 s$.
\begin{enumerate}
\item We have that
\begin{equation}
    \begin{aligned}
    x = r \left (1 - {\tau - u - u' \over r} + {\tau - u - u' \over a} - {(\tau - u - u')^2 \over 2 r a} \right ).
    \end{aligned}
\end{equation}
\item We have that
\begin{equation}
    \begin{aligned}
    x - a = -r' \left (1 - {\tau - u - u' \over r'} + {\tau - u - u' \over a} - {(\tau - u - u')^2 \over 2 r' a} \right ).
    \end{aligned}
\end{equation}
\item We have that
\begin{equation}
    \begin{aligned}
    y^2 = r^2 - r^2 \left (1 + {\tau - u - u' \over a} - {\tau - u - u' \over r} - {(\tau - u - u')^2 \over 2 r a} \right )^2 \le C r (\tau - u - u').
    \end{aligned}
\end{equation}
\item We have that
\begin{equation}
    \begin{aligned}
    y^2 = (r')^2 - (r')^2 \left (1 + {\tau - u - u' \over a} - {\tau - u - u' \over r'} - {(\tau - u - u')^2 \over 2 r' a} \right )^2 \le C r' (\tau - u - u').
    \end{aligned}
\end{equation}
\end{enumerate}
Moreover, we have that $a \ge (1 - \delta_0) s$.
\end{lemma}
\begin{proof}
We have that $x^2 + y^2 = r^2$, and that $(x - a)^2 + y^2 = (r')^2$. This second expression expands out to $x^2 - 2 x a + a^2 + y^2 = (r')^2$. Now, subtracting this from the first expression gives
\begin{equation} \label{eq:xrr'}
\begin{aligned}
2 x a = r^2 - (r')^2 + a^2.
\end{aligned}
\end{equation}
Now, we write $r = t - u$ and $r' = s - t - u'$. From this, it follows that
\begin{equation} \label{eq:rr'atau}
    \begin{aligned}
    r + r' = s - u - u' = a + \tau - u - u'.
    \end{aligned}
\end{equation}
Solving for $r'$ and plugging this in to \eqref{eq:xrr'} gives us
\begin{equation}
\begin{aligned}
2 x a = r^2 - (a + \tau - r - u - u')^2 + a^2
\\ = r^2 - r^2 + 2 a r + 2 r (\tau - u - u') - a^2 - 2 a (\tau - u - u') - (\tau - u - u')^2 + a^2
\\ = 2 a r + 2 r (\tau - u - u') - 2 a (\tau - u - u') - (\tau - u - u')^2.
\end{aligned}
\end{equation}
Thus, we have that
\begin{equation}
    \begin{aligned}
    x = r + {r (\tau - u - u') \over a} - (\tau - u - u') - {(\tau - u - u')^2 \over 2 a}
    \\ = r \left (1 + {\tau - u - u' \over a} - {\tau - u - u' \over r} - {(\tau - u - u')^2 \over 2 r a} \right ),
    \end{aligned}
\end{equation}
giving us the first equality.

For $x - a$, we note that \eqref{eq:xrr'} implies that
\[
x - a = {r^2 - (r')^2 - a^2 \over 2 a}.
\]
Thus, using \eqref{eq:rr'atau} to solve for $r$ in terms of $r'$, $a$, $\tau$, $u$, and $u'$ gives us that
\[
x - a = {-2 a r' + 2 a (\tau - u - u') - 2 r' (\tau - u - u') + (\tau - u - u')^2 \over 2 a}.
\]
Factoring out $-r'$ gives us that
\[
x - a = -r' \left (1 - {\tau - u - u' \over r'} + {\tau - u - u' \over a} - {(\tau - u - u')^2 \over 2 a r'} \right ),
\]
as desired.

The expressions for $y^2$ now follow simply from using the equation relating $r$, $x$, and $y$, and the equation relating $r'$, $x - a$, and $y$. We have that
\[
y^2 = r^2 - x^2 = r^2 - r^2 \left (1 + {\tau - u - u' \over a} - {\tau - u - u' \over r} - {(\tau - u - u')^2 \over 2 r a} \right )^2,
\]
and we have that
\[
y^2 = (r')^2 - (x - a)^2 = (r')^2 - (r')^2 \left (1 + {\tau - u - u' \over a} - {\tau - u - u' \over r'} - {(\tau - u - u')^2 \over 2 r' a} \right )^2.
\]
Now, we note that
\[
\tau - u - u' = r + r' - a \ge 0
\]
by \eqref{eq:tauuu'rr'ar}. Moreover, by Lemma~\ref{lem:rtr's-t}, we have that
\[
{\tau - u - u' \over r} \le C,
\]
and we have that
\[
{\tau - u - u' \over r'} \le C.
\]
The desired bounds on the expression for $y$ then follow after noting that ${r \over a} \le C$ and ${r' \over a} \le C$ because $\tau \le \delta_0 s$.

The final point that $a \ge (1 - \delta_0) s$ follows immediately because $a = s - \tau$.

\end{proof}

In order to control higher order nonlinearities such as $(\partial_t \psi)^4$, we must analyze the $(r,r')$ coordinate system more carefully. We have the following lemma.

\begin{lemma} \label{lem:annuliu'dec}
We have that
\begin{equation} \label{eq:rr'drdr'}
    \begin{aligned}
    r r' d r \wedge d r' = a y d x \wedge d y.
    \end{aligned}
\end{equation}
Moreover, let $\tau \ge 100$ and let $\tau \le \delta_0 s$. We then fix any $\Sigma_t$ with $0 \le t \le s$, and we take the annular region of thickness $1$ within $\Sigma_t$ given by $b \le u \le b + 1$ subject to the constraint that $-1 \le b \le \delta_0 \tau - 1$. Similarly, we take the annular region of thickness approximately $\delta_0 \tau$ within $\Sigma_t$ given by $-1 \le u' \le \delta_0 \tau$. Then, we have that
\begin{equation} \label{eq:u'decest1}
    \begin{aligned}
    \int_{\Sigma_t} \chi_{b \le u \le b + 1} \chi_{-1 \le u' \le \delta_0 \tau} {1 \over 1 + |u'|^{{1 \over 2}}} d x \le C \min(\sqrt{1 + t} \log(1 + t),\sqrt{1 + s - t}).
    \end{aligned}
\end{equation}

\end{lemma}

Before turning to the proof, we mention that this integral estimate arises from trying to control interactions where $|u|$ and $|u'|$ are small. We are able to assume that we have compact support in $|u|$ because of the strong decay in $|u|$. However, an integration by parts in the scaling vector field will result in only $|u'|^{{1 \over 2}}$ decay in $|u'|$ on the auxiliary multiplier $f$, so we must understand these integrals as opposed to simply understanding the intersections of two annular regions of thickness $1$. Moreover, the conditions we have imposed on the regions may seem very arbitrary. The estimates are used to control quartic interactions in the nonlinearity. The motivation can be summarized by saying that we are most concerned with the region along the light cones for $\psi$ and $f$, and this determines the restrictions above.

\begin{proof}
We have that
\[
r d r = x d x + y d y,
\]
and that
\[
r' d r' = (x - a) d x + y d y.
\]
From this, it immediately follows that
\[
r r' d r \wedge d r' = a y d x \wedge d y,
\]
as desired.

There are now $6$ separate regions to consider. We first consider two cases depending on the size of $t$. When $t \le {3 s \over 4}$, we consider $3$ regions adapted to $r$, and when $t \ge {s \over 4}$, we consider $3$ regions adapted to $r'$. We first consider the regions adapted to $r$ when $t \le {3 s \over 4}$. We also note that the integrand is $0$ unless if ${\tau \over 10} \le t \le s - {\tau \over 10}$ by Lemma~\ref{lem:rtr's-t}, and we note that $a \ge {s \over 2}$ because $\tau \le \delta_0 s$.

The first region is where $|\theta| > \delta_0$ and $|\pi - \theta| > \delta_0$. Figure~\ref{fig:thetavartheta'small} shows schematically what this region looks like.

\begin{figure}
    \centering
    \begin{tikzpicture}
    \draw[very thick] (0,0) circle (4);
    \draw[very thick] (0,0) circle (3.8);
    \draw[dashed] (0,0) -- (4.5,0);
    
    \draw[very thick] (6,0) circle (5);
    \draw[very thick] (6,0) circle (3);
    \draw[dashed] (6,0) -- (1.5,0);
    \draw (0,0) -- (2.85,2.65);
    \draw (6,0) -- (2.85,2.65);
    \node (theta) at (0.5,0.2) {$\theta$};
    \node (vartheta') at (5.3,0.2) {$\vartheta'$};
    \end{tikzpicture}
    \caption{One possible configuration that schematically describes the regions where $t \le {3 s \over 4}$ and $|\vartheta| \ge \delta_0$, and also where $t \ge {s \over 4}$ and $|\theta'| \ge \delta_0$. The thin annulus should be thought of as being the annulus where $b \le u \le b + 1$. The thick annulus should be thought of as being the annulus where $-1 \le u' \le \delta_0 \tau$. The points in question are those which are in both annuli. The integral we are computing involves the function ${1 \over 1 + |u'|^{{1 \over 2}}}$, which decays as we move away from the outer circle of the thick annulus towards the inner circle. In these regions, the important fact we are using is that $y \theta$ and $y \vartheta'$ satisfy good lower bounds (see Lemma~\ref{lem:yvartheta}), meaning that we have good bounds on the Jacobian \eqref{eq:rr'drdr'}. We are thus using that $u \le \delta_0 \tau$ and $u' \le \delta_0 \tau$.}
    \label{fig:thetavartheta'small}
\end{figure}
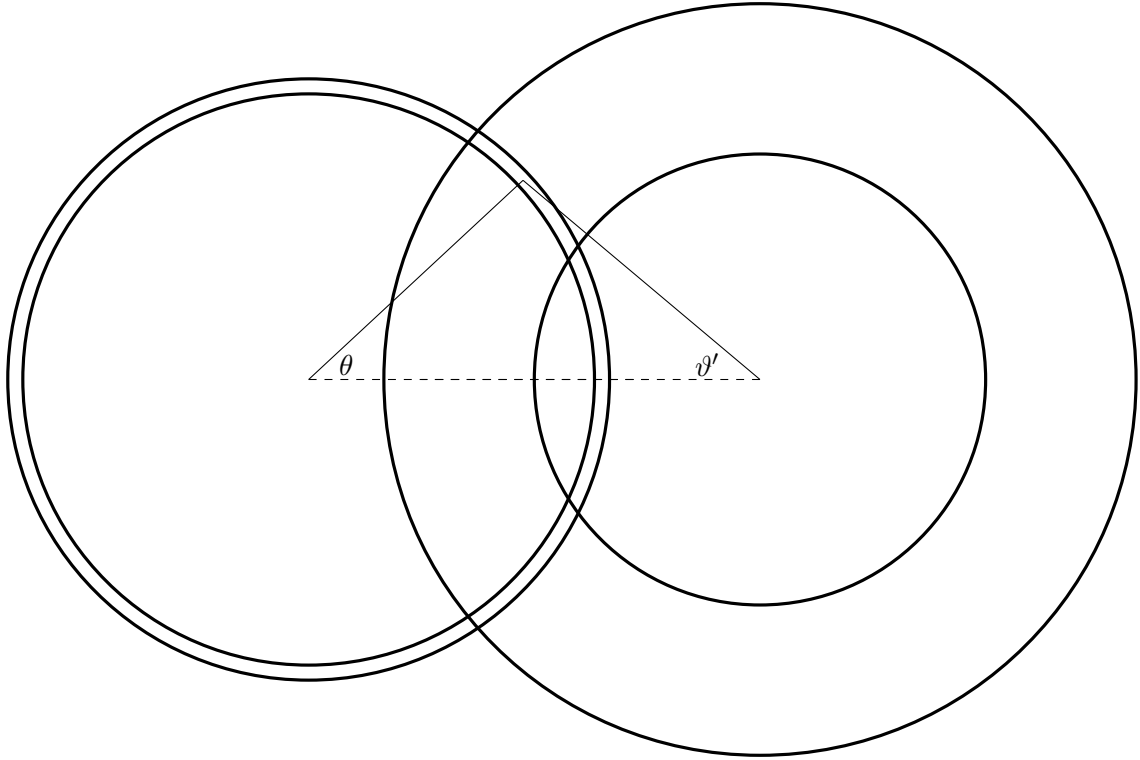

In this region, we have that $y = r \sin(\theta)$ is comparable to $r$, meaning that
\[
{r r' \over a y} \le C.
\]
Thus, the integral is controlled by
\[
\int_{t - b - 1}^{t - b} \int_{s - t - \delta_0 \tau}^{s - t + 1} {1 \over 1 + |u'|^{{1 \over 2}}} d r' d r \le C (1 + \sqrt{\tau}) \le C (1 + \sqrt{t}).
\]
The second region is where $|\theta| \le \delta_0$. This region also looks roughly like the one shown in Figure~\ref{fig:thetavartheta'small}.

We begin by noting that we must have that $|\theta| \ge c {\sqrt{\tau} \over \sqrt{s}}$ because $t \le {3 \over 4} s$ by Lemma~\ref{lem:tvartheta}. We thus also have that $y^2 \ge c r \tau$, meaning that $y \ge c \sqrt{r} \sqrt{\tau}$. Indeed, by Lemma~\ref{lem:yvartheta}, we have that $y \theta \ge {\tau \over 10}$ in this region, meaning that $y^2 \ge c r y \theta \ge c r \tau$, where we are using that $r \theta$ is comparable to $y$ because $|\theta| \le \delta_0$.

Thus, we have that
\[
{r r' \over a y} \le C {\sqrt{r} \over \sqrt{\tau}}.
\]
Thus, the integral is controlled by
\[
\int_{t - b - 1}^{t - b} \int_{s - t - \delta_0 \tau}^{s - t + 1} {\sqrt{r} \over \sqrt{\tau} \left (1 + |u'|^{{1 \over 2}} \right )} d r' d r \le C \sqrt{1 + t}.
\]

We finally consider the region where $|\vartheta| = |\pi - \theta| \le \delta_0$. Figure~\ref{fig:varthetasmall} shows schematically what this region looks like.

\begin{figure}
    \centering
    \begin{tikzpicture}
    \draw[very thick] (-4,0) circle (3);
    \draw[very thick] (-4,0) circle (2.8);
    \draw (-4,0) -- (-6.5,1.5);
    \node (vartheta1) at (-4.7,0.2) {$\vartheta$};
    \draw[dashed] (-4,0) -- (-7,0);
    \draw[very thick,domain=-10:10] plot ({18-25*cos(\x)},{25*sin(\x)});
    \draw[very thick,domain=-10:10] plot ({18-24*cos(\x)},{24*sin(\x)});
    \draw[very thick] (4,0) circle (3);
    \draw[very thick] (4,0) circle (2.8);
    \draw (4,0) -- (2,2.15);
    \node (vartheta2) at (3.55,0.2) {$\vartheta$};
    \draw[dashed] (4,0) -- (1,0);
    \draw[very thick,domain=-10:10] plot ({26.3-25*cos(\x)},{25*sin(\x)});
    \draw[very thick,domain=-10:10] plot ({26.3-24*cos(\x)},{24*sin(\x)});
    \end{tikzpicture}
    \caption{Two configurations giving a rough idea of the regime where $|\vartheta| = |\pi - \theta| \le \delta_0$. The thin annulus should be thought of as being the annulus where $b \le u \le b + 1$. The long outer arc should be thought of as being the outer edge given by $r' = s - t + 1$. The distance between the long outer arc and the long inner arc is comparable to $\delta_0 \tau + 1$ and can be much larger than $1$ in practice, which is why this thickness is much larger than the thickness of the annulus. The points in question are the points in the annulus, to the right of the outer long arc, and to the left of the inner long arc. Of course, the picture is a bit misleading because these points should also have very small $|\vartheta|$ value. The integral we are computing involves the function ${1 \over 1 + |u'|^{{1 \over 2}}}$, which decays as we move away from the outer long arc towards the inner long arc.}
    \label{fig:varthetasmall}
\end{figure}
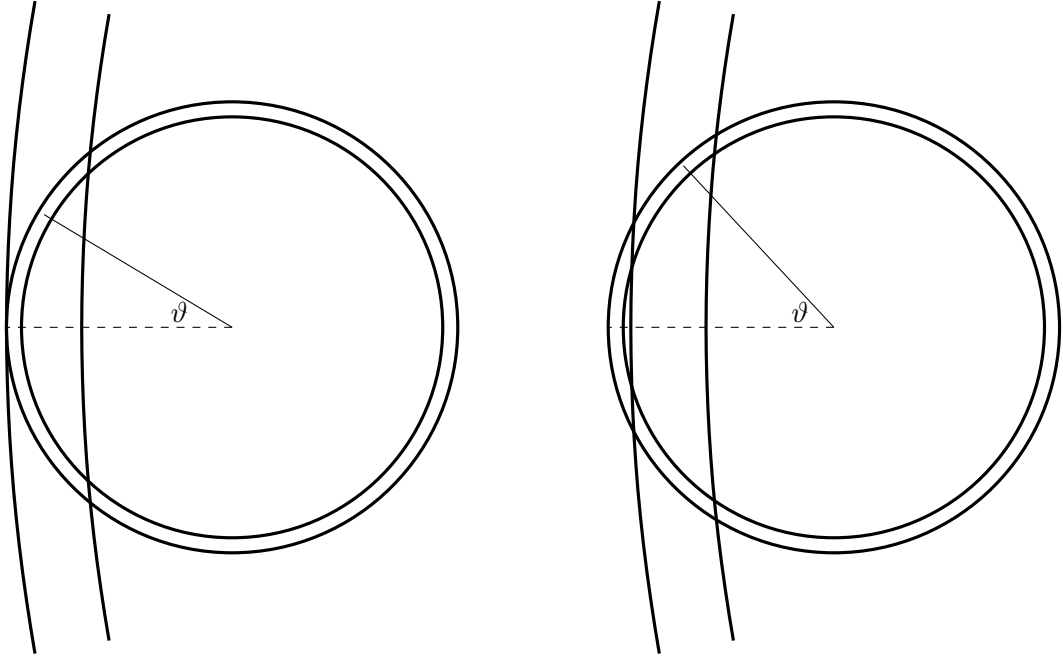

We now go into polar coordinates adapted to $r$. We now note the relation
\[
r' \cos(\vartheta') = a - x = r \cos(\vartheta) + a.
\]
We want to express $r'$ and $\vartheta'$ in terms of $r$ and $\vartheta$. We take a derivative of this expression in $\vartheta$, and we get
\begin{equation} \label{eq:dr'dvartheta1}
    \begin{aligned}
    {\partial r' \over \partial \vartheta} \cos(\vartheta') - r' \sin(\vartheta') {\partial \vartheta' \over \partial \vartheta} = -r \sin(\vartheta).
    \end{aligned}
\end{equation}
Thus, we have that
\[
{\partial r' \over \partial \vartheta} \cos(\vartheta') = -y \left (1 - {\partial \vartheta' \over \partial \vartheta} \right ),
\]
meaning that
\begin{equation} \label{eq:dr'dvartheta2}
    \begin{aligned}
    \left |{\partial r' \over \partial \vartheta} \right | \ge c |y| \left (1 - {\partial \vartheta' \over\partial \vartheta} \right ),
    \end{aligned}
\end{equation}
where we have used the fact that $|\vartheta| \le \delta_0$ implies that $|\vartheta'| \le 10 \delta_0$ by Lemma~\ref{lem:yvartheta} (in fact, we have that $|\vartheta'| \le |\vartheta|$ in this region) along with the fact that $y = r \sin(\vartheta) = r' \sin(\vartheta')$. We shall show that the derivative is approximately equal to $c -y$. We shall then be able to integrate ${\partial u' \over \partial \vartheta}$ in order to control the size of ${1 \over 1 + |u'|^{{1 \over 2}}}$ in the integrand.

We shall now bound ${\partial \vartheta' \over \partial \vartheta}$ in this region. We recall that
\[
\vartheta' = \arctan \left ({y \over a - x} \right ) = \arctan \left ({r \sin(\vartheta) \over r \cos(\vartheta) + a} \right ).
\]
Thus, we have that
\[
{\partial \vartheta' \over \partial \vartheta} = {1 \over 1 + {r^2 \sin^2 (\vartheta) \over (r \cos(\vartheta) + a)^2}} \left ({r \cos(\vartheta) \over r \cos(\vartheta) + a} + {r^2 \sin^2(\vartheta) \over (r \cos(\vartheta) + a)^2} \right )
\]
Now, we have that
\[
r \cos(\vartheta) + a = a \left ({r \over a} \cos(\vartheta) + 1 \right ) \ge {a \over 2},
\]
where we have used the fact that
\begin{equation} \label{eq:rtaucomp}
    \begin{aligned}
    {r \over a} \le C {\tau \over a}
    \end{aligned}
\end{equation}
for some $C$ independent of $\delta_0$ in the region in question, and also that $\tau \le \delta_0 s \le 2 \delta_0 a$.

We shall now pause briefly to show that ${r \over a} \le C {\tau \over a}$ for some $C$ independent of $\delta_0$ indeed holds. Because $|\vartheta| \le \delta_0$, we have that $x = -r \cos(\vartheta) \le -{r \over 2}$. Thus, using Lemma~\ref{lem:xyrr'}, we have that
\[
r \left (1 - {\tau - u - u' \over r} + {\tau - u - u' \over a} - {(\tau - u - u')^2 \over 2 r a} \right ) \le -{r \over 2}.
\]
Moreover, using that
\[
{\tau - u - u' \over a} \le 100 \delta_0,
\]
we have that
\[
-(\tau - u - u') \left (1 - {r \over a} + O(\delta_0) \right ) \le -{3 \over 2} r.
\]
Thus, because ${r \over a} \le 10$, we have that
\[
r \le 100 \tau,
\]
for $\delta_0$ sufficiently small. This allows us to say that, in fact, we have that
\[
{r \over a} \le C {\tau \over a}
\]
for some $C$ independent of $\delta_0$ (for example, we can take $C = 100$), as desired.

Using this, we have that
\begin{equation} \label{eq:dtheta'dthetaRpsi}
    \begin{aligned}
    \left |{\partial \vartheta' \over \partial \vartheta} \right | \le C {r \over a} \le C {\tau \over a}.
    \end{aligned}
\end{equation}
Because the constant $C$ in this expression does not depend on $\delta_0$ and because ${\tau \over a} \le 2 \delta_0$, this expression can be made arbitrarily small by picking $\delta_0$ sufficiently small. From this and from \eqref{eq:dr'dvartheta2}, it follows that


\begin{equation} \label{eq:dr'dvarthetaRpsi}
    \begin{aligned}
    \left |{\partial r' \over \partial \vartheta} \right | \ge c |y| = c r |\sin(\vartheta)| \ge c r |\vartheta|.
    \end{aligned}
\end{equation}
This allows us to argue as follows. We have that the integral in \eqref{eq:u'decest1} is controlled by
\[
\int_{t - b - 1}^{t - b} \int_{-\delta_0}^{\delta_0} \chi_{-1 \le u' \le \delta_0 \tau} {1 \over 1 + |u'|^{{1 \over 2}}} r d \vartheta d r.
\]
Because the integrand is symmetric with respect to reflections over the $x$ axis and because the region of integration is as well, it suffices to control the integral for $\vartheta$ from $0$ to $\delta_0$ instead. Now, for fixed $r$, we pick $\vartheta_0 (r)$ such that $u'$ is minimized under the constraint that $u' \ge -1$ (otherwise, the integrand is $0$). We note then that the integrand is supported where $\vartheta \ge \vartheta_0 (r)$ because we in fact have that ${\partial r' \over \partial \vartheta} < 0$ for $\vartheta$ small and positive by \eqref{eq:dr'dvartheta1}. Thus, the integral we must control is given by
\[
\int_{t - b - 1}^{t - b} \int_{\vartheta_0 (r)}^{\delta_0} \chi_{-1 \le u' \le \delta_0 \tau} {1 \over 1 + |u'|^{{1 \over 2}}} r d \vartheta d r.
\]



Using \eqref{eq:dr'dvarthetaRpsi}, we then have that
\[
u'(r,\vartheta) = u'(r,\vartheta_0) + \int_{\vartheta_0}^{\vartheta} {\partial u' \over \partial \vartheta} d \vartheta_1 \ge -1 + c \int_{\vartheta_0}^{\vartheta} r \vartheta_1 d \vartheta_1 \ge -1 + c r (\vartheta_0 + \vartheta) (\vartheta - \vartheta_0),
\]
where we have used that $u'(r,\vartheta_0) \ge -1$ by the presence of $\chi_{u' \ge -1}$. We now excise a neighborhood of thickness ${10 \over \sqrt{t}}$ around $\vartheta_0$. When $\vartheta - \vartheta_0 \ge {10 \over \sqrt{t}}$, we note that
\begin{equation} \label{eq:u'estRpsi}
    \begin{aligned}
    1 + |u'(r,\vartheta)| \ge c r (\vartheta - \vartheta_0)^2 \ge c.
    \end{aligned}
\end{equation}
Now, we have that
\[
\int_{t - b - 1}^{t - b} \int_{\vartheta_0 (r)}^{\delta_0} {1 \over 1 + |u'|^{{1 \over 2}}} r d \vartheta d r \le \int_{t - b - 1}^{t - b} \int_{\vartheta_0 (r)}^{\vartheta_0 (r) + {10 \over \sqrt{t}}} {1 \over 1 + |u'|^{{1 \over 2}}} r d \vartheta d r + \int_{t - b - 1}^{t - b} \int_{\vartheta_0 (r) + {10 \over \sqrt{t}}}^{\delta_0} {1 \over 1 + |u'|^{{1 \over 2}}} r d \vartheta d r,
\]
where the second integral is taken to be $0$ when $\vartheta_0 (r) + {10 \over \sqrt{t}} \ge \delta_0$.

We now examine each of these integrals. For the first integral, we note that the inner integral is controlled by
\[
C \int_{\vartheta_0 (r)}^{\vartheta_0 (r) + {10 \over \sqrt{t}}} d \vartheta \le {C \over \sqrt{t}},
\]
meaning that the first integral is controlled by $C(1 + \sqrt{t})$, as desired. For the second integral, we use \eqref{eq:u'estRpsi} to show that $|u'|$ is relatively large, giving us that
\[
\int_{\vartheta_0 (r) + {10 \over \sqrt{t}}}^{\delta_0} {1 \over 1 + |u'|^{{1 \over 2}}} d \vartheta \le C \int_{\vartheta_0 (r) + {10 \over \sqrt{t}}}^{\delta_0} {1 \over \sqrt{r} (\vartheta - \vartheta_0 (r))} d \vartheta \le {C (1 + \log(1 + t)) \over \sqrt{r}}.
\]
From this, it follows that the second integral is controlled by $C(1 + \log(1 + t) \sqrt{t})$. This proves the correct upper bound in the region where $t \le {3 \over 4} s$.


We now turn to the regions where $t \ge {s \over 4}$, which are adapted to $r'$. The first region is where $|\theta'| > \delta_0$ and $|\pi - \theta'| > \delta_0$, the second region is where $|\pi - \theta'| \le \delta_0$, and the third region is where $|\theta'| \le \delta_0$. The bounds in the first two regions follow from the same argument to control the corresponding regions adapted to $r$. They both also roughly look like the configuration found in Figure~\ref{fig:thetavartheta'small}.

We thus turn to control the final region where $|\theta'| \le \delta_0$. Figure~\ref{fig:theta'small} gives an idea of what this region looks like.

\begin{figure}
    \centering
    \begin{tikzpicture}
    \draw[very thick] (-4,0) circle (3);
    \draw[very thick] (-4,0) circle (1.5);
    \draw (-4,0) -- (-1.2,0.9);
    \node (theta'1) at (-2.8,0.2) {$\theta'$};
    \draw[dashed] (-4,0) -- (-1,0);
    \draw[very thick,domain=-10:10] plot ({25*cos(\x)-26},{25*sin(\x)});
    \draw[very thick,domain=-10:10] plot ({24.7*cos(\x)-26},{24.7*sin(\x)});
    \draw[very thick] (4,0) circle (3);
    \draw[very thick] (4,0) circle (1.5);
    \draw (4,0) -- (5.7,2.3);
    \node (theta'2) at (4.35,0.2) {$\theta'$};
    \draw[dashed] (4,0) -- (7,0);
    \draw[very thick,domain=-10:10] plot ({25*cos(\x)-19},{25*sin(\x)});
    \draw[very thick,domain=-10:10] plot ({24.7*cos(\x)-19},{24.7*sin(\x)});
    \end{tikzpicture}
    \caption{Two configurations giving a rough idea of the regime where $|\theta'| \le \delta_0$. The two large arcs correspond to the region where $b \le u \le b + 1$. The thick annulus corresponds to $-1 \le u' \le \delta_0 \tau$, and it can in general be rather large. The points in question are the points in the annulus, to the left of the outer long arc, and to the right of the inner long arc. Just like in Figure~\ref{fig:varthetasmall}, the picture is a bit misleading because these points should also have very small $|\theta'|$ value. The integral we are computing involves the function ${1 \over 1 + |u'|^{{1 \over 2}}}$, which decays as we move away from the outer circle in the annulus towards the inner circle.}
    \label{fig:theta'small}
\end{figure}
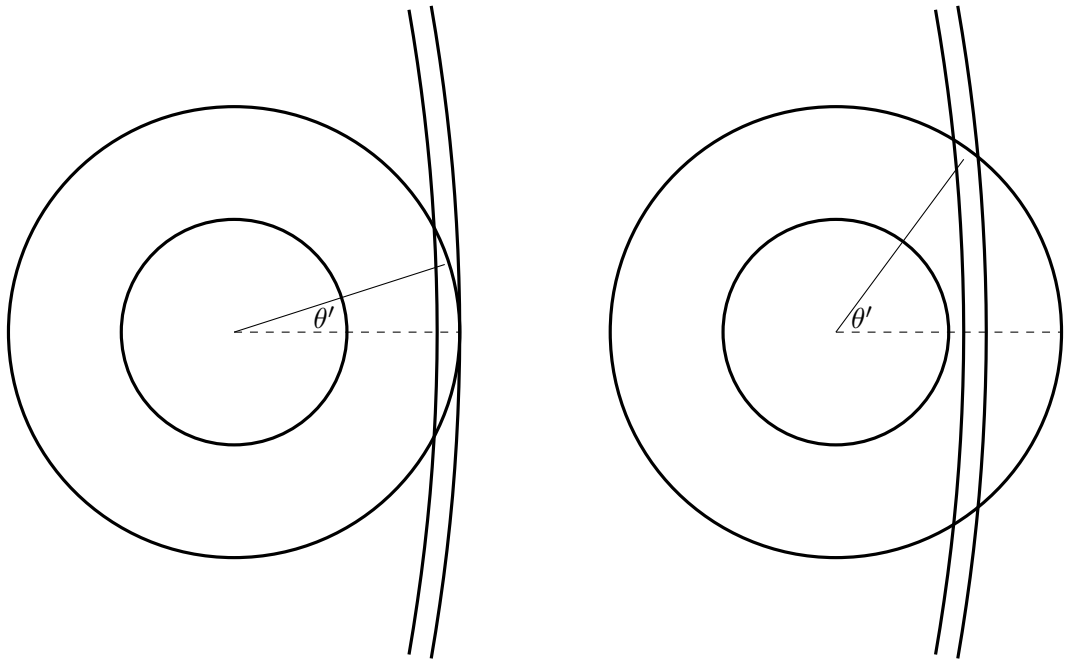


We use the relation
\[
r \cos(\theta) = r' \cos(\theta') + a.
\]
We shall analyze this in the $(r,\theta)$ coordinate system. Taking a derivative with respect to $\theta$ gives us that
\[
-r \sin(\theta) = {\partial r' \over \partial \theta} \cos(\theta') - r' \sin(\theta') {\partial \theta' \over \partial \theta}.
\]
Thus, we have that
\begin{equation} \label{eq:dr'dthetaRf}
    \begin{aligned}
    {\partial r' \over \partial \theta} \cos(\theta) = y \left ({\partial \theta' \over \partial \theta} - 1 \right ).
    \end{aligned}
\end{equation}
Now, we have that
\[
\theta' = \arctan \left ({y \over x - a} \right ) = \arctan \left ({r \sin(\theta) \over r \cos(\theta) - a} \right )
\]
From this, it follows that
\[
{\partial \theta' \over \partial \theta} = {1 \over 1 + {r^2 \sin^2 (\theta) \over (r \cos(\theta) - a)^2}} \left ({r \cos(\theta) \over r \cos(\theta) - a} + {r^2 \sin^2(\theta) \over (r \cos(\theta) - a)^2} \right )
\]

We now use the relations
\[
r \cos(\theta) - a = r' \cos(\theta')
\]
and
\[
r \sin(\theta) = r' \sin(\theta')
\]
in order to conclude that
\[
{\partial \theta' \over \partial \theta} = {1 \over 1 + {\sin^2 (\theta') \over \cos^2 (\theta')}} \left ({r \cos(\theta) \over r' \cos(\theta')} + {\sin^2(\theta') \over \cos^2(\theta')} \right ).
\]
Thus, we have that
\begin{equation} \label{eq:dtheta'dthetaRf}
    \begin{aligned}
    {1 \over 1 + \delta_0} {r \cos(\theta) \over r' \cos(\theta')} \le {\partial \theta' \over \partial \theta} \le (1 + \delta_0) {r \cos(\theta) \over r' \cos(\theta')}.
    \end{aligned}
\end{equation}

Now, we have that the integral in \eqref{eq:u'decest1} is controlled by
\begin{equation} \label{eq:u'decest2}
    \begin{aligned}
    \int_{t - b - 1}^{t - b} \int_{-\theta_1 (r)}^{\theta_1 (r)} {1 \over 1 + |u'|^{{1 \over 2}}} r d \theta d r
    \end{aligned}
\end{equation}
for some $\theta_1 (r)$ (note that we have used that this region is symmetric with respect to reflections over the $x$ axis). These points are determined by the restrictions that $r' \le s - t + 1$ and $-\delta_0 \le \theta' \le \delta_0$. We shall now bound the inner integral
\[
\int_{-\theta_1 (r)}^{\theta_1 (r)} {1 \over 1 + |u'|^{{1 \over 2}}} d \theta
\]
by $C {\sqrt{s - t} \over t}$.


By symmetry, it suffices to control the integral from $0$ to $\theta_1 (r)$ instead. Now, we note that ${\partial \theta' \over \partial \theta}$ is bounded from below by ${1 \over 2} {r \over r'}$ by \eqref{eq:dtheta'dthetaRf} (note that ${r \over r'} \ge {1 \over 10 \delta_0}$ in this region by an argument similar to the one used to show \eqref{eq:rtaucomp}). Thus, using this, \eqref{eq:dr'dthetaRf}, and Lemma~\ref{lem:rtr's-t}, we have that
\[
{\partial r' \over \partial \theta} \ge c {r^2 \over r'} \theta \ge c {t^2 \over s - t} \theta
\]
for some $c$ in this region. Now, at $(r,\theta_1 (r))$, we note that we have that $u' = -1$. Thus, for $0 \le \theta \le \theta_1 (r)$, we have that
\[
|u'(r,\theta)| = \left |u'(r,\theta_1 (r)) - \int_\theta^{\theta_1 (r)} {\partial u' \over \partial \theta} d \theta_0 \right | \ge c {t^2 \over s - t} \theta_1 (r) (\theta_1 (r) - \theta) - 1.
\]
Thus, in this region, we have that
\[
{1 \over 1 + |u'|^{{1 \over 2}}} \le {C \over 1 + {t \sqrt{\theta_1 (r)} \over \sqrt{s - t}} \sqrt{\theta_1 (r) - \theta}}.
\]

Now, we have that
\begin{equation}
    \begin{aligned}
    \int_0^{\theta_1 (r)} {1 \over 1 + {t \sqrt{\theta_1 (r)} \over \sqrt{s - t}} \sqrt{\theta_1 (r) - \theta}} d \theta \le {\sqrt{s - t} \over t \sqrt{\theta_1 (r)}} \int_0^{\theta_1 (r)} {1 \over \sqrt{\theta_1 (r) - \theta}} d \theta = {\sqrt{s - t} \over t \sqrt{\theta_1 (r)}} \int_0^{\theta_1 (r)} {1 \over \sqrt{x}} d x
    \\ \le C {\sqrt{s - t} \over t}.
    \end{aligned}
\end{equation}
Plugging this into \eqref{eq:u'decest2} and integrating in $r$ then gives us the desired estimate in this region.




\end{proof}

Before proceeding, we make a brief remark about the proof of Lemma~\ref{lem:annuliu'dec}. It is interesting to note the asymmetry in the logarithmic factor that occurs in $t$ but not in $s - t$. This can be explained by comparing Figure~\ref{fig:varthetasmall} and Figure~\ref{fig:theta'small}. As can be seen by examining \eqref{eq:rr'drdr'}, the worst contributions from the volume form come when the level sets of $r$ and $r'$ become tangent. Moreover, we note that we have a decaying factor in $u'$. In Figure~\ref{fig:varthetasmall}, $|u'|$ is essentially minimized when these level sets are tangent, while in Figure~\ref{fig:theta'small}, $|u'|$ is maximized when these level sets are tangent. Thus, the region of tangency and worst decay coincide in the region corresponding to Figure~\ref{fig:varthetasmall} (which occurs when $t$ is relatively small), but not in the region corresponding to Figure~\ref{fig:theta'small}.

We shall also need a simpler estimate involving the intersection of an annulus of thickenss one and an annulus of thickness at most $\tau$.

\begin{lemma} \label{lem:annuliarea}
Let $100 \le \tau \le \delta_0 s$. Moreover, in $\Sigma_t$, let $-1 \le b \le 10 \tau$ and let $-1 \le b' \le \tau$. Then, the area of intersection of an annulus of thickness $1$ and radius $r = t - b \le t + 1$ adapted to $r$ and another annulus of thickness $10 \tau$ and radius $r' = s - t - b' \le s - t + 1$ adapted to $r'$ is at most $C (1 + \sqrt{t}) \sqrt{\tau}$. In other words, we have that
\begin{equation}
    \begin{aligned}
    \int_{\Sigma_t} \chi_{b \le u \le b + 1} \chi_{b' \le u' \le b' + 100 \tau} d x \le C (1 + \sqrt{t}) \sqrt{\tau}
    \end{aligned}
\end{equation}
Similarly, the intersection an annulus of thickness $1$ and radius $r = t - b \le t + 1$ and another annulus of thickness $1$ and radius $r' = s - t - b' \le s - t - 1$ has area at most $C \min(1 + \sqrt{t},1 + \sqrt{s - t})$. In other words, we have that
\[
\int_{\Sigma_t} \chi_{b \le u \le b + 1} \chi_{b' \le u' \le b' + 1} d x \le C \min(1 + \sqrt{t},1 + \sqrt{s - t}).
\]

\end{lemma}

\begin{proof}
The proof follows from similar considerations to those in the proof of Lemma~\ref{lem:annuliu'dec}. We shall now describe the geometric considerations.

We consider several different regimes depending on the relative size of $t$, $s - t$, and $\tau$. We begin by noting that the result is clearly true when $t \le {1 \over 100 \delta_0} \tau$ because $C (1 + \sqrt{t}) \sqrt{\tau}$ controls the area of the annulus of radius $t$ and thickness $1$. We thus assume that $t \ge {1 \over 100 \delta_0} \tau$. This implies that $r = t - u \ge {1 \over 100 \delta_0} \tau - 10 \tau - 1 \ge 1000 \tau$, which implies that $|\vartheta| \ge \delta_0$ (see \eqref{eq:rtaucomp}). We then consider two additional cases, one where $|\theta| \ge \delta_0$ and one where $|\theta| \le \delta_0$. When $|\theta| \ge \delta_0$, the Jacobian in the $(r,r')$ coordinate system is well behaved (note that $\tau \le \delta_0 s$), meaning that we may argue as in the proof of Lemma~\ref{lem:annuliu'dec} (see Figure~\ref{fig:thetavartheta'small}). We may thus restrict ourselves to the case where $|\theta| \le \delta_0$. Now, we first assume that $s - t \le {1 \over 100 \delta_0} \tau$. We then have that the diameter of the annulus adapted to $r'$ is controlled by $C \tau$, and we have that $t \ge {9 \over 10} s$ in this region. Thus, using Lemma~\ref{lem:largecirclearc}, we get that the area of intersection is at most $C \tau \le C (1 + \sqrt{t}) \sqrt{\tau}$. Thus, we may assume that $s - t \ge {1 \over 100 \delta_0} \tau$. This forces us to have $|\theta'| \ge \delta_0$ (this is analogous to how having $t \ge {1 \over 100 \delta_0} \tau$ forces $|\vartheta| \ge \delta_0$). We now further consider two cases, one where $|\vartheta'| \ge \delta_0$ and another where $|\vartheta'| \le \delta_0$. In the first case, the Jacobian in $(r,r')$ coordinates is well behaved, and we can argue as in the proof of Lemma~\ref{lem:annuliu'dec} (see Figure~\ref{fig:thetavartheta'small}). We are thus left with the case where $|\theta| \le \delta_0$ and $|\vartheta'| \le \delta_0$. In this region, we consider the natural foliation of the $r$ adapted annulus by circles. If we look at the intersection of any circle in this foliation with the $r'$ annulus, we pick the point with largest $x$ coordinate in this intersection, and we call this value $x_0$. Because the region is symmetric in reflection across the $x$ axis, we can assume that $y \ge 0$. We then note that the smallest possible $x$ coordinate that is still in the intersection is greater than or equal to $x_0 - 100 \tau$. We can simply compute that increasing $\theta$ by $1000 {\sqrt{\tau} \over \sqrt{r}}$ will always decrease the $x$ coordinate by at least $100 \tau$ (note that $r \ge 1000 \tau$ for $\delta_0$ sufficiently small because $t \ge \delta_0^{-1} \tau$). This means that the arc length of the circle contained in this annulus is at most $1000 r {\sqrt{\tau} \over \sqrt{r}} \le C (1 + \sqrt{t}) \sqrt{\tau}$. This gives us the first desired result.

The second result follows from similar considerations. The most substantially different cases are when $|\vartheta| \le \delta_0$ or when $|\theta'| \le \delta_0$, so we shall now describe how to treat these regions. Because the integrals are now symmetric in $r$ and $r'$ after sending $t$ to $s - t$ (this is because both annuli have thickness $1$ now), we can assume that we are considering the case where $|\vartheta| \le \delta_0$. Because the region is once again symmetric with respect to reflections over the $x$ axis, we can restrict ourselves to $\vartheta \ge 0$. We look at the natural foliation of the $r$ annulus by circles. Taking any of these circles, we look at the minimum value of $\vartheta \ge 0$ for which the circle intersects the $r'$ annulus. We can then simply compute that increasing the value of $\vartheta$ by ${C \over \sqrt{r}}$ will case the value of $r'$ to decrease by at least $1$ (see \eqref{eq:dr'dvarthetaRpsi}), meaning that after this change in $\vartheta$, the circle must exit the support of the $r'$ annulus, which has thickness $1$. Thus, the arc length of the circle contained in the $r'$ annulus is controlled by $C (1 + \sqrt{t})$, giving us the desired result.
\end{proof}



We must now get estimates for the $(\theta,u,u')$ coordinate system when $\tau$ is large (meaning that $\tau \ge \delta_0 s$). This will allow us to effectively integrate in this region. The following result is effectively a consequence of the fact that an upward opening cone and a downward opening cone intersect uniformly transversally when $\tau$ is comparable to $s$ (see Figure~\ref{fig:lightconesaux}).

\begin{lemma} \label{lem:thetauu'}
We have that
\[
-r d \theta \wedge d u \wedge d u' = (1 + \sin(\theta) \sin(\theta') + \cos(\theta) \cos(\theta')) d t \wedge d x \wedge d y.
\]
Moreover, let $\tau \ge \delta_0 s$. Then, in the region where $-1 \le u \le \delta_0 \tau$ and $-1 \le u' \le \delta_0 \tau$, we have that
\[
|1 + \sin(\theta) \sin(\theta') + \cos(\theta) \cos(\theta')| \ge c > 0,
\]
meaning that
\[
d t \wedge d x \wedge d y = J_{\theta,u,u'} r d \theta \wedge d u \wedge d u'
\]
with $|J_{\theta,u,u'}| \le C$.
\end{lemma}
\begin{proof}
We have that
\[
-r d \theta = {y \over r} d x - {x \over r} d y.
\]
Moreover, we have that $d \theta \wedge d r \wedge d r'$ vanishes because the one forms involved are all intrinsic to the $\Sigma_t$ hypersurfaces, which are two dimensional. Thus, we have that
\[
-r d \theta \wedge d u \wedge d u' = r d \theta \wedge d t \wedge (d r + d r') = d t \wedge (-r d \theta) \wedge (d r + d r').
\]
We shall now estimate $-r d \theta \wedge (d r + d r')$.

We have that
\[
d r + d r' = \left ({x \over r} d x + {x - a \over r'} \right ) d x + \left ({y \over r} + {y \over r'} \right ) d y.
\]
Thus, we have that
\begin{equation}
    \begin{aligned}
    -r d \theta \wedge (d r + d r') = \left [{y \over r} \left ({y \over r} + {y \over r'} \right ) + {x \over r} \left ({x \over r} + {x - a \over r'} \right ) \right ] d x \wedge d y
    \\ = \left (1 + \sin(\theta) \sin(\theta') + \cos(\theta) \cos(\theta') \right ) d x \wedge d y,
    \end{aligned}
\end{equation}
as desired.

We shall now prove a lower bound on
\[
\left |1 + \sin(\theta) \sin(\theta') + \cos(\theta) \cos(\theta') \right |,
\]
from which the desired result will follow.

We begin by noting that $\sin(\theta)$ and $\sin(\theta')$ always have the same sign, meaning that $\sin(\theta) \sin(\theta')$ is always positive. Thus, when $x \le 0$ or $x \ge a$, we have a lower bound of $1$ because $\cos(\theta)$ and $\cos(\theta')$ also have the same sign, meaning that $\cos(\theta) \cos(\theta')$ is also positive. We may thus restrict ourselves to the region where $0 \le x \le a$.

Now, because $\tau \ge \delta_0 s$ and because of the bounds on $|u|$ and $|u'|$, we note that $|\theta| \ge c$ and $|\theta'| \ge c$ when $0 \le x \le x - a$ for some constant $c$ depending only on $\delta_0$ by Lemma~\ref{lem:xyrr'}. Indeed, when $0 \le x \le a$, $y^2$ is minimized when $x = 0$ and $x = a$. The only other critical point in this region corresponds to a maximum for $y^2$ (this behavior can be seen geometrically by examining the red ellipse in Figure~\ref{fig:lightconesaux}, and it can be explicitly computed by calculating the intersection of the cones of constant $u$ and constant $u'$ subject to the constraint that $u \le \delta_0 \tau$ and $u' \le \delta_0 \tau$). Now, at $x = 0$, we still have that $r$ is comparable to $s$ by Lemma~\ref{lem:rtr's-t}. This means that $y^2$ must still be comparable to $s^2$ at this point (we can in fact explicitly compute that it is equal to $(\tau - u - u')^2 \left (1 - {\tau - u - u' \over s} + {(\tau - u - u')^2 \over s^2} \right ) \ge c s^2$ for $\tau \ge \delta_0 s$ when $x = 0$ and $x = a$, see Lemma~\ref{lem:xyrr'}). By symmetry, $y^2$ must also be comparable to $s^2$ at $x = a$, meaning that $y^2$ is uniformly comparable to $s^2$ whenever $0 \le x \le a$, proving the lower bound on $|\theta|$ and $|\theta'|$ because $y^2 = r^2 \sin^2 (\theta) = (r')^2 \sin^2 (\theta')$. Thus, we have that $1 + \sin(\theta) \sin(\theta') \ge 1 + c^2$, and this implies the desired result.
\end{proof}


Finally, we must get estimates for the $(\overline{r},r')$ coordinate system. We shall only need these estimates when $\tau \le \delta_0 s$, $r' \ge 10$, $t \ge (1 - 20 \delta_0) s$, and $|\vartheta'| \le \delta_0$. As a consequence of these assumptions, we note that the level sets of $r'$ and $r$ approximately coincide. Thus, the fact that this is true morally follows from Lemma~\ref{lem:rrbar}.


\begin{lemma} \label{lem:dr'drbar}
Let $\tau \le \delta_0 s$. Then, in the region where $t \ge (1 - 20 \delta_0) s$, $|\vartheta'| \le \delta_0$, $-1 \le u$, $-1 \le \overline{u} \le \delta_0 \tau$, and $-1 \le u' \le \delta_0 \tau$, we have that
\[
d x \wedge d y = J_{\overline{r},r'} d \overline{r} \wedge d r',
\]
where
\[
|J_{\overline{r},r'}| \le C.
\]
\end{lemma}
\begin{proof}
For this proof only, we shall denote by $C$ a constant which is independent of $\delta_0$. Moreover, because we need to analyze the geometry of all three light cones in question at the same time, we shall not assume $b = 0$. Thus, the downward opening light cone has its tip at the point $(s,a,b)$ (see Section~\ref{sec:coordinates}).

We recall that $\Theta$ denotes the $\theta$ coordinate of the points where $r' = 0$. We also recall that $\Theta_i$ for $i = 1, 2, 3, 4$ denote the four lines of intersection of the cones $u = 0$ and $\overline{u} = 0$, and we recall the analogous five quantities in terms of $\overline{\theta}$. We shall first show that $\Theta$ must be very close to one of the $\Theta_i$.

Without loss of generality, we can assume that $0 \le \Theta \le {\pi \over 2}$. We shall show that $\Theta$ must be very close to $\Theta_1$ in this case, and this will give us the desired estimates. The arguments for the other cases follow in an analogous way by symmetry.

The first observation is that, in the region in question, we have that $\Theta_1 - C \delta_0 \le \theta \le \Theta_1 + C \delta_0$ where $C$ depends only on $\lambda_1$ and $\lambda_2$. Indeed, outside of this region, we have that $\overline{u}$ will no longer be between $-1$ and $\delta_0 \tau$. This is because ${\partial \overline{u} \over \partial \theta} (r,\theta) \ge c r$ for some $c$ depending only on $\lambda_1$ and $\lambda_2$. To see this, we recall that $\partial_\theta = x \partial_y - y \partial_x$, and we note that
\[
2 \overline{r} {\partial \overline{r} \over \partial \theta} (r,\theta) = (x \partial_y - y \partial_x) (\lambda_1^2 x^2 + \lambda_2^2 y^2) = 2 (\lambda_2^2 - \lambda_1^2) x y.
\]
The desired result then follows because $x$ and $y$ are both comparable to $r$ and $\overline{r}$ in this region. A similar argument works to show that $\overline{\Theta}_1 - C \delta_0 \le \overline{\theta} \le \overline{\Theta_1} + C \delta_0$ because we have that $u \le 100 \delta_0 s$ in the region in question.


The second observation is that we must have that $|\Theta - \Theta_1| \le C \delta_0$ where $C$ depends only on $\lambda_1$ and $\lambda_2$ in order for the region to be nonempty. Indeed, given that we are restricting to the region where $t \ge (1 - 20 \delta_0) s$, we note that the region in question requires that $r' \le 20 \delta_0 s + 1$. This restricts us to the region where $|\theta - \Theta| \le 100 \delta_0$ (see also Lemma~\ref{lem:largecirclearc}). Thus, because $\theta$ is already restricted to being close to $\Theta_1$, we thus have that $|\Theta - \Theta_1| \le C \delta_0$, as desired. See Figure~\ref{fig:vartheta'smalltlarge} to see a geometric description of what is going on. These considerations will allow us to compare $\overline{\theta}$ and $\theta'$ with $\overline{\Theta}_1$ and $\Theta_1$, respectively.

\begin{figure}
    \centering
    \begin{tikzpicture}
    \draw[very thick] (0,0) circle (4.5);
    \draw[very thick] (0,0) circle (3.8);
    \draw[very thick] (0,0) ellipse (6 and 3);
    \draw[very thick] (0,0) ellipse (5.7 and 2.7);
    \draw[red,fill=red] (3.1,2.6) circle (0.15);
    \end{tikzpicture}
    \caption{A figure showing why $\theta$ must be close to $\Theta$ and $\Theta_1$, and why $\overline{\theta}$ must be close to $\overline{\Theta}_1$, where close means within some constant multiplied by $\delta_0$. The red disk represents $r' \le 20 \delta_0 s + 1$, which is the largest $r'$ can be when $t \ge (1 - 20 \delta_0) s$. The line $\theta = \Theta$ is the line going from the center of the circular annulus to the center of the red disk. The circular annulus is $(1 - 50 \delta_0) s \le r \le s + 1$ and the elliptical annulus is $(1 - 30 \delta_0) s \le \overline{r} \le s + 1$. We are interested in points contained in the intersection of both annuli and the red disk. Because $\tau \le \delta_0 s$, we can see that we are restricted to a small range of $\theta$ and $\overline{\theta}$ values. Of course, we additionally have the restriction that $|\vartheta'| \le \delta_0$, but we have not drawn this restriction.}
    \label{fig:vartheta'smalltlarge}
\end{figure}
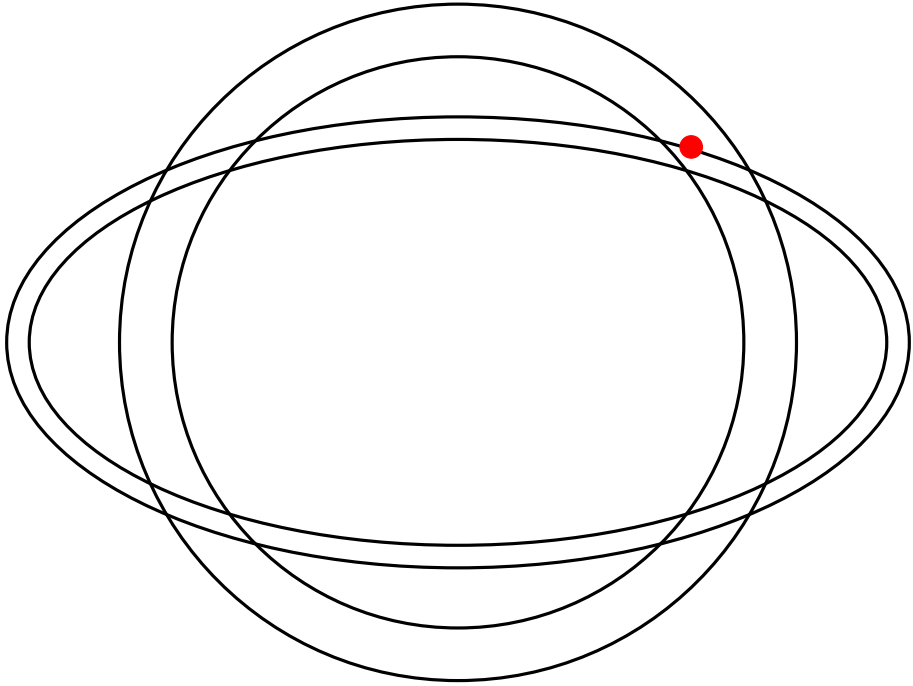

Because we have that $|\vartheta'|
\le \delta_0$ and that $|\overline{\theta} - \overline{\Theta}_1| \le C \delta_0$, it is natural to get expressions for $d \overline{r}$ and $d r'$ in terms of $\overline{\theta}$ and $\vartheta'$. We have that
\[
\overline{r} d \overline{r} = \lambda_1^2 x d x + \lambda_2^2 y d y.
\]
Using that $x = \lambda_1^{-1} \overline{x} = \lambda_1^{-1} \overline{r} \cos(\overline{\theta})$ and that $y = \lambda_2^{-1} \overline{y} = \lambda_2^{-1} \overline{r} \sin(\overline{\theta})$, this gives us that
\[
d \overline{r} = \lambda_1 \cos(\overline{\theta}) d x + \lambda_2 \sin(\overline{\theta}) d y.
\]
Now, for $d r'$, we first note that
\[
\theta' = \arctan \left ({y - b \over x - a} \right ) - \Theta.
\]
Moreover, we have that
\[
d r' = {x - a \over r'} d x + {y - b \over r'} d y.
\]
Using that ${x - a \over r'} = \cos \left (\arctan \left ({y - b \over x - a} \right ) \right )$ and that ${y - b \over r'} = \sin \left (\arctan \left ({y - b \over x - a} \right ) \right )$ then gives us that
\[
d r' = \cos(\theta' + \Theta) d x + \sin(\theta' + \Theta) d y.
\]
Thus, we have that
\[
d \overline{r} \wedge d r' = (\lambda_1 \cos(\overline{\theta}) \sin(\theta' + \Theta) - \lambda_2 \sin(\overline{\theta}) \cos(\theta' + \Theta)) d x \wedge d y.
\]

Now, we have that
\[
\vartheta' = \pi - \theta' = \pi - \arctan \left ({y - b \over x - a} \right ) - \Theta.
\]
From the bounds on $\vartheta'$, it follows that
\[
\left |\arctan \left ({y - b \over x - a} \right ) - \pi - \Theta \right | \le \delta_0.
\]
Thus, we have that $\theta' + \Theta = \pi + \Theta +  O(\delta_0) = \pi + \Theta_1 + O(\delta_0)$.

From this, it follows that
\[
d \overline{r} \wedge d r' = (\lambda_2 \sin(\overline{\Theta}_1) \cos(\Theta_1) - \lambda_1 \cos(\overline{\Theta_1}) \sin(\Theta_1) + O(\delta_0)) d x \wedge d y.
\]

This implies the desired result for $\delta_0$ sufficiently small because the quantity is nonzero when $\delta_0 = 0$. Indeed, dividing by $\cos(\Theta_1)$ and $\cos(\overline{\Theta}_1)$ (which we can do because $\Theta_1$ and $\overline{\Theta}_1$ are uniformly bounded away from half integer multiples of $\pi$ given our restrictions on $\lambda_1$ and $\lambda_2$), we get that
\begin{equation}
    \begin{aligned}
    {1 \over \cos(\Theta_1) \cos(\overline{\Theta}_1)} d \overline{r} \wedge d r' = (\lambda_2 \tan(\overline{\Theta}_1) - \lambda_1 \tan(\Theta_1) + O(\delta_0)) d x \wedge d y
    \\ = \left ({\lambda_2 \overline{y} \over \overline{x}} - {\lambda_1 y \over x} + O(\delta_0) \right ) d x \wedge d y = \left ({(\lambda_2^2 - \lambda_1^2) \sqrt{1 - \lambda_1^2} \over \lambda_1 \sqrt{\lambda_2^2 - 1}} + O(\delta_0) \right ) d x \wedge d y,
    \end{aligned}
\end{equation}
which is bounded away from $0$ uniformly by the restrictions we have placed on $\lambda_1$ and $\lambda_2$ for $\delta_0$ sufficiently small. We note that we have simply solved for ${y \over x}$ in the equations
\[
t^2 = x^2 + y^2 = \lambda_1^2 x^2 + \lambda_2^2 y^2,
\]
which gives the different values of $\tan(\theta)$ and $\tan(\overline{\theta})$ corresponding to the four lines of intersection of the two forward opening light cones. This gives us the desired result.
\end{proof}

The final geometric result we shall need concerns the intersections of circles with sets having small diameter relative to the radius of the circle.

\begin{lemma} \label{lem:largecirclearc}
Let $S$ be a circle of radius $r$, and let $D$ be a region with diameter $d \le {1 \over 2} r$. Then, we have that the length of the portion of $S$ contained in $D$ is controlled by $C d$.
\end{lemma}
\begin{proof}
This follows immediately by going into polar coordinates adapted to $S$ and noting that the set of all $\theta$ for which $S \cap D$ is nonempty must be contained in an interval of length at most $C {d \over r}$. Indeed, if this were not true, it would contradict that $D$ has diameter $d$ because the arc length along $S$ is comparable to the Euclidean distance at scales which are small compared to the radius of the circle.
\end{proof}

We note that an analogous statement holds for ellipses.



\subsection{Scaling vector field geometry} \label{sec:SGeometry}
Throughout this Section, we recall that we can take $s \ge {1 \over \delta_0^{10}}$. Moreover, we shall once again be assuming without loss of generality that $b = 0$, meaning that $a = s - \tau$ (see Section~\ref{sec:coordinates}).

We shall use the above ideas in addition with the scaling vector field $S = t \partial_t + r \partial_r = t \partial_t + x \partial_x + y \partial_y$. We first note that $S$ satisfies good commutation properties with both $\Box$ and $\Box'$.

\begin{lemma}
The vector field $S$ has that $[\Box,S] = 2 \Box$. Similarly, we have that $[\Box',S] = 2 \Box'$.
\end{lemma}
\begin{proof}
This follows easily from a computation.
\end{proof}
Because of this, $S$ can effectively be used as a commutator for this system of equations.

In order to use $S$ in the way described in Section~\ref{sec:anisotropicdescription}, we must find a way to write general vectors in the frame consisting of $S$ and $\overline{\partial}_f$, the good derivatives of $f$. This is the content of the following lemma, which writes $\underline{L}'$ in terms of $S$ and $\overline{\partial}_f$. We introduce the expression
\begin{equation}
    \begin{aligned}
    \gamma = -{1 \over t + r' + {a (x - a) \over r'}}.
    \end{aligned}
\end{equation}
This expression arises as a coefficient in frame decompositions, and it shall be very important in the following.

\begin{lemma} \label{lem:du'S}
Let $\tau \ge 100$. In the region where $u' \le \delta_0 \tau$, we have that
\begin{equation} \label{eq:du'S}
    \begin{aligned}
    \underline{L}' h = 2 \gamma S(h) + \gamma \left (t - r' - {a (x - a) \over r'} \right ) L' (h) + \gamma {2 a y \over (r')^2} \partial_{\theta'} (h)
    \end{aligned}
\end{equation}
where $h$ is any smooth function.
\end{lemma}
\begin{proof}
We have that
\[
S = t \partial_t + r \partial_r = -{1 \over 2} t (L' + \underline{L}') + r' \partial_{r'} + a \partial_x = -{1 \over 2} t (L' + \underline{L}') + {1 \over 2} r' (L' - \underline{L}') + a \partial_x.
\]
Moreover, we have that
\[
\partial_x = {\partial r' \over \partial x} \partial_{r'} + {\partial \theta' \over \partial x} \partial_{\theta'} = {x - a \over r'} \partial_{r'} - {y \over (r')^2} \partial_{\theta'} = {x - a \over 2 r'} (L' - \underline{L}') - {y \over (r')^2} \partial_{\theta'}.
\]
Thus, we have that
\[
2 S = -t (L' + \underline{L}') + r' (L' - \underline{L}') + {a (x - a) \over r'} (L' - \underline{L}') - {2 a y \over (r')^2} \partial_{\theta'}.
\]
Thus, we have that
\[
-\left (t + r' + {a (x - a) \over r'} \right ) \underline{L}' = 2 S + \left (t - r' - {a (x - a) \over r'} \right ) L' + {2 a y \over (r')^2} \partial_{\theta'}.
\]
If we recall that $\gamma^{-1} = -\left (t + r' + {a (x - a) \over r'} \right ) = -{t r' + (r')^2 + a (x - a) \over r'}$, we thus have that
\[
\underline{L}' = 2 \gamma S + \gamma \left (t - r' - {a (x - a) \over r'} \right ) L' + \gamma {2 a y \over (r')^2} \partial_{\theta'},
\]
as desired.
\end{proof}

Terms not involving $S$ and having the factor of $\gamma$ will be improved because $\gamma$ will be shown to be small. Moreover, they contain good derivatives of $f$, meaning that they gain a weight of ${1 \over r'}$, which is comparable to ${1 \over s - t}$ in the region in question. The term with $S$ is dangerous because $S f$ can be large for the auxiliary multiplier. However, we can integrate it by parts and use the fact that we can commute the equation with $S$. The gain is then that the factor of $\gamma$ remains. However, when integrating by parts, it is possible for $S$ to hit $\gamma$, meaning that we must control $S(\gamma)$ as well.

Geometrically, the fact that the function $\gamma$ is well behaved follows from the fact that $S$ everywhere pierces the downward opening light cone of $f$, and from the fact that it has length comparable to $t + r$. The angle between $S$ and the normal to the tangent space of the downward pointing light cone for $f$ may, however, be close to ${\pi \over 2}$, meaning that $S$ may be almost tangent to the light cone for $f$. The reader may wish to keep Figure~\ref{fig:lightconesaux} in order to get an idea of the interplay of geometry between the scaling vector field and the light cones in question.

The smallness of $\gamma$ and the control of $S(\gamma)$ are the content of the following lemma.

\begin{lemma} \label{lem:gammaSgammabound}
Let $\tau \ge 100$, let $-1 \le u' \le \delta_0 \tau$, and let $|u| \le 100 \tau$.
\begin{enumerate}
    \item We have that
    \begin{equation} \label{eq:gammaest}
    \begin{aligned}
    |\gamma| \le {C \over \tau}.
    \end{aligned}
    \end{equation}
    \item In the region where $|\vartheta'| = |\pi - \theta'| \ge \delta_0 / 2$ and $r' \ge 10$, we have that
    \begin{equation} \label{eq:Sgammaesttheta'} 
    \begin{aligned}
    |S(\gamma)| \le {C \over r'}.
    \end{aligned}
    \end{equation}
    \item In the region where in fact $-1 \le u \le \delta_0 \tau$, we have that
    \begin{equation} \label{eq:Sgammaestusmall}
    \begin{aligned}
    |S(\gamma)| \le {C \over \tau}.
    \end{aligned}
    \end{equation}
    \item In the region where $t \le (1 - 10 \delta_0) s$, we have that
    \begin{equation} \label{eq:Sgammaesttlarge}
        \begin{aligned}
        |S(\gamma)| \le {C \over \tau}.
        \end{aligned}
    \end{equation}
    \item In the region where $\tau \ge \delta_0 s$ and $r' \ge 10$, we have that
    \begin{equation} \label{eq:Sgammaesttaularge}
        \begin{aligned}
        |S(\gamma)| \le {C \over r'}.
        \end{aligned}
    \end{equation}
\end{enumerate}
\end{lemma}
Before proceeding to the proof, let us try to briefly motivate these different regions. The choice of the regions is motivated by the fact that we have to take advantage of different things in different areas. We can then decompose the error integrals in Section~\ref{sec:closingpointwise} into sums of integrals over these localized regions, control each one individually, and then add together the (finitely many) resulting bounds.

When $\tau \ge \delta_0 s$, the scaling vector field uniformly pierces the light cone for $f$, so we expect to have good estimates in this region. When $s - t$ is comparable to $s$, we can use the fact that all of the expressions will end up having good powers of $r'$. Because $r'$ is then comparable to $s - t$ along the light cone for $f$, we can hope to take advantage of this. When $u \le \delta_0 \tau$ and $u' \le \delta_0 \tau$, we can control how tangent the two circles from Figure~\ref{fig:psifSigma_t} are. This is important for controlling $\gamma$ and $S(\gamma)$ because the coordinate $y$ naturally appears, and preventing the circles from becoming too tangent is the same thing as getting a lower bound on the $y$ coordinate, which is the vertical distance in Figure~\ref{fig:psifSigma_t}. When $|\vartheta'| \ge \delta_0$, we can get strong estimates for $\gamma$, allowing us to control everything.
\begin{proof}
We shall do several computations and specialize as necessary. The case where $\tau \le \delta_0 s$ will require the most work.


We have that
\begin{equation} \label{eq:gammaest1}
\begin{aligned}
\gamma = -{r' \over t r' + (r')^2 + a (x - a)} = -{1 \over t + r' + a \cos(\theta')} = -{1 \over t + s - t - u' + (s - \tau) \cos(\theta')}
\\ = -{1 \over s - u' + (s - \tau) \cos(\theta')}.
\end{aligned}
\end{equation}
Now, we note that
\[
s - u' + (s - \tau) \cos(\theta') \ge {\tau \over 2}
\]
because $u' \le \delta_0 \tau$. This proves the first assertion \eqref{eq:gammaest} about $\gamma$.

Before proceeding, we shall make a few remarks. When $\tau \ge \delta_0 s$, the denominator in \eqref{eq:gammaest1} is very large, and this makes the proof easier. Moreover, we note that it will be natural to consider $\vartheta' = \pi - \theta'$. This is because the most delicate region is where $\theta'$ is close to $\pi$. As we shall see below, we shall break up the expressions into regions where $\vartheta'$ is of different sizes depending on $\tau$ and $s$.

We note that
\begin{equation} \label{eq:Sgammaest1}
\begin{aligned}
S (\gamma) = -S \left ({1 \over t + r' + a \cos(\theta')} \right ) = {t \over (t + r' + a \cos(\theta'))^2} + {r \partial_r r' + a r \partial_r \cos(\theta') \over (t + r' + a \cos(\theta'))^2}
\\ = (t + r \partial_r r' + a r \partial_r \cos(\theta')) \gamma^2.
\end{aligned}
\end{equation}

The way in which these terms can be controlled is geometrically motivated. We begin by noting that $\partial_r r' \approx -1$ on a large set when $\tau$ is small relative to $s$ (when $\tau \ge \delta_0 s$, the proof is more immediate, as we shall see). Thus, we think that $\gamma^2 (t + r \partial_r r') \approx \gamma^2 u$. Moreover $y$ should be small relative to other quantities on a large set by Lemma~\ref{lem:xyrr'}. Thus, we expect to be able to control expressions with $y$ in the numerator and other quantities in the denominator. More precisely, we note that $t + r \partial_r r' + a r \partial_r \cos(\theta')$ vanishes in the limit as $\tau \rightarrow 0$. Of course, when $\tau \rightarrow 0$, we have that $\gamma$ blows up because the downward opening and upward opening light cones become tangent (see Figure~\ref{fig:lightconesaux}). However, this gives us hope that this term can be controlled when $\tau$ is small relative to $s$, and that \eqref{eq:Sgammaest1} is small. We now make these ideas precise.

We shall first calculate the quantity $t + r \partial_r r' + a r \partial_r \cos(\theta')$. We have that
\[
r \partial_r r' = x \partial_x (r') + y \partial_y (r') = {x (x - a) \over r'} + {y^2 \over r'},
\]
and that
\[
r \partial_r \cos{\theta'} = x \partial_x \left ({x - a \over r'} \right ) + y \partial_y \left ({x - a \over r'} \right ) = {x \over r'} - {x (x - a)^2 \over (r')^3} - {y^2 (x - a) \over (r')^3}.
\]
Thus, we have that
\begin{equation} \label{eq:Sr'theta'}
    \begin{aligned}
    S(\gamma) = \gamma^2 \left (t + r \partial_r r' + a r \partial_r \cos(\theta') \right ) = \gamma^2 \left [t + {x (x - a) \over r'} + {y^2 \over r'} + {a x \over r'} - {a x (x - a)^2 \over (r')^3} - {a y^2 (x - a) \over (r')^3} \right ].
    \end{aligned}
\end{equation}
This already suffices to control $S(\gamma)$ when $\tau \ge \delta_0 s$ (we note that $\tau \le s$ always holds). Indeed, in this region, we have that $\gamma^2 \le {C \over s^2}$. Thus, every term in \eqref{eq:Sr'theta'} multiplied by $\gamma^2$ is bounded by ${C \over r'}$. This implies \eqref{eq:Sgammaesttaularge}, and it also implies \eqref{eq:Sgammaesttheta'}, \eqref{eq:Sgammaestusmall}, and \eqref{eq:Sgammaesttlarge} in the region where $\tau \ge \delta_0 s$ (this requires comparing $r'$ with $\tau$ using Lemma~\ref{lem:rtr's-t} for \eqref{eq:Sgammaestusmall} and Lemma~\ref{lem:tvartheta} for \eqref{eq:Sgammaesttlarge}). Thus, for the remainder of the proof, we shall assume that $\tau \le \delta_0 s$.

Now, before proceeding, we shall explicitly group the terms in \eqref{eq:Sr'theta'} into three different categories because they will be dealt with differently. Following the discussion at the beginning of the proof, we shall bound the terms in \eqref{eq:Sr'theta'} with powers of $y$ in the numerator directly. We shall treat $t + {x (x - a) \over r'}$ as a single term and extract a power of $u$ plus other terms which are well behaved. Then, we shall treat ${a x \over r'} - {a x (x - a)^2 \over (r')^3}$ as a single term, noting that it is equal to ${a x y^2 \over (r')^3}$. Thus, the first group consists of
\[
\gamma^2 \left [t + {x (x - a) \over r'} \right ].
\]
The second group consists of
\[
\gamma^2 \left [{a x \over r'} - {a x (x - a)^2 \over (r')^3} \right ].
\]
The third group consists of
\[
\gamma^2 \left [{y^2 \over r'} - {a y^2 (x - a) \over (r')^3} \right ].
\]
The second group of terms is the hardest to control.

Following the discussion at the beginning of the proof, we shall bound the terms in \eqref{eq:Sr'theta'} with powers of $y$ in the numerator directly. We shall treat $t + {x (x - a) \over r'}$ as a single term and extract a power of $u$ plus other terms which are well behaved. Then, we shall treat ${a x \over r'} - {a x (x - a)^2 \over (r')^3}$ as a single term, noting that it is equal to ${a x y^2 \over (r')^3}$.

Using Lemma~\ref{lem:xyrr'}, we have that
\begin{equation}
    \begin{aligned}
    {x (x - a) \over r'} = -r \left (1 - {\tau - u - u' \over r} + {\tau - u - u' \over a} - {(\tau - u - u')^2 \over 2 r a} \right )
    \\ \times \left (1 - {\tau - u - u' \over r'} + {\tau - u - u' \over a} - {(\tau - u - u')^2 \over 2 r' a} \right ).
    \end{aligned}
\end{equation}
Expanding and grouping this with $t$ in \eqref{eq:Sr'theta'} as described above, we are able to extract a factor of $u = t - r$. Indeed, using \eqref{eq:tauuu'rr'ar}, we have that 
\[
t + {x (x - a) \over r'} = t - r + {r \over r'} (\tau - u - u') + O(\tau - u - u'),
\]
where we have used that $0 \le \tau - u - u' \le C r'$ (see Lemma~\ref{lem:rtr's-t}). We have that $\gamma^2 (t - r) \le C \gamma$ because $|u| \le C \tau$. Moreover, we have that $\gamma^2 O(\tau - u - u') \le C \gamma$ because $\gamma \le {C \over \tau}$. The only remaining term is ${r \over r'} (\tau - u - u')$. Because $a \ge {s \over 2}$ when $\tau \le \delta_0 s$, we have that


\[
{r \over r'} (\tau - u - u') \le C {a r (\tau - u - u') \over (r')^2},
\]
and this term (see \eqref{eq:Sgammaworstterm}) is controlled below. Altogether, we have so far shown that
\begin{equation} \label{eq:Sgammat-r}
    \begin{aligned}
    \gamma^2 \left [t + {x (x - a) \over r'} \right ] \le C \gamma + C \gamma^2 {a r (\tau - u - u') \over (r')^2}.
    \end{aligned}
\end{equation}
This reduces controlling the first group of terms to controlling \eqref{eq:Sgammaworstterm}, which comes up in the second group of terms.


Now, by Lemma~\ref{lem:xyrr'}, we have that
\[
y^2 = (r')^2 - (r')^2 \left (1 - {\tau - u - u' \over r'} + {\tau - u - u' \over a} - {(\tau - u - u')^2 \over 2 r' a} \right )^2 \le C r'(\tau - u - u')
\]
because we are now assuming that $\tau \le \delta_0 s$. Thus, we have that
\begin{equation} \label{eq:Sgammaworstterm}
    \begin{aligned}
    \left |{a x \over r'} - {a x (x - a)^2 \over (r')^3} \right | = \left |{a x ((r')^2 - (x - a)^2) \over (r')^3} \right | = \left |{a x y^2 \over (r')^3} \right | \le C {a r (\tau - u - u') \over (r')^2}.
    \end{aligned}
\end{equation}
Multiplying by $\gamma^2$, we have that
\[
\left |\gamma^2 \left ({a x \over r'} - {a x (x - a)^2 \over (r')^3} \right ) \right | \le C {a r (\tau - u - u') \over (r')^2 (t + r' + a \cos(\theta'))^2}.
\]


We now have everything in place to establish \eqref{eq:Sgammaesttheta'}. We first recall that
\[
t + r' + a \cos(\theta') = t + s - t - u' + (s - \tau) \cos(\theta') = s - u' + (s - \tau) \cos(\theta').
\]
Now, for $|\vartheta'| = |\theta' - \pi| \ge \delta_0 / 2$, we have that
\begin{equation} \label{eq:gammavarthetalarge}
    \begin{aligned}
    |s - u' + (s - \tau) \cos(\theta')| \ge c s^2
    \end{aligned}
\end{equation}
for some $c > 0$, resulting in the bound
\[
|\gamma| \le {C \over s},
\]
and also the bound
\[
\left |\gamma^2 \left ({a x \over r'} - {a x (x - a)^2 \over (r')^3} \right ) \right | \le C {a r (\tau - u - u') \over s^2 (r')^2} \le {C \over r'},
\]
where we have used Lemma~\ref{lem:rtr's-t} to say that ${|\tau - u - u'| \over r'} \le C$. This controls the second group of terms in this region. We also note that
\[
\gamma^2 {y^2 \over r'} \le {C \over r'},
\]
and that
\[
\gamma^2 {a y^2 (x - a) \over (r')^3} \le {C \over s},
\]
which controls the third group of terms in this region. Along with \eqref{eq:Sgammat-r}, these considerations prove \eqref{eq:Sgammaesttheta'}, and they also establish \eqref{eq:Sgammaestusmall} in the region where $|\vartheta'| \ge \delta_0 / 2$ after using Lemma~\ref{lem:rtr's-t} to say that $r' \ge c \tau$ when $u \le \delta_0 \tau$, which is an assumption for the estimate \eqref{eq:Sgammaestusmall}.

All that remains is to establish \eqref{eq:Sgammaesttlarge} and \eqref{eq:Sgammaestusmall}. More precisely, we must only consider the case of $\tau \le \delta_0 s$ in the region where $|\vartheta'| \le \delta_0$ by what we have done above. We shall focus first on the most difficult term given by
\[
\gamma^2 {a r (\tau - u - u') \over (r')^2}.
\]
In the grouping we defined earlier, this corresponds to controlling the second group of terms. The third group of terms will be controlled after.




Now, when $|\vartheta'| \le \delta_0$, we compute the Taylor expansion around $\vartheta' = \pi - \theta' = 0$ for $t + r' + a \cos(\theta') = s - u' + (s - \tau) \cos(\theta')$, giving us that
\begin{equation}
    \begin{aligned}
    s - u' + (s - \tau) \cos(\theta') = s - u' - (s - \tau) + {1 \over 2} (s - \tau) (\vartheta')^2 + (s - \tau) O((\vartheta')^4)
    \\ = (\tau - u') + {1 \over 2} (s - \tau) (\vartheta')^2 + (s - \tau) O((\vartheta')^4).
    \end{aligned}
\end{equation}
Thus, it suffices to control
\begin{equation} \label{eq:Sgammaworstterm1}
    \begin{aligned}
    {a r \tau \over (r')^2 \left (\tau - u' + {1 \over 2} (s - \tau) (\vartheta')^2 + (s - \tau) O((\vartheta')^4) \right )^2}.
    \end{aligned}
\end{equation}
Let us first assume that $t \le (1 - 10 \delta_0 s)$. Then, the quantity \eqref{eq:Sgammaworstterm1} is controlled by
\[
C {a r \tau \over (r')^2 \tau^2}.
\]
By Lemma~\ref{lem:tvartheta}, we have that $r' \ge c s$ in this region, meaning that this quantity is controlled by
\[
{C \over \tau}.
\]
This controls this term in the region where $t \le (1 - 10 \delta_0 s)$.

More generally, there is a transition of dominant terms in the denominator between $\tau - u'$ and ${1 \over 2} (s - \tau) (\vartheta')^2$ when they are equal. This occurs when $(\vartheta')^2 = 2{\tau - u' \over s - \tau}$. Now, we have that $\tau - u' \ge (1 - \delta_0) \tau$ because $u' \le \delta_0 \tau$. Moreover, we are assuming that $\tau \le \delta_0 s$. This means that $(\vartheta')^2$ must be comparable to ${\tau \over s}$. Thus, it is natural to consider two regions, one where $(\vartheta')^2 \le C {\tau \over s}$ and the other where $(\vartheta')^2 > C {\tau \over s}$ for some constant $C$. However, for technical reasons, it is easier to work instead with the region where $t \le c s$ and $t \ge c s$ for some $0 < c < 1$. We prove in Lemma~\ref{lem:tvartheta} that $t \ge (1 - 10 \delta_0) s$, $u' \le \delta_0 \tau$, and $u \le \delta_0 \tau$ implies that $|\vartheta'| > c {\sqrt{\tau} \over \sqrt{s}}$ when $\tau \le \delta_0 s$. We note that we can now freely assume all of these conditions because all that remains to be shown for the second group of terms is that they are controlled by ${C \over \tau}$ with these assumptions.


Now, in this region, we have that the quantity \eqref{eq:Sgammaworstterm1} is controlled by
\begin{equation} \label{eq:gammabound1}
    \begin{aligned}
    C {a r \tau \over (r')^2 s^2 (\vartheta')^4}.
    \end{aligned}
\end{equation}
Moreover, we note that $r' \vartheta'$ is comparable to $y$ (note that we are assuming that $|\vartheta'| \le \delta_0$). Then, because $y \vartheta'$ is comparable to $\tau$ by Lemma~\ref{lem:yvartheta} below, we have that
\[
{a r \tau \over (r')^2 s^2 (\vartheta')^4} \le C {\tau \over (r')^2 (\vartheta')^4} \le C {\tau \over y^2 (\vartheta')^2} \le {C \over \tau}.
\]
This completes the control of the second group of terms in every region.

We now proceed to the third and final group of terms. We have that
\[
\gamma^2 {a y^2 (x - a) \over (r')^3} \le C \gamma^2 {a \tau \over r'},
\]
where we have used Lemma~\ref{lem:xyrr'} to control $y^2$. Now, for $t \le (1 - 10 \delta_0) s$, we have that $r'$ is comparable to $s$ (see Lemma~\ref{lem:tvartheta}), meaning that this quantity is controlled by
\[
C \gamma^2 \tau \le {C \over \tau}.
\]
When $t \ge (1 - 10 \delta_0) s$, we have that ${r \over r'} \ge 1$ because $\tau \le \delta_0 s$, meaning that
\[
C \gamma^2 {a \tau \over r'} \le C \gamma^2 {a r \tau \over (r')^2},
\]
and the bound follows in the same way as for the term \eqref{eq:Sgammaworstterm} above.

Finally, for the last term in this third group of terms, we use Lemma~\ref{lem:xyrr'} to control $y^2$, giving us that
\[
\gamma^2 {y^2 \over r'} \le C \gamma^2 \tau \le {C \over \tau},
\]
as desired.
\end{proof}

We now control the other coefficients from the expansion in \eqref{eq:du'S}, which will allow us to finally control the error integrals. Now, because the $L'$ derivative gains a weight of ${1 \over r'}$ when applied to $f$ (see \eqref{eq:assumeddecay2}), it is natural to consider instead the expression
\[
2 \gamma S(h) + \gamma {1 \over r'} \left (t - r' - {a (x - a) \over r'} \right ) r' L' (h) + \gamma {2 a y \over (r')^2} \partial_{\theta'} (h).
\]
Note that $\partial_{\theta'}$ is already correctly normalized. We then have the following lemma which controls the coefficients of $L'$ and $\partial_{\theta'}$.

\begin{lemma} \label{lem:du'dv'dtheta'}
Let $u \le 100 \tau$, and let $-1 \le u' \le \delta_0 \tau$.

\begin{enumerate}
    \item In the region where $u \le \delta_0 \tau$, we have that
    \[
    \left |\gamma {1 \over r'} \left (t - r' - {a (x - a) \over r'} \right ) \right | \le {C \over \tau},
    \]
    and we have that
    \[
    \left |\gamma {2 a y \over (r')^2} \right | \le {C \over \tau}.
    \]
    \item In the region where $r' \ge 10$ and $|\vartheta'| = |\pi - \theta| \ge \delta_0 / 2$, we have that
    \[
    \left |\gamma {1 \over r'} \left (t - r' - {a (x - a) \over r'} \right ) \right | \le {C \over r'},
    \]
    and we have that
    \[
    \left |\gamma {2 a y \over (r')^2} \right | \le {C \over r'}.
    \]
    \item In the region where $t \le (1 - 10 \delta_0) s$ and $r' \ge 10$, we have that
    \[
    \left |\gamma {1 \over r'} \left (t - r' - {a (x - a) \over r'} \right ) \right | \le {C \over \tau} + {C \over r'},
    \]
    and we have that
    \[
    \left |\gamma {2 a y \over (r')^2} \right | \le {C \over \tau} + {C \over r'}.
    \]
    \item In the case where $\tau \ge \delta_0 s$, we have that
    \[
    \left |\gamma {1 \over r'} \left (t - r' - {a (x - a) \over r'} \right ) \right | \le {C \over r'},
    \]
    and we have that
    \[
    \left |\gamma {2 a y \over (r')^2} \right | \le {C \over r'}.
    \]
\end{enumerate}


\end{lemma}
\begin{proof}
We first consider the relative size of $\tau$ and $s$. When $\tau \ge \delta_0 s$, the result is immediate because $\gamma$ is controlled by ${C \over s}$ by Lemma~\ref{lem:gammaSgammabound}. We may thus assume that $\tau \le \delta_0 s$. Moreover, when $t \le (1 - 10 \delta_0) s$, we note that $r'$ is comparable to $s$ (see Lemma~\ref{lem:tvartheta}). Thus, in this region, we can bound the coefficients by $C \gamma$, as desired.

We now consider the region where $t \ge (1 - 20 \delta_0) s$ and $|\vartheta'| \le \delta_0$. With these conditions, we are only interested in the case where we additionally have that $u \le \delta_0 \tau$. We note that the coefficient of $r' L'$ is bounded by
\[
C \gamma + {C \gamma a \over r'}.
\]
The first of these terms is bounded by ${C \over \tau}$ by Lemma~\ref{lem:gammaSgammabound}. For the second term, we note that
\[
\gamma {a \over r'} \le C {a \over r' s (\vartheta')^2} \le C {1 \over y \vartheta'},
\]
where we have bounded $\gamma$ in the same way as in \eqref{eq:gammabound1}. This quantity is then controlled by ${C \over \tau}$ by Lemma~\ref{lem:yvartheta}. Then, we have that
\[
\left |\gamma {2 a y \over (r')^2} \right | \le C \gamma {a \over r'},
\]
meaning that the same argument works here.

Finally, in the region where $|\vartheta'| \ge \delta_0 / 2$ and $r' \ge 10$, we note that $\gamma$ is comparable to ${1 \over s}$. Thus, we have that
\[
\gamma {a \over r'} \le {C \over r'},
\]
giving us the desired result. We note that we are using Lemma~\ref{lem:rtr's-t} to say that $r' \ge c \tau$ in the region where $u \le \delta_0 \tau$.


\end{proof}

\begin{lemma} \label{lem:tvartheta}
Let $\tau \le \delta_0 s$.
\begin{enumerate}
    \item In the region where $t \ge {1 \over 4} s$, $u' \le \delta_0 \tau$, and $u \le \delta_0 \tau$, we have that $|\vartheta'| \ge c {\sqrt{\tau} \over \sqrt{s}}$, and we have that $|\vartheta'| \ge {1 \over 10} |\theta|$. An analogous statement holds when $t \le {3 \over 4} s$ with the roles of $\vartheta'$ and $\theta$ interchanged.
    \item In the region where $t \le (1 - 10 \delta_0) s$ and $u' \le \delta_0 \tau$, we have that $r' \ge c s$. Similarly, in the region where $t \ge 10 \delta_0 s$ and $u \le \delta_0 \tau$, we have that $r \ge c s$.
\end{enumerate}
\end{lemma}
\begin{proof}
We first assume that $t \le (1 - 10 \delta_0) s$ and that $u' \le \delta_0 \tau$. We then have that $s - t \ge 10 \delta_0 s$. Moreover, because $u' \le \delta_0 \tau$ and $\tau \le s$, we have that $u' \le \delta_0 s$. Thus, we have that
\[
r' = s - t - u' \ge 9 \delta_0 s.
\]
An analogous argument shows a similar lower bound for $r$ instead of $r'$. This proves the second assertion.

We shall now prove the first assertion. We shall perform the argument to find a lower bound for $|\vartheta'|$. Finding a lower bound for $|\theta'|$ follows in an analogous way.

We assume that $t \ge {1 \over 4} s$, $u' \le \delta_0 \tau$, and that $u \le \delta_0 \tau$. We then have that $r = t - u \ge {1 \over 5} s$, and we have that
\begin{equation} \label{eq:r'boundtlarge}
    \begin{aligned}
    r' \le s - t + 1 \le {3 \over 4} s + 1.
    \end{aligned}
\end{equation}
Then, because $y = r \sin(\theta) = r' \sin(\theta') = r' \sin(\vartheta')$, we must have that $|\vartheta'| \ge {1 \over 10} |\theta|$. This can be seen by consulting Figure~\ref{fig:psifSigma_t}. More precisely, we have that
\[
{\sin(\vartheta') \over \sin(\theta)} = {r \over r'} \ge {1 \over 5}.
\]
Moreover, without loss of generality, we have that $0 \le \theta \le {\pi \over 2}$ (we cannot have ${\pi \over 2} \le \theta \le \pi$ because we have that $|x - a| \le r' \le {3 \over 4} s + 1$). We thus have that ${2 \theta \over \pi} \le \sin(\theta) \le \theta$. Without loss of generality, we may also assume that $0 \le \vartheta' \le \pi$. In the region where $\vartheta' \ge {\pi \over 2}$, the inequality is then obvious, so we may also assume that $0 \le \vartheta' \le {\pi \over 2}$. In this region, we once again have that ${2 \vartheta' \over \pi} \le \sin(\vartheta') \le \vartheta'$. Thus, we have that
\[
\vartheta' \ge \sin(\vartheta') \ge {1 \over 5} \sin(\theta) \ge {1 \over 5} {2 \over \pi} \theta \ge {1 \over 8} \theta.
\]
This implies that $|\vartheta'| \ge {1 \over 8} |\theta|$ in this region, giving us the desired result.

We finally turn to showing that $|\vartheta'| \ge c {\sqrt{\tau} \over \sqrt{s}}$. By symmetry, we can take $\vartheta' > 0$, and without loss of generality, we may assume that $\vartheta' \le \delta_0$ (otherwise, there is nothing left to show). By Lemma~\ref{lem:xyrr'}, we have that
\begin{equation}
    \begin{aligned}
    a - x = r' \left (1 - {\tau - u - u' \over r'} + {\tau - u - u' \over a} - {(\tau - u - u')^2 \over 2 r' a} \right )
    \\ = r' - r' {\tau - u - u' \over r'} \left (1 - {r' \over a} + {\tau - u - u' \over 2 a} \right ).
    \end{aligned}
\end{equation}
Because $r' \le {3 \over 4} s + 1$ and because $a \ge (1 - \delta_0 s)$ (see Lemma~\ref{lem:xyrr'}), we have that
\begin{equation} \label{eq:factolowerbound}
    \begin{aligned}
    1 - {r' \over a} + {\tau - u - u' \over 2 a} \ge {1 \over 10}.
    \end{aligned}
\end{equation}
Then, because $\vartheta' \le \delta_0$, we have that $y \le r' \delta_0$, meaning that the triangle inequality gives us that
\[
r' - r' \delta_0 \le a - x = r' - r' {\tau - u - u' \over r'} \left (1 - {r' \over a} + {\tau - u - u' \over 2 a} \right ) \le r' - r' {\tau - u - u' \over 10 r'}.
\]
From this, it follows that
\[
{\tau - u - u' \over r'} \le 10 \delta_0.
\]

Now, we have that
\[
y^2 = (r')^2 - (r')^2 \left (1 + {\tau - u - u' \over a} - {\tau - u - u' \over r'} - {(\tau - u - u')^2 \over 2 r' a} \right )^2.
\]
Because $\vartheta' \le \delta_0$ and $\sin(\vartheta')$ is comparable to $\vartheta'$ for $\vartheta'$ small, we can write that
\[
(r')^2 (\vartheta')^2 + (r')^2 O((\vartheta')^4) = (r')^2 - (r')^2 \left (1 + {\tau - u - u' \over a} - {\tau - u - u' \over r'} - {(\tau - u - u')^2 \over 2 r' a} \right )^2.
\]
Thus, we have that
\begin{equation}
    \begin{aligned}
    (r')^2 (\vartheta')^2 + (r')^2 O((\vartheta')^4) = (r')^2 - (r')^2 \left [1 + {\tau - u - u' \over r'} \left ({r' \over a} - 1 - {\tau - u - u' \over 2 a} \right ) \right ]^2
    \\ = (r')^2 - (r')^2
    \\ \times \left [1 + {2 (\tau - u - u') \over r'} \left ({r' \over a} - 1 - {\tau - u - u' \over 2 a} \right ) + {(\tau - u - u')^2 \over (r')^2} \left ({r' \over a} - 1 - {(\tau - u - u')^2 \over 2 r' a} \right )^2 \right ]
    \\ = 2 r' (\tau - u - u') \left (1 - {r' \over a} + {\tau - u - u' \over 2 a} + O(\delta_0) \right ).
    \end{aligned}
\end{equation}
Thus, using \eqref{eq:factolowerbound}, we have that
\[
(r')^2 (\vartheta')^2 + (r')^2 O((\vartheta')^4) \ge {1 \over 20} r' (\tau - u - u').
\]
From this, it follows that
\[
(\vartheta')^2 \ge c {\tau - u - u' \over r'},
\]
giving us the desired result.

\end{proof}

\begin{lemma} \label{lem:yvartheta}
Let $\tau \le \delta_0 s$. In the region where $t \ge {1 \over 4} s$, $|\vartheta'| \le \delta_0$, $u \le \delta_0 \tau$ and $u' \le \delta_0 \tau$,  we have that $y \vartheta' \ge {1 \over 10} \tau$. Similarly, in the region where $t \le {3 \over 4} s$, $|\theta| \le \delta_0$, $u \le \delta_0 \tau$, and $u' \le \delta_0 \tau$, we have that $y \theta \ge {1 \over 10} \tau$.
\end{lemma}
\begin{proof}
We shall prove the result for $\vartheta'$. The result for $\theta$ follows in an analogous way.

We begin by noting that we must have that $0 \le x \le a = s - \tau$ in this region. Thus, we have that $r \cos(\theta) + r' \cos(\vartheta') = a = s - \tau$, where we recall that $\vartheta' = \pi - \theta'$ (see Figure~\ref{fig:psifSigma_t} to see that this is true because $r \cos(\theta) + r' \cos(\theta')$ gives the horizontal distance between the center adapted to $r$ and the center adapted to $r'$, and the line segment connecting these two points has length $s - \tau$). By symmetry, we can consider the region where $y \ge 0$. We have that $y = r \sin(\theta) = r' \sin(\theta') = r' \sin(\vartheta')$. Moreover, by Lemma~\ref{lem:tvartheta}, we have that $0 \le \theta \le 10 \vartheta'$. Thus, we have that
\[
r \sqrt{1 - \sin^2 (\theta)} + r' \sqrt{1 - \sin^2 (\vartheta')} = s - \tau.
\]
Taylor expanding, this gives us that
\[
r + r' - {1 \over 2} \sin^2 (\theta) r - {1 \over 2} \sin^2 (\vartheta') r' + r O(\sin^4 (\theta)) + r' O(\sin^4 (\vartheta')) = s - \tau.
\]
Thus, we have that
\[
-{1 \over 2} y \sin(\theta) - {1 \over 2} y \sin(\vartheta') + y O(\theta^3 + (\vartheta')^3) = s - \tau - r - r'.
\]
We note that $r = t - u \ge t - \delta_0 \tau$ and $r' = s - t - u' \ge s - t - \delta_0 \tau$. Thus, we have that
\[
s - \tau - r - r' \le s - \tau - (t - u) - (s - t - u') = -\tau + u + u' \le -{9 \tau \over 10},
\]
meaning that
\[
-y \sin(\vartheta') + y O((\vartheta')^3) \le -{\tau \over 7}.
\]
From this, it follows that $y \vartheta' \ge {\tau \over 10}$ when $\delta_0$ is sufficiently small, as desired.
 
\end{proof}


 


With these estimates, we can now control all of the nonlinear error integrals effectively.

\subsection{Bootstrap assumptions} \label{sec:bootstrap}
We now list the bootstrap assumptions. For $N$ a sufficiently large integer which will be chosen to make the proof work, we let $T$ be the largest real number such that 
\begin{equation}
\begin{aligned}
\sum_{|\alpha| \le N} \Vert (1 + t)^{-\epsilon} \partial \Gamma^\alpha \psi \Vert_{L_t^\infty ([0,T]) L_x^2} \le \epsilon^{{3 \over 4}},
\\ \sum_{|\alpha| \le N} \Vert (1 + t)^{-{1 \over 2} - {\delta \over 4}} \partial \Gamma^\alpha \psi \Vert_{L_t^2 ([0,T]) L_x^2} \le \epsilon^{{3 \over 4}},
\\ \sum_{|\alpha| \le {3 N \over 4}} \Vert (1 + t)^{{1 \over 2}} (1 + |u|)^{{3 \over 2} - \delta} \partial \Gamma^\alpha \psi \Vert_{L_t^\infty ([0,T]) L_x^\infty} \le \epsilon^{{3 \over 4}}.
\end{aligned}
\end{equation}
We shall show that assuming these estimates allows us to improve the constants on the right hand side of these norms to $C \epsilon$. This will prove global stability of the trivial solution. In the remainder of this section, we shall recover the energy bootstrap assumptions, which follow easily. We shall also interpolate between the pointwise estimates and the energy estimates. The next section will recover the pointwise bootstrap assumptions.

We shall now recover the bootstrap assumption for the energy. For every $\alpha$ with $|\alpha| \le N$, it suffices to show that $\Vert \partial \Gamma^\gamma \phi \Vert_{L^2 (\Sigma_t)} \le C \epsilon$. Thus, we commute the equation with $\Gamma^\alpha$. We have that
\[
\Box \Gamma^\alpha \psi + (\partial_t \phi)^2 (\partial_t \psi) (\partial_t^2 \Gamma^\alpha \psi) = F'
\]
for an appropriate function $F'$. We shall now explicitly do the energy estimate at this level. If we use $\partial_t \Gamma^\alpha \psi$ as a multiplier, we see that
\begin{equation} \label{eq:enest}
    \begin{aligned}
    \int_{\Sigma_s} (1 - (\partial_t \phi)^2 (\partial_t \psi)) (\partial_t \Gamma^\alpha \psi)^2 + (\partial_x \Gamma^\alpha \psi)^2 + (\partial_y \Gamma^\alpha \psi)^2 d x
    \\ \le \int_{\Sigma_s} (1 - (\partial_t \phi)^2 (\partial_t \psi)) (\partial_t \Gamma^\alpha \psi)^2 + (\partial_x \Gamma^\alpha \psi)^2 + (\partial_y \Gamma^\alpha \psi)^2 d x + \int_0^s \int_{\Sigma_t} |F| |\partial_t \Gamma^\alpha \psi| d x d t,
    \end{aligned}
\end{equation}
where $F$ is such that
\begin{equation} \label{eq:enesterror}
    \begin{aligned}
    |F| \le C \sum_{\beta \le \alpha} |\Gamma^\beta ((\partial_t \phi) (\partial_t \psi)^2)| + C \sum_{\beta \le \alpha} |\Gamma^\beta ((\partial_t \phi)^2 (\partial_t \psi))| + C \sum_{\beta \le \alpha} |\Gamma^\beta ((\partial_t \psi)^4)| = F_1 + F_2 + F_3.
    \end{aligned}
\end{equation}
Now, at least two of the factors in $F_1$ and $F_2$ must have fewer than ${3 N \over 4}$ derivatives on them after applying the chain rule with the $\Gamma^\beta$. Similarly, at least three of the factors in $F_3$ must have fewer than ${3 N \over 4}$ derivatives of them. Moreover, we note that
\[
|\Gamma^\beta (\partial_t \psi)| \le \sum_{\beta_1 \le \beta} |\partial_t \Gamma^{\beta_1} \psi|
\]
because $[\partial_t,\partial] = 0$ and $[\partial_t,S] = \partial_t$. Using the pointwise bootstrap assumptions on these factors with fewer derivatives thus tells us that
\begin{equation} \label{eq:topordererror}
    \begin{aligned}
    |F| \le {C \epsilon^{{6 \over 4}} \over (1 + t)} \sum_{\beta \le \alpha} (|\partial \Gamma^\beta \psi| + |\partial \Gamma^\beta \phi|).
    \end{aligned}
\end{equation}
Repeating the same argument for $\phi$ gives us an expression analogous to \eqref{eq:enest} that satisfies the same estimate.

We now set
\[
E(t) = \sum_{|\alpha| \le N} \int_{\Sigma_t} (\partial_t \Gamma^\alpha \psi)^2 + (\partial_x \Gamma^\alpha \psi)^2 + (\partial_y \Gamma^\alpha \psi)^2 d x + \sum_{|\alpha| \le N} \int_{\Sigma_t} (\partial_t \Gamma^\alpha \phi)^2 + (\partial_x \Gamma^\alpha \phi)^2 + (\partial_y \Gamma^\alpha \phi)^2 d x
\]
By the bootstrap assumptions and \eqref{eq:topordererror}, we have that \eqref{eq:enest} implies that
\[
E(s) \le 2 E(0) + C \epsilon^{{6 \over 4}} \int_0^s {1 \over (1 + t)} \sum_{|\beta| \le N} (|\partial \Gamma^\beta \psi| + |\partial \Gamma^\beta \phi|) |\partial_t \Gamma^\alpha \psi| d x d t \le 2 E(0) + C \epsilon^{{6 \over 4}} \int_0^s {1 \over (1 + t)} E(t) d t.
\]
Thus, by Gronwall's inequality, we have that
\[
E(t) \le 2 E(0) (1 + t)^{2 \epsilon} \le 2 \epsilon^2 (1 + t)^{2 \epsilon}.
\]
This recovers the bootstrap assumption for the energy.

Now, taking $(1 + t)^{-{\delta \over 4}} \partial_t \psi$ as a multiplier and using the bootstrap assumptions, we see that see that the same argument shows that
\[
\int_0^s \int_{\Sigma_t} (1 + t)^{-1 - {\delta \over 4}} \left [(\partial_t \Gamma^\gamma \psi)^2 + (\partial_x \Gamma^\gamma \psi)^2 + (\partial_y \Gamma^\gamma \psi)^2 \right ] d x d t \le C \epsilon^2 + C \epsilon^3 \int_0^s \int_{\Sigma_t} (1 + t)^{-1 + 2 \epsilon - {\delta \over 4}}.
\]
This last integral converges when $\delta > 8 \epsilon$, giving us the desired result.

All that remains is to recover the pointwise bootstrap assumptions. In order to recover these, we shall first interpolate between the pointwise bootstrap assumptions and the energy bootstrap assumptions. The next section uses these interpolated estimates in order to recover the pointwise bootstrap assumptions, which will complete the proof of the main Theorem. The interpolated estimates are the content of the following proposition.

\begin{lemma} \label{lem:interpolation}
Let $|\gamma| \le {3 N \over 4}$. For $N$ sufficiently large as a function of $\delta$ and for $|\beta| \le 10$, we have that
\[
|\Gamma^\beta \partial \Gamma^\gamma \psi| (t,r,\theta) \le {C \epsilon^{{3 \over 4}} \over (1 + t)^{{1 \over 2} - {\delta \over 2}} (1 + |u|^{{3 \over 2} - 2 \delta})}
\]
\end{lemma}
\begin{proof}
We must only consider the region where $t \ge 10$ because the weights do not matter for small $t$. Let $h = \partial \Gamma^\gamma \psi$. We define a coordinate system on $t > 0$ in $\R^{2 + 1}$ as follows. We associate the coordinates $(\rho,\tilde{X},\tilde{Y})$ to the point $(t,x,y)$ in the usual flat coordinates by $\rho = t$, $\tilde{X} = {x \over t}$, and $\tilde{Y} = {y \over t}$. These coordinates are motivated simply by the fact that the scaling vector field points in the same direction as one of the coordinate vector fields (specifically, $\partial_\rho$, as is shown in \eqref{eq:Spartial_rho}).


We shall use Sobolev embedding plus an interpolation result in these coordinates in order to prove the desired result. We note that these coordinates are motivated by the fact that we require a coordinate system which will allow us to effectively interpolate with $S = t \partial_t + r \partial_r$.

We must first examine the coordinate system a bit more carefully. We have that
\[
d \rho = d t, \hspace{5 mm} d \tilde{X} = {1 \over \rho} d x - {x \over t^2} d t, \hspace{5 mm} d \tilde{Y} = {1 \over \rho} d y - {y \over t^2} d t.
\]
Thus, we have that
\[
d vol = d t \wedge d x \wedge d y = d \rho \wedge \left (\rho d \tilde{X} + {\tilde{X} \over \rho} d \rho \right ) \wedge \left (\rho d \tilde{Y} + {\tilde{Y} \over \rho} d \rho \right ) = \rho^2 d \rho \wedge d \tilde{X} \wedge d \tilde{Y}.
\]
Similarly, we have that
\begin{equation} \label{eq:Spartial_rho}
    \begin{aligned}
    \partial_\rho = {\partial t \over \partial \rho} \partial_t + {\partial x \over \partial \rho} \partial_x + {\partial y \over \partial \rho} \partial_y = \partial_t + \tilde{X} \partial_x + \tilde{Y} \partial_y.
    \end{aligned}
\end{equation}
Thus, we have that
\[
\rho \partial_\rho = \rho \partial_t + \rho \tilde{X} \partial_x + \rho \tilde{Y} \partial_y = t \partial_t + x \partial_x + y \partial_y = S.
\]
We also have that
\[
\partial_{\tilde{X}} = {\partial t \over \partial \tilde{X}} \partial_t + {\partial x \over \partial \tilde{X}} \partial_x + {\partial y \over \partial \tilde{X}} \partial_y = \rho \partial_x,
\]
and similarly that
\[
\partial_{\tilde{Y}} = \rho \partial_y.
\]

Let $p$ be a point where we want to show the desired estimate. We shall let $(t_0,x_0,y_0)$ denote its usual coordinates and $(\rho_0,\tilde{X}_0,\tilde{Y}_0)$ denote its coordinates in this modified coordinate system. We shall also denote by $u_0$ its $u$ coordinate. Now, we take two smooth cutoff functions. The first $\chi_\rho (s)$ is defined to be $1$ for $s \ge {3 \over 4}$ and $0$ for $s \le {1 \over 2}$. The second $\chi_{x,y} (s)$ is defined to be $1$ for $|s| \le {1 \over 4}$ and $0$ for $|s| \ge {1 \over 2}$. We then consider the function $\tilde{h} = \chi_\rho \left ({\rho \over \rho_0} \right ) \chi_{x,y} (\rho_0 \sqrt{(\tilde{X} - \tilde{X}_0)^2 + (\tilde{Y} - \tilde{Y}_0)^2}) h$. This has localized $h$ around the point $(\rho_0,\tilde{X}_0,\tilde{Y}_0)$ at scale $1$ in the $\tilde{X}$ and $\tilde{Y}$ directions and at scale $\rho_0$ in the $\rho$ direction (by scale, we mean relative to the $(t,x,y)$ coordinate system). Setting $X = \rho_0 \tilde{X}$ and $Y = \rho_0 \tilde{Y}$ renormalizes the length of partial derivatives to have length comparable to $1$. Indeed, in the support of $\tilde{h}$, we note that the volume form in $(\rho,X,Y)$ coordinates is equal to
\[
{\rho^2 \over \rho_0^2} d \rho \wedge d X \wedge d Y.
\]
This is comparable to $d \rho \wedge d X \wedge d Y$. We also have that $\partial_X = {\rho \over \rho_0} \partial_x$ is comparable to $\partial_x$ are comparable, and similarly with $\partial_Y$ and $\partial_y$.

Now, the function $\tilde{h}$ can be thought of as being defined in $(\rho,X,Y)$ in the region where $\rho \le \rho_0$. Moreover, the pointwise bootstrap assumptions tell us that
\[
\int_{\rho \le \rho_0} \rho_0^{-1 - {\delta \over 2}} \tilde{h}^2 {\rho^2 \over \rho_0^2} d \rho d X d Y \le C \epsilon^{{3 \over 2}} {1 \over (1 + t_0)^{1 + {\delta \over 2}} (1 + |u_0|^{3 - 2 \delta})},
\]
where we have used the fact that $t$ and $1 + |u|$ are comparable to $t_0$ and $1 + |u_0|$, respectively, in the support of $\tilde{h}$, and where we have used the fact that the volume of the support of $h$ is comparable to $t_0$. Moreover, the energy bootstrap assumptions tell us that
\[
\sum_{|\beta| \le N - |\alpha|} \int_{\rho \le \rho_0} \rho_0^{-1 - {\delta \over 2}} (\Gamma^\beta \tilde{h})^2 {\rho^2 \over \rho_0^2} d \rho d x'' d y'' \le C \sum_{|\beta| \le N - |\alpha|} \int_0^{t_0} \int_{\Sigma_t} (1 + t)^{-1 - {\delta \over 2}} (\Gamma^\beta \partial \psi)^2 d x d t \le C \epsilon^{{3 \over 2}}.
\]
Let us consider the function
\[
H = \tilde{h}^2 {\rho^2 \over \rho_0^2}.
\]
We shall show that
\[
\int_{\rho \le \rho_0} (\Gamma_{\rho_0}^\alpha H)^2 d \rho d X d Y \le C \sum_{|\beta| \le |\alpha|} \int_{\rho \le \rho_0} (\Gamma^\beta \tilde{h})^2 {\rho^2 \over \rho_0^2} d \rho d X d Y,
\]
where $\Gamma_{\rho_0}^\alpha$ corresponds to strings of the operators $\rho_0 \partial_\rho$, $\partial_X$, and $\partial_Y$. Indeed, we recall that $\Gamma$ consists of translation vector fields along with $S = t \partial_t + r \partial_r$. This means that $\partial_X = {t \over \rho_0} \partial_x$, $\partial_Y = {t \over \rho_0} \partial_y$, and $\rho_0 \partial_\rho = {\rho_0 \over t} S = \rho_0 \partial_t + {r \rho_0 \over t} \partial_r$. This means that $\Gamma_{\rho_0} = Q \Gamma$ for an appropriate coefficient $Q$. Because $\rho_0$ is comparable to $t$ in the support of $H$, we see that we have that
\[
|\Gamma_{\rho_0} H| \le C |\Gamma H|,
\]
i.e., these coefficients are bounded. 
Moreover, we see that we have that
\[
\Gamma^\beta \left ({t \over \rho_0} \right ) \le C,
\]
and that
\[
\Gamma^\beta \left ({\rho_0 \over t} \right ) \le C.
\]

These facts imply the desired result by expanding the $\Gamma_{\rho_0}$ operators in terms of the $\Gamma$ operators. Indeed, in the case of two operators, we schematically have that
\[
|\Gamma_{\rho_0} (\Gamma_{\rho_0} H))| = |Q \Gamma ( Q \Gamma H)| \le |Q| |\Gamma(Q)| |\Gamma(H)| + |Q| |Q| |\Gamma (\Gamma (H))|,
\]
and the bounds on the coefficients above implies the desired result. Thus, we have in fact shown that
\[
|\Gamma_{\rho_0}^\alpha H| \le C \sum_{|\beta| \le |\alpha|} |\Gamma^\beta H|.
\]
A similar argument shows us that
\begin{equation} \label{eq:commutatorsrescaled}
    \begin{aligned}
    |\Gamma^\alpha H| \le C \sum_{|\beta| \le |\alpha|} |\Gamma^\beta H|.
    \end{aligned}
\end{equation}

Thus, interpolating between these $L^2$ based bounds (we can use an interpolation result in the half spaces $\rho \le \rho_0$ in $(\rho,X,Y)$ coordinate space because they have a nice boundary, see \cite{AdaFou03}) tells us that, for $N$ sufficiently large in terms of $\delta$, we have that
\[
\sum_{|\beta| \le 10} \int_{\rho \le \rho_0} \rho_0^{-1 - {\delta \over 2}} (\Gamma_{\rho_0}^\beta H)^2 d \rho d X d Y \le C \epsilon^{{3 \over 2}} {1 \over (1 + t_0) (1 + |u_0|^{3 - 3 \delta})}.
\]
Then, applying a Sobolev inequality that is weighted in the $\rho$ direction by $\rho_0$ in the $(\rho,X,Y)$ coordinate system tells us that
\[
\sum_{|\beta| \le 8} |\rho_0^{-{1 \over 2} - {\delta \over 4}} \Gamma_{\rho_0}^\beta \tilde{h}| \le C \epsilon^{{3 \over 4}} {1 \over (1 + t_0) (1 + |u_0|)^{{3 \over 2} - {3 \over 2} \delta}}.
\]
Thus, we altogether have that
\[
\sum_{|\beta| \le 8} |\Gamma_{\rho_0}^\beta \tilde{h}| \le C \epsilon^{{3 \over 4}} {1 \over (1 + t_0)^{{1 \over 2} - {\delta \over 2}} (1 + |u_0|)^{{3 \over 2} - 2 \delta}}.
\]
Using \eqref{eq:commutatorsrescaled} above then gives us the desired result.
\end{proof}

We shall also use the fact that $\psi$ is supported where $u \ge -1$ and $\phi$ is supported where $\overline{u} \ge -1$. As discussed in Section~\ref{sec:relateddirections}, we do not believe that having compact support is fundamental to the argument working. We also note that this is trivial when the equation is semilinear instead of quasilinear.

\begin{proposition} \label{prop:compactsupport}
Let $0 \le t \le T$. The bootstrap assumptions imply that $\psi$ is supported where $u \ge -1$ and that $\phi$ is supported where $\overline{u} \ge -1$.
\end{proposition}

\begin{proof}
Let $0 \le s \le T$. We take a point $(s,x_0,y_0)$ in $\Sigma_s$ with $s - \sqrt{x_0^2 - y_0^2} \le -1$. We then take the ball of radius $\delta$ around this point in $\Sigma_s$. We look at the domain of influence of this ball intersected with $\Sigma_t$ for $0 \le t \le s$. In every $\Sigma_t$, this is a ball whose radius is equal to $s - t + \delta$ and whose center is $(s - t,x_0,y_0)$. By construction, at $t = 0$, this does not intersect the support of the data for $\psi$ for $\delta$ sufficiently small.

We set
\[
E(t) = \int_{B_{s - t + \delta} (x_0,y_0)} (\partial_t \psi)^2 + (\partial_x \psi)^2 + (\partial_y \psi)^2 d x.
\]
This is the energy in the intersection of the domain of influence of the small ball chosen above and the time slice $\Sigma_t$. We now think of the equation for $\psi$ as a wave equation whose principal part if $\Box$, putting both the quasilinear and semilinar terms on the right hand side. We do a $\partial_t$ energy estimate on the truncated cone determined by this construction and ignore the positive terms on the side of the cone, giving us that
\[
E(s) \le E(0) + \int_0^s {C \epsilon^{{3 \over 2}} \over 1 + t} E(t) d t,
\]
where we have used the pointwise bootstrap assumptions. Gronwall's inequality then gives us that $E(s) = 0$ because $E(0) = 0$. An analogous argument works for $\phi$.
\end{proof}

\subsection{Closing the pointwise estimates} \label{sec:closingpointwise}
In order to recover the bootstrap assumptions for the pointwise estimates, we must show that
\[
|\partial_t \Gamma^{\sigma_1} \psi| (t,x) \le {C \epsilon \over (1 + t)^{{1 \over 2}} (1 + |u|)^{{3 \over 2} - \delta}}
\]
for all $|\sigma_1| \le {3 N \over 4}$. Commuting the equation with $\Gamma^\alpha$, we get the equation
\[
\Box \Gamma^\alpha \psi + (\partial_t \psi) (\partial_t^2 \phi) (\partial_t \Gamma^\alpha \psi) = F^\alpha
\]
for an appropriate $F^\alpha$ which contains only semilinear terms.

Now, if we want to apply Proposition~\ref{prop:decay} in order to show pointwise decay for this quantity, we must additionally commute with the operators $\partial^{\sigma_2}$ tangent to $\Sigma_t$ for all $|\sigma_2| \le 6$. Thus, it suffices to commute with $\Gamma^{\sigma_2}$ for all $|\sigma_2| \le 6$. We thus commute the equation with $\Gamma^{\sigma_2} \Gamma^{\sigma_1}$. Thus, in order to recover the pointwise bootstrap assumptions, it suffices to show that
\[
M [\Gamma^\sigma \psi] (s,x_0^i) \le {C \epsilon \over (1 + s)^{{1 \over 2}} (1 + |\tau|^{{3 \over 2} - \delta})}
\]
for all $|\sigma| \le {3 N \over 4} + 6$ and for all $s \le T$ (see Proposition~\ref{prop:decay} for a description of $M$).

Now, when controlling $M$, we note that the estimate on the term from data follows immediately from assuming linear decay for the auxiliary multipliers $f$. Thus, for any admissible auxiliary multiplier $f$, we must simply control the error integrals
\begin{equation} \label{eq:errorint}
    \begin{aligned}
    \int_0^s \int_{\Sigma_t} F^\sigma (\partial_t f) d x d t - \int_0^s \int_{\Sigma_t} (\partial_t \phi)^2 (\partial_t \psi) (\partial_t^2 \Gamma^\sigma \psi) (\partial_t f) d x d t.
    \end{aligned}
\end{equation}
Now, by examining the equations \eqref{eq:anisotropic}, we note that $F^\sigma$ arises from applying vector fields to one of three different semilinear terms. The first is
\[
(\partial_t \psi)^2 (\partial_t \phi),
\]
the second is
\[
(\partial_t \psi) (\partial_t \phi)^2,
\]
and the third is
\[
(\partial_t \psi)^4.
\]
Thus, in addition to the quasilinear term in \eqref{eq:errorint}, we must control the error integrals arising from these semilinear terms. We discussed in the remarks following the statement of Theorem~\ref{thm:mainthm} that the same proof works for other nonlinearities. The other nonlinear terms can all be controlled in essentially the same way as one of these terms.

The error integral \eqref{eq:errorint} can thus be controlled by
\begin{equation} \label{eq:ErrorIntegral}
    \begin{aligned}
    \sum_{\sigma' \le \sigma} \left |\int_0^s \int_{\Sigma_t} \Gamma^{\sigma'} ((\partial_t \psi) (\partial_t \phi)^2) (\partial_t f) d x d t \right | + \sum_{\sigma' \le \sigma} \left |\int_0^s \int_{\Sigma_t} \Gamma^{\sigma'} ((\partial_t \psi)^4) (\partial_t f) d x d t \right |
    \\ + \sum_{\sigma' \le \sigma} \left |\int_0^s \int_{\Sigma_t} \Gamma^{\sigma'} ((\partial_t \psi)^2 (\partial_t \phi)) (\partial_t f) d x d t \right | + \left |\int_0^s \int_{\Sigma_t} (\partial_t \phi)^2 (\partial_t \psi) (\partial_t^2 \Gamma^\sigma \psi) (\partial_t f) d x d t \right |
    \end{aligned}
\end{equation}


We recall that the data for $f$ is contained in the ball of radius $1$ in $\Sigma_s$ whose center is $(s,x_0,y_0)$. Moreover, we recall that the $u$ coordinate of the center of this ball is given by $\tau$, meaning that $\tau = s - \sqrt{x_0^2 + y_0^2}$.

In the following, we shall always use the pointwise estimates for every term up to a ${\delta \over 2}$ loss. These interpolation arguments allowing the use of pointwise estimates (with small losses) on all factors in the nonlinearity when recovering estimates that do not involve the most number of derivatives are standard, but we describe it here for completeness. We note that this argument is far from optimal in terms of the number of derivatives required.

We recall that using Proposition~\ref{prop:decay} to recover the pointwise bootstrap assumptions requires commuting with a total of ${3 N \over 4} + 6$ derivatives. This means that cubic nonlinearities will have at least two factors having no more than ${3 N \over 8} + 5$ derivatives on them while quartic nonlinearities will have at least three factors with this property. Thus, for ${3 N \over 8} + 5 \le {3 N \over 4}$, we can directly apply the bootstrap assumptions to use pointwise estimates on these factors. Moreover, we note that there can be at most one factor in every nonlinearity having at most ${3 N \over 4} + 8$ derivatives on it. For this factor, we can use the interpolation result in Lemma~\ref{lem:interpolation} to give us pointwise decay up to ${\delta \over 2}$ losses for this factor. We see that the resulting pointwise bounds are worst for one of the three semilinear terms described above. Thus, it suffices to control the error integrals involving these terms in \eqref{eq:ErrorIntegral}. We shall now consider each kind of term separately. Moreover, we shall assume throughout that $\tau \ge 100$ and that $s \ge {1 \over \delta_0^{10}}$. When $s \le {1 \over \delta_0^{10}}$, we may control terms simply by choosing $\epsilon$ sufficiently small (the decay does not really matter at this point for controlling the nonlinearity because we can use the higher power of $\epsilon$ and the fact that $s$ is uniformly bounded to absorb everything). Similarly, when $\tau \le 100$, the dependence on $\tau$ no longer matters, as it can be absorbed by choosing $\epsilon$ smaller. The estimates when $\tau \le 100$ can be recovered from similar considerations as in the following (and by taking $\epsilon$ sufficiently small).

We begin with terms of the form
\begin{equation} \label{eq:psi2phi1}
    \begin{aligned}
    \left |\int_0^s \int_{\Sigma_t} \Gamma^\sigma ((\partial_t \phi) (\partial_t \psi)^2) (\partial_t f) d x d t \right |.
    \end{aligned}
\end{equation}
The error integrals must be broken up into several regions. Broadly speaking, there are three main regions. The first and most delicate region is the region close to the light cones for both $\psi$ and $f$, where ``close" is measured with respect to $\tau$. In this region, we integrate by parts using the scaling vector field $S$ and use the results of Section~\ref{sec:SGeometry}. The second region is the region away from the light cone for $\psi$, and the third region is away from the light cone for $f$. These regions are better because of decay in $u$ and $u'$. All three of these main regions must be further decomposed mainly for the technical fact that Lemma~\ref{lem:rrbar} means that the Jacobian in $(t,r,\overline{r})$ coordinates on $\R^{2 + 1}$ is bounded only when we are comparable to $t$ close to the light cones for both $\psi$ and $\phi$. Thus, we must for example decompose depending on how far away we are from the light cone associated to $\phi$.

We must introduce several cutoff functions in the proof, so let us now describe the notation a bit. These cutoff functions will allow us to effectively use the geometric estimates from Sections~\ref{sec:geometry} and \ref{sec:SGeometry}. There will be times when these cutoff functions will be hit by derivatives, and it will be appropriate to bound them in terms of cutoff functions adapted to slightly larger regions. In these cases, we shall abuse notation and simply use the same cutoff functions. These cutoff functions will be pullbacks of cutoff functions on $\R$. Thus, for example, a cutoff function may be of the form $\chi \left ({u \over t} \right )$ where $\chi$ is a smooth function from $\R \rightarrow \R$ and $u$ and $t$ are the usual coordinates. Thus, given such a cutoff function $\chi$, a vector field $V$ hitting $\chi$ will always result in something of the form $h \chi'$ for some smooth function $h$. In the example above, if $V = S$, we have that
\[
S \chi \left ({u \over t} \right ) = -{u \over t} \chi' \left ({u \over t} \right ) - {r \over t} \chi' \left ({u \over t} \right ) = -{1 \over t} \chi' \left ({u \over t} \right ).
\]

Let $\chi : \R \rightarrow \R$ be a smooth cutoff function equal to $1$ for $x \le {\delta_0 \over 2}$ and equal to $0$ for $x \ge \delta_0$. Then, we take the functions $\chi_{\phi,C} = \chi \left ({\overline{u} \over t} \right )$, $\chi_\psi = \chi \left ({u \over \tau} \right )$, $\chi_{\psi,C} = \chi \left ({u \over t} \right )$, and $\chi_f = \chi \left ({u' \over \tau} \right )$. We also define the cutoff functions $\chi_\psi^c = 1 - \chi_\psi$, and similarly for the others. We shall sometimes need other analogous cutoffs, such as $\chi_\phi = \chi \left ({\overline{u} \over \tau} \right )$, and we shall introduce them following the notational conventions set above. These cutoff functions will be used to localize the error integral to the regions described above. The most delicate regions are those along the light cones for both $\psi$ and $f$ because they require decomposition $\partial_{u'}$ in terms of $S$ and $\overline{\partial}_f$. The other regions can be bounded directly using the bootstrap assumptions. We note that the region along all three light cones will be the only one that returns ${1 \over (1 + s)^{{1 \over 2}} \tau^{{3 \over 2} - \delta}}$ up to a power of $\delta$. The other regions are all better.



We first describe the region away from the light cone for $f$. This is given by the integral
\begin{equation} \label{eq:psi2phi1fc}
    \begin{aligned}
    \left |\int_0^s \int_{\Sigma_t} \chi_f^c \Gamma^\sigma ((\partial_t \phi) (\partial_t \psi)^2) (\partial_t f) d x d t \right |.
    \end{aligned}
\end{equation}
When also along the light cones for $\psi$ and $\phi$, we use the bootstrap assumptions, the interpolation result Lemma~\ref{lem:interpolation}, and the assumed decay rates for $f$ (see \eqref{eq:assumeddecay0} and \eqref{eq:assumeddecay1}) to give us that this integral is controlled by
\begin{equation}
    \begin{aligned}
    \left |\int_0^s \int_{\Sigma_t} \chi_{\psi,C} \chi_{\phi,C} \chi_f^c \Gamma^\sigma ((\partial_t \phi) (\partial_t \psi)^2) (\partial_t f) d x d t \right |
    \\ \le {C \epsilon^{{9 \over 4}} \over \tau^{{3 \over 2}}} \int_0^s \int_{\Sigma_t} \chi_{r \le t + 1} \chi_{\psi,C} \chi_{\phi,C} {1 \over (1 + t)^{{3 \over 2} - \delta}} {1 \over 1 + |u|^{3 - 3 \delta}} {1 \over 1 + |\overline{u}|^{{3 \over 2} - 2 \delta}} {1 \over (1 + s - t)^{{1 \over 2}}} d x d t.
    \end{aligned}
\end{equation}
We ignore the power of $\epsilon$ because it is already large enough to recover the bootstrap assumption. Going into $r \overline{r}$ coordinates as in Lemma~\ref{lem:rrbar} and integrating in each $\Sigma_t$, the remaining quantity is controlled by
\[
{C \over \tau^{{3 \over 2}}} \int_0^s {1 \over (1 + t)^{{3 \over 2} - \delta}} {1 \over (1 + s - t)^{{1 \over 2}}} d t \le {C \over (1 + s)^{{1 \over 2}} \tau^{{3 \over 2}}}.
\]
Now, in the region which is instead away from the light cone for $\phi$, we have that
\begin{equation}
\begin{aligned}
\left |\int_0^s \int_{\Sigma_t} \chi_{\psi,C} \chi_{\phi,C}^c \chi_f^c \Gamma^\sigma ((\partial_t \phi) (\partial_t \psi)^2) (\partial_t f) d x d t \right |
\\ \le {C \over \tau^{{3 \over 2}}} \int_0^s \int_{\Sigma_t} \chi_{r \le t + 1} \chi_{\psi,C} \chi_{\phi,C}^c {1 \over (1 + t)^{3 - 3 \delta}} {1 \over 1 + |u|^{3 - 3 \delta}} {1 \over (1 + s - t)^{{1 \over 2}}} d x d t
\\ \le {C \over \tau^{{3 \over 2}}} \int_0^s {1 \over (1 + t)^{2 - 3 \delta}} {1 \over (1 + s - t)^{{1 \over 2}}} d t \le {C \over (1 + s)^{{1 \over 2}} \tau^{{3 \over 2}}}.
\end{aligned}
\end{equation}
Then, in the region which is instead away from the light cone for $\psi$, we have that
\begin{equation} \label{eq:psiIc}
\begin{aligned}
\left |\int_0^s \int_{\Sigma_t} \chi_{\psi,C}^c \chi_f^c \Gamma^\sigma ((\partial_t \phi) (\partial_t \psi)^2) (\partial_t f) d x d t \right |
\\ \le {C \epsilon^{{9 \over 4}} \over \tau^{{3 \over 2}}} \int_0^s \int_{\Sigma_t} \chi_{\psi,C}^c {1 \over (1 + t)^{{9 \over 2} - 3 \delta}} {1 \over 1 + |\overline{u}|^{{3 \over 2} - \delta}} {1 \over (1 + s - t)^{{1 \over 2}}} d x d t
\\ \le {C \epsilon^{{9 \over 4}} \over \tau^{{3 \over 2}}} \int_0^s {1 \over (1 + t)^{{7 \over 2} - 3 \delta}} {1 \over (1 + s - t)^{{1 \over 2}}} d t \le {C \epsilon^{{9 \over 4}} \over (1 + s)^{{1 \over 2}} \tau^{{3 \over 2}}}.
\end{aligned}
\end{equation}

This has completely treated the integrals in the region away from the light cone for $f$. We now consider the region away from the light cone for $\psi$. We must in fact decompose this into two regions depending on how far away we are. We first consider the regions that are far from the light cone for $\psi$, but not too far. This corresponds to using the cuttoffs $\chi_\psi^c$ and $\chi_{\psi,C}$. When also close to the light cone for $\phi$, we can write
\begin{equation} \label{eq:errorpsifarphiclose11}
    \begin{aligned}
    \left |\int_0^s \int_{\Sigma_t} \chi_\psi^c \chi_{\psi,C} \chi_{\phi,C} \Gamma^\sigma ((\partial_t \phi) (\partial_t \psi)^2) (\partial_t f) d x d t \right |
    \\ \le C \epsilon^{{9 \over 4}} \int_0^s \int_{\Sigma_t} \chi_\psi^c \chi_{\psi,C} \chi_{\phi,C} {1 \over (1 + t)^{{3 \over 2} - \delta}} {1 \over 1 + |u|^{3 - 3 \delta}} {1 \over 1 + |\overline{u}|^{{3 \over 2} - 2 \delta}} {1 \over (1 + s - t)^{{1 \over 2}}} d x d t.
    \end{aligned}
\end{equation}
Going into $r \overline{r}$ coordinates once again and using the fact that $u \ge {\delta_0 \over 10} \tau$ in this region gives us that the integral is in fact controlled by
\[
{C \over \tau^{2 - 3 \delta}} \int_0^s {1 \over (1 + t)^{{3 \over 2} - \delta}} {1 \over (1 + s - t)^{{1 \over 2}}} d t \le {C \over (1 + s)^{{1 \over 2}} \tau^{2 - 3 \delta}}.
\]

Now, when far from the light cone for $\phi$, we have that
\begin{equation} \label{eq:errorpsifarphiclose12}
\begin{aligned}
\left |\int_0^s \int_{\Sigma_t} \chi_\psi^c \chi_{\psi,C} \chi_{\phi,C}^c \Gamma^\sigma ((\partial_t \phi) (\partial_t \psi)^2) (\partial_t f) d x d t \right |
\\ \le C \epsilon^{{9 \over 4}} \int_0^s \int_{\Sigma_t} \chi_\psi^c \chi_{\psi,C} \chi_{\phi,C}^c {1 \over (1 + t)^{3 - 3 \delta}} {1 \over 1 + |u|^{3 - 3 \delta}} {1 \over (1 + s - t)^{{1 \over 2}}} d x d t
\\ \le {C \epsilon^{{9 \over 4}} \over \tau^{3 - 3 \delta}} \int_0^s {1 \over (1 + t)^{2 - 3 \delta}} {1 \over (1 + s - t)^{{1 \over 2}}} d t \le {C \epsilon^{{9 \over 4}} \over (1 + s)^{{1 \over 2}} \tau^{3 - 3 \delta}}.
\end{aligned}
\end{equation}

We now consider the same region except with $\chi_{\psi,C}^c$ instead. Using that ${1 \over 1 + |u|^{3 - 3 \delta}} \le C {1 \over (1 + t)} {1 \over \tau^{2 - 3 \delta}}$ in this region, we then have that
\begin{equation} \label{eq:psiveryfar1}
    \begin{aligned}
    \left |\int_0^s \int_{\Sigma_t} \chi_\psi^c \chi_{\psi,C}^c \Gamma^\sigma ((\partial_t \phi) (\partial_t \psi)^2) (\partial_t f) d x d t \right |
    \\ \le {C \epsilon^{{9 \over 4}} \over \tau^{2 - 3 \delta}} \int_0^s \int_{\Sigma_t} \chi_{\overline{r} \le t + 1} {1 \over (1 + t)^{{5 \over 2} - \delta}} {1 \over 1 + |\overline{u}|^{{3 \over 2} - 2 \delta}} {1 \over (1 + s - t)^{{1 \over 2}}} d x d t.
    \end{aligned}
\end{equation}
Going into $(t,\overline{r},\overline{\theta})$ coordinates and using the decay in $\overline{u}$ to control the integrals on $\Sigma_t$ by $C (1 + t)$, the integral is then bounded by
\[
{C \epsilon^{{9 \over 4}} \over \tau^{2 - 3 \delta}} \int_0^s {1 \over (1 + t)^{{3 \over 2} - \delta}} {1 \over (1 + s - t)^{{1 \over 2}}} d t \le {C \epsilon^{{9 \over 4}} \over (1 + s)^{{1 \over 2}} \tau^{2 - 3 \delta}}.
\]

The only region that remains is the region along the light cones for both $\psi$ and $f$. This integral is given by
\begin{equation} \label{eq:errorpsiphif}
    \begin{aligned}
    \int_0^s \int_{\Sigma_t} \chi_\psi \chi_f \Gamma^\sigma ((\partial_t \phi) (\partial_t \psi)^2) (\partial_t f) d x d t.
    \end{aligned}
\end{equation}
We first pick coordinates where $b = 0$, meaning that the center adapted to $r'$ lies on the new $x$ axis as is described in Section~\ref{sec:coordinates} (we shall do this freely below without mentioning it again). Now, we write
\begin{equation} \label{eq:dtfdecomp}
    \begin{aligned}
    \partial_t = -{1 \over 2} (L' + \underline{L}') = -{1 \over 2} L' - \gamma S - {1 \over 2} \gamma \left (t - r' - {a (x - a) \over r'} \right ) L' - \gamma {a y \over (r')^2} \partial_{\theta'},
    \end{aligned}
\end{equation}
where we have used Lemma~\ref{lem:du'S}. This results in the integral
\begin{equation} \label{eq:errorintworst1}
    \begin{aligned}
    \int_0^s \int_{\Sigma_t} \chi_\psi \chi_f \Gamma^\sigma ((\partial_t \phi) (\partial_t \psi)^2) \left [-{1 \over 2} L' - \gamma S - {1 \over 2} \gamma \left (t - r' - {a (x - a) \over r'} \right ) L' - \gamma {a y \over (r')^2} \partial_{\theta'} \right ] f d x d t.
    \end{aligned}
\end{equation}

We focus first on the term with $S$. The integral we must control is
\[
\int_0^s \int_{\Sigma_t} \chi_\psi \chi_f \Gamma^\sigma ((\partial_t \phi) (\partial_t \psi)^2) (\gamma S f) d x d t.
\]
We integrate this term by parts to put the scaling vector field on the other terms. Doing so shows that it suffices to control
\begin{equation} \label{eq:Sibp}
\begin{aligned}
\left |\int_0^s \int_{\Sigma_t} \chi_\psi \chi_f \Gamma^\sigma ((\partial_t \phi) (\partial_t \psi)^2) (S \gamma) f d x d t\right | + \left |\int_0^s \int_{\Sigma_t} \gamma \chi_\psi \chi_f S \Gamma^\sigma ((\partial_t \phi) (\partial_t \psi)^2) f d x d t \right |
\\ + \left |\int_0^s \int_{\Sigma_t} \gamma S(\chi_\psi \chi_f) \Gamma^\sigma ((\partial_t \phi) (\partial_t \psi)^2) f d x d t \right | + \left |\int_0^s \int_{\Sigma_t} \gamma \chi_\psi \chi_f \Gamma^\sigma ((\partial_t \phi) (\partial_t \psi)^2) f d x d t \right |,
\end{aligned}
\end{equation}
where we have used the fact that the integrand is compactly supported where ${\tau \over 10} \le t \le s - {\tau \over 10}$, meaning that there are no boundary terms from integrating by parts. We begin by estimating $\gamma S (\chi_\psi)$ and $\gamma S (\chi_f)$. We note that $\partial_v \chi_\psi = 0$ and $|\partial_u \chi_\psi| \le {C \over \tau}$. Thus, we have that
\[
|\gamma S (\chi_\psi)| = |\gamma| |v \partial_v \chi_\psi + u \partial_u \chi_\psi| \le C |\gamma|.
\]
An analogous argument works to control $\gamma S$ applied to cutoffs adapted to $\phi$ (see also Section~\ref{sec:coordinates}), and a similar argument works for the other kinds of cutoffs adapted to $\psi$ and $\phi$. Now, for $\chi_f$, we shall use the background Euclidean structure on $\R^3$. We note that $\nabla_e \chi_f = h \underline{L}'$ for some smooth function $h$ with $|h| \le {C \over \tau}$ where $\nabla_e \chi_f$ is the Euclidean gradient of $\chi_f$. We have that
\[
|S (\chi_f)| = |\langle S,\nabla_e \chi_f \rangle_e| = |h| |\langle S,\underline{L}' \rangle_e|,
\]
where $\langle \cdot,\cdot \rangle_e$ denotes the Euclidean inner product on $\R^3$. Now, using Lemma~\ref{lem:du'S} to write $S$ in terms of the frame $L'$, $\underline{L}'$, and ${1 \over r'} \partial_{\theta'}$, and noting that these vectors are mutually orthogonal with respect to the Euclidean inner product, we have that
\[
|\langle S,\underline{L}' \rangle_e| \le {C \over \gamma}.
\]
Thus, we altogether have that
\[
|S(\chi_f)| \le {C \over \tau \gamma},
\]
meaning that
\[
|\gamma S(\chi_f)| \le {C \over \tau}.
\]

Using these estimates along with Lemma~\ref{lem:gammaSgammabound} gives us that these integrals are all controlled by
\begin{equation} \label{eq:worsterrorint1}
    \begin{aligned}
    {C \epsilon^{{9 \over 4}} \over \tau} \int_{\tau \over 10}^{s - {\tau \over 10}} \int_{\Sigma_t} \chi_{r \le t + 1} \chi_f \chi_\psi {1 \over (1 + t)^{{3 \over 2} - 2 \delta}} {1 \over 1 + |u|^{3 - 3 \delta}} {1 \over 1 + |\overline{u}|^{{3 \over 2} - 2 \delta}} {1 \over (1 + s - t)^{{1 \over 2}}} d x d t.
    \end{aligned}
\end{equation}
We note that we have used the fact that the integrand is only supported for ${\tau \over 10} \le t \le s - {\tau \over 10}$ (see Lemma~\ref{lem:rtr's-t}).

Now, we further decompose this integral depending on the distance to the light cone for $\phi$. When close to the light cone for $\phi$, we can go into $r \overline{r}$ coordinates and the Jacobian will be bounded by Lemma~\ref{lem:rrbar}. We thus have that
\begin{equation}
\begin{aligned}
{1 \over \tau} \int_{\tau \over 10}^{s - {\tau \over 10}} \int_{\Sigma_t} \chi_{r \le t + 1} \chi_f \chi_\psi \chi_{\phi,C} {1 \over (1 + t)^{{3 \over 2} - \delta}} {1 \over 1 + |u|^{3 - 3 \delta}} {1 \over 1 + |\overline{u}|^{{3 \over 2} - 2 \delta}} {1 \over (1 + s - t)^{{1 \over 2}}} d x d t
\\ \le {C \over \tau} \int_{\tau \over 10}^{s - {\tau \over 10}} \int_{\Sigma_t} \chi_{r \le t + 1} \chi_f \chi_\psi \chi_{\phi,C} {1 \over (1 + t)^{{3 \over 2} - \delta}} {1 \over 1 + |u|^{3 - 3 \delta}} {1 \over 1 + |\overline{u}|^{{3 \over 2} - 2 \delta}} {1 \over (1 + s - t)^{{1 \over 2}}} d r d \overline{r} d t
\\ \le {C \over \tau} \int_{{\tau \over 10}}^{s - {\tau \over 10}} {1 \over (1 + t)^{{3 \over 2} - \delta}} {1 \over (1 + s - t)^{{1 \over 2}}} d t \le {C \over (1 + s)^{{1 \over 2}} \tau^{{3 \over 2} - \delta}}.
\end{aligned}
\end{equation}

We now consider the region far from the light cone of $\phi$. We have that
\begin{equation}
\begin{aligned}
{C \over \tau} \int_{\tau \over 10}^{s - {\tau \over 10}} \int_{\Sigma_t} \chi_{r \le t + 1} \chi_f \chi_\psi \chi_{\phi,C}^c {1 \over (1 + t)^{{3 \over 2} - \delta}} {1 \over 1 + |u|^{3 - 3 \delta}} {1 \over 1 + |\overline{u}|^{{3 \over 2} - 2 \delta}} {1 \over (1 + s - t)^{{1 \over 2}}} d x d t
\\ \le {C \over \tau} \int_{\tau \over 10}^{s - {\tau \over 10}} \int_{\Sigma_t} \chi_{r \le t + 1} \chi_f \chi_\psi {1 \over (1 + t)^{3 - 3 \delta}} {1 \over 1 + |u|^{3 - 3 \delta}} {1 \over (1 + s - t)^{{1 \over 2}}} d x d t
\\ \le {C \over \tau} \int_{{\tau \over 10}}^{s - {\tau \over 10}} {1 \over (1 + t)^{2 - 3 \delta}} {1 \over (1 + s - t)^{{1 \over 2}}} d t \le {C \over (1 + s)^{{1 \over 2}} \tau^{2 - 3 \delta}}
\end{aligned}
\end{equation}

We now consider the terms without $S$ in \eqref{eq:errorintworst1}. Using Lemma~\ref{lem:du'dv'dtheta'} and improved decay of good derivatives of $f$ (see \eqref{eq:assumeddecay2}), we see that the integrals are once again controlled by
\[
{C \epsilon^{{9 \over 4}} \over \tau} \int_{{\tau \over 10}}^{s - {\tau \over 10}} \int_{\Sigma_t} \chi_{r \le t + 1} \chi_\psi \chi_f {1 \over (1 + t)^{{3 \over 2} - \delta}} {1 \over 1 + |u|^{3 - 3 \delta}} {1 \over 1 + |\overline{u}|^{{3 \over 2} - 2 \delta}} {1 \over (1 + s - t)^{{1 \over 2}}} d x d t.
\]
Thus, the bounds for these terms follow in the same way as for \eqref{eq:worsterrorint1}.



We now turn to control the terms of the form
\[
\left |\int_0^s \int_{\Sigma_t} \Gamma^\sigma ((\partial_t \psi) (\partial_t \phi)^2) (\partial_t f) d x d t \right |.
\]
We once again introduce cutoff functions adapted to the light cones for $\psi$, $\phi$, and $f$. We first consider the region along all three light cones. This is given by
\[
\left |\int_0^s \int_{\Sigma_t} \chi_\psi \chi_\phi \chi_f \Gamma^\sigma ((\partial_t \psi) (\partial_t \phi)^2) (\partial_t f) d x d t \right |
\]
This integral can be handled in exactly the same way as \eqref{eq:errorpsiphif}. Moreover, the integral over the region away from the light cone for $f$ can be controlled in a similar way to the term \eqref{eq:psi2phi1fc}, and the integral away from the light cone for $\phi$ can be controlled in a similar way to the terms \eqref{eq:errorpsifarphiclose11}, \eqref{eq:errorpsifarphiclose12}, and \eqref{eq:psiveryfar1}.


We now turn to the final region (which is the most difficult new region for these terms compared to the terms of the form \eqref{eq:psi2phi1}). This is the region away from the light cone of $\psi$ but still along the light cones for $\phi$ and $f$. This is given by
\[
\left |\int_0^s \int_{\Sigma_t} \chi_\psi^c \chi_\phi \chi_f \Gamma^\sigma ((\partial_t \psi) (\partial_t \phi)^2) (\partial_t f) d x d t \right|
\]
Let us briefly describe why this term is problematic for this cubic nonlinearity but not the other one. The worst region is when we are along the light cone for $f$ and the linear factor in the nonlinearity (which is composed of one linear and one quadratic factor in both cases). In this region, we want to decompose $\partial_t$ in terms of $S$ and $\overline{\partial}_f$, as was done for the term \eqref{eq:errorintworst1}. For that term, because we are restricted to $u$ being small relative to $\tau$, we have good estimates for $S$ by Lemma~\ref{lem:gammaSgammabound}. However, in this term, $u$ is comparable to $\tau$, so we do not have good estimates for $S$ everywhere. We do, however, have a good power of $\tau$ (this is in fact almost enough already to close the argument). Thus, we must decompose into regions where $S$ is still useful and other regions where the geometry helps us.

For technical reasons (to avoid $r' = 0$), we must further decompose this into a region where $r'$ is comparable to $s - t$ and a region where $r'$ is much smaller than $s - t$. We thus define the function
\[
\chi_{f,I} = \chi \left ({r' - 100 \over 1 + s - t} \right ).
\]
Then, in the support of $\chi_{f,I}$ and where $s - t \ge 0$, we note that $1 + |u'| \ge c (1 + s - t)$. We now consider first the region which is within the light cone for $f$ in this sense. This is given by
\begin{equation} \label{eq:errorIntf2}
    \begin{aligned}
    \left |\int_0^s \int_{\Sigma_t} \chi_\psi^c \chi_{f,I} \chi_\phi \chi_f \Gamma^\sigma ((\partial_t \psi) (\partial_t \phi)^2) (\partial_t f) d x d t \right|.
    \end{aligned}
\end{equation}
Now, using that $1 + |u'| \ge c (1 + s - t)$ in this region, we have that this integral is controlled by
\[
C \epsilon^{{9 \over 4}} \int_0^s \int_{\Sigma_t} \chi_\psi^c \chi_{f,I} \chi_\phi \chi_f {1 \over (1 + t)^{{3 \over 2} - \delta}} {1 \over 1 + \tau^{{3 \over 2} - 2 \delta}} {1 \over 1 + |\overline{u}|^{3 - 2 \delta}} {1 \over (1 + s - t)^{{3 \over 2}}} d x d t.
\]
We in fact could improve the power of $1 + s - t$ from ${3 \over 2}$ to $2$, but we use ${3 \over 2}$ in order to treat this term the same way as terms that follow. We ignore the power of $\epsilon$ which is larger than $\epsilon^{{3 \over 4}}$, and we focus on the powers of $s$ and $\tau$. We break this integral up into two regions, one where $t \ge (1 - 10 \delta_0) s$ and another where $t \le (1 - 10 \delta_0) s$. In the first region, we note that the support of $\chi_f \chi_{f,I}$ is a set of very small diameter compared to the scale of the ellipses which are the level sets of $r'$ in the region where $u' \le \delta_0 \tau$. Thus, we use Lemma~\ref{lem:largecirclearc} (or more specifically, the version for an ellipse described after Lemma~\ref{lem:largecirclearc}) to note that the integral is controlled by
\[
C {1 \over (1 + s)^{{3 \over 2} - \delta}} {1 \over 1 + \tau^{{3 \over 2} - 2 \delta}} \int_{(1 - 10 \delta_0) s}^s {1 \over (1 + s - t)^{{1 \over 2}}} d t \le C {1 \over (1 + s)^{1 - \delta}} {1 \over 1 + \tau^{{3 \over 2} - 2 \delta}}
\]
Then, in the region where $t \le (1 - 10 \delta_0) s$, we note that $1 + s - t \ge c (1 + s)$. Thus, the integral is controlled by
\[
C {1 \over (1 + s)^{{3 \over 2}}} {1 \over 1 + \tau^{{3 \over2} - 2 \delta}} \int_0^{(1 - 10 \delta_0) s} \int_{\Sigma_t} \chi_{\overline{r} \le t + 1} {1 \over (1 + t)^{{3 \over 2} - \delta}} {1 \over 1 + |\overline{u}|^{3 - 2 \delta}} d x d t.
\]
Going into $(t,\overline{r},\overline{\theta})$ coordinates then gives us that this is controlled by
\[
C {1 \over (1 + s)^{{3 \over2}}} {1 \over 1 + \tau^{{3 \over 2} - 2 \delta}} \int_0^{(1 - 10 \delta_0) s} {1 \over (1 + t)^{{1 \over 2} - \delta}} d t \le C {1 \over (1 + s)^{1 - \delta}} {1 \over 1 + \tau^{{3 \over 2} - 2 \delta}}.
\]

We now consider the region associated with $\chi_{f,I}^c$ instead, and we note that $r' \ge 10$ and $s - t \ge 10$ in the support of $\chi_{f,I}^c$. The integral we must control is then

\[
\left |\int_0^s \int_{\Sigma_t} \chi_\psi^c \chi_{f,I}^c \chi_\phi \chi_f \Gamma^\sigma ((\partial_t \psi) (\partial_t \phi)^2) (\partial_t f) d x d t \right|
\]

Now, using \eqref{eq:dtfdecomp} to decompose $\partial_t f$, we are left with an integral of the form
\begin{equation} \label{eq:errorintworst2}
    \begin{aligned}
    \int_0^s \int_{\Sigma_t} \chi_\psi^c \chi_{f,I}^c \chi_f \chi_\phi \Gamma^\sigma ((\partial_t \phi)^2 (\partial_t \psi)) \left [-{1 \over 2} L' - \gamma S - {1 \over 2} \gamma \left (t - r' - {a (x - a) \over r'} \right ) L' - \gamma {a y \over (r')^2} \partial_{\theta'} \right ] f d x d t.
    \end{aligned}
\end{equation}
Most of the regions that follow will use this decomposition of $\partial_t$ in terms of $S$ and $\overline{\partial}_f$. There is only one region that will not.

We must now consider two cases, one in which $\tau \le \delta_0 s$ and one in which $\tau \ge \delta_0 s$. We begin with the easier case of $\tau \ge \delta_0 s$. We first consider the terms not involving $S$. By Lemma~\ref{lem:du'dv'dtheta'}, we know that the integral of these terms is controlled by
\[
C \int_0^s \int_{\Sigma_t} \chi_\psi^c \chi_{f,I}^c \chi_f \chi_\phi \left |\Gamma^\sigma ((\partial_t \phi)^2 (\partial_t \psi)) \right | {1 \over r'} {1 \over (1 + s - t)^{{1 \over 2}}} {1 \over 1 + |u'|^{{1 \over 2}}} d x d t.
\]
These terms can thus be controlled in the same way as \eqref{eq:errorIntf2} above because $r'$ is comparable to $1 + s - t$ in the support of the integrand.

For the term involving $S$, we integrate by parts in $S$ as in \eqref{eq:Sibp} and note that, by Lemma~\ref{lem:gammaSgammabound}, we have that $|S(\gamma)| \le {C \over r'}$. Thus, the integral is controlled by
\begin{equation}
    \begin{aligned}
    C \int_0^s \int_{\Sigma_t} |\gamma| |S(\chi_\psi^c \chi_{f,I}^c \chi_\phi \chi_f)| |\Gamma^\sigma ((\partial_t \psi) (\partial_t \phi)^2)| |f| d x d t
    \\ + C \int_0^s \int_{\Sigma_t} \chi_\psi^c \chi_{f,I}^c \chi_\phi \chi_f |\Gamma^\sigma ((\partial_t \psi) (\partial_t \phi)^2)| {1 \over r'} |f| d x d t
    \\ + C \int_0^s \int_{\Sigma_t} \chi_\psi^c \chi_{f,I}^c \chi_\phi \chi_f |S \Gamma^\sigma ((\partial_t \psi) (\partial_t \phi)^2)| {1 \over r'} |f| d x d t.
    \end{aligned}
\end{equation}
We note that $S$ is well behaved when applied to any of the cutoff functions involved. Indeed, the discussion after \eqref{eq:Sibp} shows that $S$ is well behaved when applied to some of the cutoff functions involved, and controlling the others follows similarly after writing $S$ in terms of $L$, $\underline{L}'$, and $\partial_{\theta'}$ and then using Lemma~\ref{lem:du'dv'dtheta'}. More precisely, we have that
\begin{equation} \label{eq:Scutoffs1}
    \begin{aligned}
    |\gamma| |S(\chi_\psi^c \chi_{f,I}^c \chi_\phi \chi_f)| \le \left ({C \over \tau} + {C \over r'} \right ) \chi_\psi^c \chi_{f,I}^c \chi_\phi \chi_f,
    \end{aligned}
\end{equation}
where we recall that we are abusing notation and using the cutoffs on the right hand side to represent appropriate cutoffs over slightly larger regions. Thus, these terms can be controlled in the same way as \eqref{eq:errorIntf2} above. We also note that there are no boundary terms from the integration by parts because the integrand is compactly supported where $s - t \ge 10$ (because of the cutoff $\chi_{f,I}^c$) and $t \ge {\tau \over 10}$ (by Lemma~\ref{lem:rtr's-t}).

When $\tau \le \delta_0 s$, we must further decompose. We first consider the case of small $t$ (where $s - t$ is comparable to $s$) and the case of large $t$. In the case of large $t$, we will have to further decompose in terms of the behavior of $f$. This motivates us to introduce the cutoff function $\chi_t = \chi \left ({s - t \over 20 s} \right )$. This function is supported where $t \ge (1 - 20 \delta_0) s$, and it is equal to $1$ when $t \ge (1 - 10 \delta_0) s$. Thus, the function $\chi_t^c = 1 - \chi_t$ is supported where $t \le (1 - 10 \delta_0) s$.

When $t \le (1 - 10 \delta_0) s$, we note that $|S(\gamma)| \le {C \over \tau}$ by Lemma~\ref{lem:gammaSgammabound}. We then decompose $\partial_t$ in terms of $S$ as we have done several times before, giving us the integral

\begin{equation} \label{eq:worstint2tsmall}
    \begin{aligned}
    \int_0^s \int_{\Sigma_t} \chi_t^c \chi_\psi^c \chi_{f,I}^c \chi_f \chi_\phi \Gamma^\sigma ((\partial_t \phi)^2 (\partial_t \psi))
    \\ \times \left [-{1 \over 2} L' - \gamma S - {1 \over 2} \gamma \left (t - r' - {a (x - a) \over r'} \right ) L' - \gamma {a y \over (r')^2} \partial_{\theta'} \right ] f d x d t.
    \end{aligned}
\end{equation}

For the term involving $S$, we integrate by parts as we have done before, and the bound on $|\gamma|$ and $|S(\gamma)|$ given by Lemma~\ref{lem:gammaSgammabound} in the support of the integrand gives us that the resulting integral is controlled by
\begin{equation}
    \begin{aligned}
    \int_0^s \int_{\Sigma_t} |\gamma| |S(\chi_\psi^c \chi_{f,I}^c \chi_\phi \chi_f \chi_t^c)| |\Gamma^\sigma ((\partial_t \psi) (\partial_t \phi)^2)| |f| d x d t
    \\ + {C \over \tau} \int_0^s \int_{\Sigma_t} \chi_\psi^c \chi_{f,I}^c \chi_\phi \chi_f \chi_t^c |S \Gamma^\sigma ((\partial_t \psi) (\partial_t \phi)^2)| |f| d x d t
    \\ + {C \over \tau} \int_0^s \int_{\Sigma_t} \chi_\psi^c \chi_{f,I}^c \chi_\phi \chi_f \chi_t^c |\Gamma^\sigma ((\partial_t \psi) (\partial_t \phi)^2)| |f| d x d t.
    \end{aligned}
\end{equation}
We recall that we have good control over $|\gamma| |S(\chi_\psi^c \chi_{f,I}^c \chi_\phi \chi_f \chi_t^c)|$ (see \eqref{eq:Scutoffs1} and note that $S \chi_t \le C$).

Now, by Lemma~\ref{lem:rtr's-t}, we note that these integrals are supported in the region where $t \ge {\tau \over 10}$. We may thus control these terms in a similar same way as either \eqref{eq:worsterrorint1} or \eqref{eq:errorIntf2} above, giving us that the integrals are controlled by
\[
C \epsilon^{{9 \over 4}} {1 \over (1 + s)^{{1 \over 2}} (1 + \tau^{{3 \over 2} - \delta})},
\]
as desired (we note that we are in fact throwing away good powers of $\tau$ that we do not need). We can then use Lemma~\ref{lem:du'dv'dtheta'} and improved decay of good derivatives of $f$ to control the other terms in the error integral which do not involve $S$ in a similar way. Indeed, this gives us that those integrals are controlled by
\[
C \epsilon^{{9 \over 4}} \int_0^s \int_{\Sigma_t} \left ({1 \over \tau} + {1 \over r'} \right ) \chi_\psi^c \chi_{f,I}^c \chi_\phi \chi_f \chi_t^c |\Gamma^\sigma ((\partial_t \psi) (\partial_t \phi)^2)| {1 \over (1 + s - t)^{{1 \over 2}}} {1 \over 1 + |u'|^{{1 \over 2}}} d x d t,
\]
which can be dealt with in the same way as the terms above.

In the remaining region where $t \ge (1 - 20 \delta_0) s$, we introduce more cutoff functions adapted to $f$. We take $\chi_{f,\theta'} (\theta') = \chi(|\vartheta'|) = \chi(|\pi - \theta'|)$. This localizes us to the set of points where $|\vartheta'| \le \delta_0$.



We note that, in the support of $\chi_{f,I}'$ and $(\chi_{f,I} ^c)'$, we have that there exists some constant $c$ such that $1 + |u'| \ge c (1 + s - t)$.

Now, decomposing $\partial_t$ in terms of $S$ as we have done before, we have that
\[
\chi_f \chi_t \chi_{f,\theta'}^c \chi_{f,I} ^c \partial_t f = -{1 \over 2} \chi_f \chi_t \chi_{f,\theta'}^c \chi_{f,I} ^c (L' + \underline{L}') f.
\]
The term with the $L'$ derivative on $f$ is better, so we focus on the other term. We have that
\[
\chi_f \chi_t \chi_{f,\theta'}^c \chi_{f,I} ^c \underline{L}' f = \underline{L}' (\chi_f \chi_t \chi_{f,\theta'}^c \chi_{f,I} ^c f) - \underline{L'} (\chi_f \chi_t \chi_{f,\theta'}^c \chi_{f,I} ^c) f.
\]
Now, we have that
\[
|\underline{L}' (\chi_f \chi_t \chi_{f,\theta'}^c \chi_{f,I} ^c)| \le {C \over \tau} + {C \over s} + {C \over 1 + s - t}.
\]
Let $h = \chi_f \chi_t \chi_{f,\theta'}^c \chi_{f,I} ^c f$. We now use Lemma~\ref{lem:du'S} to write that
\[
\underline{L}' (h) = 2 \gamma S(h) + {\gamma \over r'} \left (t - r' - {a (x - a) \over r'} \right ) r' L' (h) + \gamma {2 a y \over (r')^2} \partial_{\theta'} (h)
\]
Now, we note that
\[
|r' L'(\chi_f \chi_t \chi_{f,\theta'}^c \chi_{f,I} ^c)| \le C,
\]
and similarly, we note that
\[
|\partial_{\theta'} (\chi_f \chi_t \chi_{f,\theta'}^c \chi_{f,I} ^c)| \le C.
\]
Thus, by Lemma~\ref{lem:du'dv'dtheta'} (note that $\vartheta' \ge \delta_0 / 2$ in the support of the integrand), we have that
\begin{equation}
    \begin{aligned}
    \left |\int_0^s \int_{\Sigma_t} \chi_\psi^c \chi_{f,I}^c \chi_\phi \chi_f \Gamma^\sigma ((\partial_t \psi) (\partial_t \phi)^2) (\partial_t f) d x d t \right |
    \\ \le C \left |\int_0^s \int_{\Sigma_t} \chi_\phi \chi_\psi^c \Gamma^\sigma ((\partial_t \psi) (\partial_t \phi)^2) \gamma S(\chi_f \chi_{f,I}^c \chi_t \chi_{f,\theta'}^c f) d x d t \right |
    \\ + \int_0^s \int_{\Sigma_t} \left ({C \over \tau} + {C \over 1 + s - t} \right ) \chi_\psi^c \chi_\phi \chi_f \chi_{f,I}^c \chi_t \chi_{f,\theta'}^c |\Gamma^\sigma ((\partial_t \psi) (\partial_t \phi)^2)| {1 \over (1 + s - t)^{{1 \over 2}}} {1 \over 1 + |u'|^{{1 \over 2}}} d x d t.
    \end{aligned}
\end{equation}
The second of these terms can be controlled in the same way as the terms above (see \eqref{eq:worsterrorint1} and \eqref{eq:errorIntf2}). For the term involving $S$, we integrate by parts in $S$ just as before. Because we have appropriate control over $S(\chi_\phi)$, $S(\chi_\psi^c)$, and $S(\gamma)$ in the region in question by Lemma~\ref{lem:gammaSgammabound}, the desired result follows in an analogous way as the terms above (more precisely, the integral can once again be controlled in the same way as either \eqref{eq:worsterrorint1} or \eqref{eq:errorIntf2}).

We now consider the same region where $|\vartheta'| \le \delta_0$. For this term, we shall not decompose $\partial_t$ in terms of $S$ and $\overline{\partial}_f$, and the integral is instead given by

\[
\int_0^s \int_{\Sigma_t} \chi_\psi^c \chi_{f,I}^c \chi_\phi \chi_f \chi_{f,\theta'} \Gamma^\sigma ((\partial_t \psi) (\partial_t \phi)^2) (\partial_t f) d x d t.
\]
In this region, we note that the Jacobian in $(\overline{r},r')$ coordinates is well behaved by Lemma~\ref{lem:dr'drbar}. Thus, we have that
\begin{equation}
    \begin{aligned}
    \left |\int_0^s \int_{\Sigma_t} \chi_\psi^c \chi_{f,I}^c \chi_\phi \chi_t \chi_f \chi_{f,\theta'} \Gamma^\sigma ((\partial_t \psi) (\partial_t \phi)^2) (\partial_t f) d x d t \right |
    \\ \le C \epsilon^{{9 \over 4}} {1 \over (1 + s)^{{3 \over 2} - \delta}} {1 \over 1 + \tau^{{3 \over 2} - 2 \delta}} \int_{{s \over 2}}^s \int_{\Sigma_t} \chi_\psi^c \chi_{f,I}^c \chi_\phi \chi_f \chi_{f,\theta'} {1 \over 1 + |\overline{u}|^{3 - 4 \delta}} {1 \over (1 + s - t)^{{1 \over 2}}} {1 \over 1 + |u'|^{{3 \over 2}}} d x d t
    \\ \le C \epsilon^{{9 \over 4}} {1 \over (1 + s)^{{3 \over 2} - \delta}} {1 \over 1 + \tau^{{3 \over 2} - 2 \delta}} \int_{s \over 2}^s \int_{\Sigma_t} {1 \over 1 + |\overline{u}|^{3 - 4 \delta}} {1 \over 1 + |u'|^{{3 \over 2}}} {1 \over (1 + s - t)^{{1 \over 2}}} d \overline{r} d r' d t
    \\ \le C \epsilon^{{9 \over 4}} {1 \over (1 + s)^{{3 \over 2} - \delta}} {1 \over 1 + \tau^{{3 \over 2} - 2 \delta}} \int_{{s \over 2}}^s {1 \over (1 + s - t)^{{1 \over 2}}} d t \le C \epsilon^{{9 \over 4}} {1 \over (1 + s)^{1 - \delta}} {1 \over 1 + \tau^{{3 \over 2} - 2 \delta}},
    \end{aligned}
\end{equation}
as desired.

We finally consider quartic terms. The integrals we must control are of the form
\[
\int_0^s \int_{\Sigma_t} \Gamma^\sigma ((\partial_t \psi)^4) (\partial_t f) d x d t.
\]
We begin with the region away from the light cone for $\psi$, which is the easiest to control.

We first assume that $\tau \le \delta_0 s$. Let $j_\tau$ denote the greatest integer which is less than or equal to ${\delta_0 \tau \over 10}$. After discretizing the integral within each $\Sigma_t$ in terms of $u$ and $u'$, we have that
\begin{equation}
    \begin{aligned}
    \left |\int_0^s \int_{\Sigma_t} \chi_\psi^c \Gamma^\sigma ((\partial_t \psi)^4) (\partial_t f) d x d t \right |
    \\ \le C \epsilon^3 \int_0^s \sum_{j = j_\tau}^\infty \sum_{k = 1}^\infty \int_{\Sigma_t} \chi_{j - 1 \le u \le j + 1} \chi_{k - 1 \le u' \le k + 1} {1 \over (1 + j)^{6 - 5 \delta}} {1 \over (1 + k)^{{3 \over 2}}} {1 \over (1 + t)^{2 - \delta}} {1 \over (1 + s - t)^{{1 \over 2}}} d x d t.
    \end{aligned}
\end{equation}

We note that the lower bound $j = j_\tau$ comes from the fact that $\chi_\psi^c = 0$ outside of this region.

Now, by Lemma~\ref{lem:annuliarea}, we have that
\[
\int_{\Sigma_t} \chi_{j - 1 \le u \le j + 1} \chi_{k - 1 \le u' \le k + 1} d x \le C (1 + t)^{{1 \over 2}}.
\]
Thus, we have that
\begin{equation}
    \begin{aligned}
    \left |\int_0^s \int_{\Sigma_t} \chi_\psi^c \Gamma^\sigma ((\partial_t \psi)^4) (\partial_t f) d x d t \right | &\le C \epsilon^3 {1 \over (1 + \tau)^{5 - 5 \delta}} \int_0^s {1 \over (1 + t)^{{3 \over 2} - \delta}} {1 \over (1 + s - t)^{{1 \over 2}}} d t
    \\ &\le C \epsilon^3 {1 \over (1 + \tau)^{5 - 5 \delta}} {1 \over (1 + s)^{{1 \over 2}}}.
    \end{aligned}
\end{equation}

We now assume that $\tau \ge \delta_0 s$. We then have that $1 + |u| \ge c (1 + s)$, meaning that we have that
\begin{equation}
    \begin{aligned}
    \left |\int_0^s \int_{\Sigma_t} \chi_\psi^c \Gamma^\sigma ((\partial_t \psi)^4) (\partial_t f) d x d t \right | &\le C \epsilon^3 {1 \over (1 + s)^{5 - 5 \delta}} \int_0^s {1 \over (1 + t)^{1 - \delta}} {1 \over (1 + s - t)^{{1 \over 2}}} d t
    \\ &\le C \epsilon^3 {1 \over (1 + s)^{{11 \over 2} - 6 \delta}},
    \end{aligned}
\end{equation}
as desired.

We now consider the integral in the region determined by $\chi_\psi$. We must once again consider two cases depending on the size of $\tau$ relative to $s$. We begin with the easier case given by $\tau \ge \delta_0 s$ when $\tau$ is comparable to $s$.

When $\tau \ge \delta_0 s$, we further consider two cases depending on how large $u'$ is compared to $\tau$ in the usual way. When $u' \le \delta_0 \tau$, we must control the integral
\[
\left |\int_0^s \int_{\Sigma_t} \chi_\psi \chi_f \Gamma^\sigma ((\partial_t \psi)^4) (\partial_t f) d x d t \right |.
\]
Decomposing $\partial_t$ in terms of the frame consisting of $S$ and good derivatives for $f$ as we have done several times, we are left with
\[
\left |\int_0^s \int_{\Sigma_t} \chi_\psi \chi_f \Gamma^\sigma ((\partial_t \psi)^4) \left [-{1 \over 2} L' - \gamma S - {1 \over 2} \gamma \left (t - r' - {a (x - a) \over r'} \right ) L' - \gamma {a y \over (r')^2} \partial_{\theta'} \right ] d x d t \right |.
\]

We integrate by parts in $S$ as we have done several times before and use Lemma~\ref{lem:du'dv'dtheta'} to bound the other terms directly. Using in addition Lemma~\ref{lem:gammaSgammabound}, we get that the resulting integrals are controlled by
\[
C \epsilon^3 {1 \over s} \int_0^s \int_{\Sigma_t} \chi_\psi \chi_f {1 \over (1 + t)^{2 - \delta}} {1 \over 1 + |u|^{6 - 8 \delta}} {1 \over (1 + s - t)^{{1 \over 2}}} {1 \over 1 + |u'|^{{1 \over 2}}} d x d t,
\]
where we have used that $r$ and $r'$ are comparable to $s$ where the integrand us supported by Lemma~\ref{lem:rtr's-t} (note that this also implies that $t$ and $s - t$ are comparable to $s$ in this region). We then go into $(\theta,u,u')$ coordinates as in Lemma~\ref{lem:thetauu'}, giving us that the integral is controlled by
\[
C \epsilon^3 {1 \over (1 + s)^{{5 \over 2} - \delta}} \int_0^{2 \pi} \int_{-1}^{\delta_0 \tau} \int_{-1}^{\delta_0 \tau} {1 \over 1 + |u|^{6 - 8 \delta}} {1 \over 1 + |u'|^{{1 \over 2}}} d u' d u d \theta \le C \epsilon^3 {1 \over (1 + s)^{2 - \delta}},
\]
as desired.

Now, when $u' \ge \delta_0 \tau$, we note that the integral is controlled by
\[
C \epsilon^3 {1 \over (1 + s)^{{3 \over 2}}} \int_0^s \int_{\Sigma_t} \chi_{r \le t + 1} {1 \over (1 + t)^{2 - \delta}} {1 \over 1 + |u|^{6 - 8 \delta}} {1 \over (1 + s - t)^{{1 \over 2}}} d x d t,
\]
where we are using that $\tau \ge \delta_0 s$. Thus, using the decay in $u$ to bound the integral on each $\Sigma_t$ by $C (1 + t)$, we have that this is controlled by
\[
C \epsilon^3 {1 \over (1 + s)^{{3 \over 2}}} \int_0^s {1 \over (1 + t)^{1 - \delta}} {1 \over (1 + s - t)^{{1 \over 2}}} d t \le C \epsilon^3 {1 \over (1 + s)^{2 - \delta}},
\]
as desired.


We finally consider the case where $\tau \le \delta_0 s$. We once again decompose relative to the distance to the light cone adapted to $f$. We first consider the region where $u' \ge \delta_0 \tau$. Now, we note that when $u' \ge 10 \tau$, we have that $\psi$ is identically $0$ by domain of dependence. Thus, we have that the integrand in question is supported where $\delta_0 \tau \le u' \le 10 \tau$. We shall use the fact that this interval in $u'$ is comparable to $\tau$ in length.


The integral we must control is given by
\[
\int_0^s \int_{\Sigma_t} \chi_\psi \chi_{\delta_0 \tau \le u' \le 10 \tau} \Gamma^\sigma ((\partial_t \psi)^4) (\partial_t f) d x d t.
\]
This integral is controlled by
\[
C \epsilon^3 {1 \over \tau^{{3 \over 2}}} \int_0^s \int_{\Sigma_t} \chi_\psi \chi_{\delta_0 \tau \le u' \le 10 \tau} {1 \over (1 + t)^{2 - \delta}} {1 \over 1 + |u|^{6 - 8 \delta}} {1 \over (1 + s - t)^{{1 \over 2}}} d x d t.
\]
We break up the integral into two regions in $t$, one from $0$ to $\tau$ and the other from $\tau$ to $s$. In the first region, we use the decay in $u$ to bound the integral on each $\Sigma_t$ by $C (1 + t)$, giving us that the integral is controlled by
\[
C \epsilon^3 {1 \over (1 + s)^{{1 \over 2}}} {1 \over \tau^{{3 \over 2}}} \int_0^\tau {1 \over (1 + t)^{1 - \delta}} d t \le C \epsilon^3 {1 \over (1 + s)^{{1 \over 2}}} {1 \over \tau^{{3 \over 2} - \delta}}.
\]

The remaining integral is from $\tau$ to $s$. Discretizing this integral in $u$ within each $\Sigma_t$ then gives us that this is controlled by
\[
C \epsilon^3 \sum_{j = -1}^\infty {1 \over \tau^{{3 \over 2}}} \int_\tau^s \int_{\Sigma_t} \chi_\psi \chi_{\delta_0 \tau \le u' \le 10 \tau} \chi_{j \le u \le j + 1} {1 \over (1 + t)^{2 - \delta}} {1 \over (2 + j)^{6 - 8 \delta}} {1 \over (1 + s - t)^{{1 \over 2}}} d x d t.
\]
Now, for $j$ fixed, we note that the inner integral in $x$ in the expression
\[
\int_\tau^s \int_{\Sigma_t} \chi_\psi \chi_{u'} \chi_{j \le u \le j + 1} {1 \over (1 + t)^{2 - \delta}} {1 \over (1 + s - t)^{{1 \over 2}}} d x d t
\]
is an integral over the intersection of two annular regions, one adapted to $\psi$ having thickness $1$ and the other adapted to $f$ having thickness comparable to $\tau$. Thus, by Lemma~\ref{lem:annuliarea}, we have that
\[
\int_{\Sigma_t} \chi_\psi \chi_{u'} \chi_{j \le u \le j + 1} d x \le C (1 + \sqrt{t}) \sqrt{\tau}
\]

Integrating the remaining integral in $t$ and summing in $j$ then gives us that the integral we desired to bound is controlled by
\[
C \epsilon^3 {1 \over (1 + s)^{{1 \over 2}}} {1 \over \tau^{{3 \over 2} - \delta}},
\]
as desired.

The only remaining region is where $u' \le \delta_0 \tau$. The integral we must control is given by
\[
\int_0^s \int_{\Sigma_t} \chi_\psi \chi_f \Gamma^\sigma ((\partial_t \psi)^4) (\partial_t f) d x d t.
\]
Decomposing $\partial_t$ in terms of $S$ and good derivatives for $f$ as we have done before results in the integral
\[
\int_0^s \int_{\Sigma_t} \chi_\psi \chi_f \Gamma^\sigma ((\partial_t \psi)^4) \left [-{1 \over 2} L' - \gamma S - {1 \over 2} \gamma \left (t - r' - {a (x - a) \over r'} \right ) \partial_{v'} - \gamma {a y \over (r')^2} \partial_{\theta'} \right ] f d x d t.
\]
We integrate by parts the term having $S$ and bound the others directly. We first focus on the term with $S$. Integrating by parts as we have several times before, we have that the integral is controlled by
\[
C \epsilon^3 {1 \over \tau} \int_{{\tau \over 10}}^{s - {\tau \over 10}} \int_{\Sigma_t} \chi_\psi \chi_f {1 \over (1 + t)^{2 - {\delta \over 2}}} {1 \over 1 + |u|^{6 - 8 \delta}} {1 \over (1 + s - t)^{{1 \over 2}}} {1 \over 1 + |u'|^{{1 \over 2}}} d x d t.
\]
We note that we are now using the slightly improved interpolation for the power of $t$ (i.e., a ${\delta \over 2}$ loss instead of a $\delta$ loss). This is because this integral will have a further logarithmic loss coming from Lemma~\ref{lem:annuliu'dec}. Indeed, discretizing the integral in $u$ (and taking $j_\tau$ to be the smallest integer such that $j_\tau \ge \delta_0 \tau$) gives us that this is controlled by
\[
C \epsilon^3 {1 \over \tau} \sum_{j = -1}^{j_\tau} {1 \over (2 + j)^{6 - 8 \delta}} \int_{{\tau \over 10}}^{s - {\tau \over 10}} \int_{\Sigma_t} \chi_f \chi_{j \le u \le j + 1} {1 \over (1 + t)^{2 - {\delta \over 2}}} {1 \over (1 + s - t)^{{1 \over 2}}} {1 \over 1 + |u'|^{{1 \over 2}}} d x d t.
\]
Now, by Lemma~\ref{lem:annuliu'dec}, we note that
\[
\int_{\Sigma_t} \chi_f \chi_{j \le u \le j + 1} {1 \over 1 + |u'|^{{1 \over 2}}} d x \le C (1 + \sqrt{t}) \log(1 + t) \le C (1 + t)^{{1 \over 2} + {\delta \over 2}}.
\]
Thus, the remaining integral is controlled by
\[
\int_{{\tau \over 10}}^{s - {\tau \over 10}} {1 \over (1 + t)^{{3 \over 2} - \delta}} {1 \over (1 + s - t)^{{1 \over 2}}} d t \le C (1 + s)^{-{1 \over 2}} \tau^{-{1 \over 2} + \delta}.
\]
Summing in $j$ then gives us that the whole integral is controlled by
\[
C \epsilon^3 {1 \over (1 + s)^{{1 \over 2}} \tau^{{3 \over 2} - \delta}},
\]
as desired. The terms not involving $S$ can be controlled in a similar way after using Lemma~\ref{lem:du'dv'dtheta'}.

Because we have now controlled all of the error integrals required to use Proposition~\ref{prop:decay}, we have recovered the pointwise bootstrap assumptions, and we have thus completed the proof of Theorem~\ref{thm:mainthm}.

\bibliographystyle{abbrv}
\bibliography{referencesmlc}

\begin{thebibliography}{10}

\bibitem{AdaFou03}
R.~A. Adams and J.~J.~F. Fournier.
\newblock {\em Sobolev spaces}, volume 140 of {\em Pure and Applied Mathematics
  (Amsterdam)}.
\newblock Elsevier/Academic Press, Amsterdam, second edition, 2003.

\bibitem{Ali10}
S.~Alinhac.
\newblock {\em Geometric analysis of hyperbolic differential equations: an
  introduction}, volume 374 of {\em London Mathematical Society Lecture Note
  Series}.
\newblock Cambridge University Press, Cambridge, 2010.

\bibitem{AndKla21}
J.~Anderson and S.~Klainerman.
\newblock Hyperbolic equations with multiple characteristics, In preparation.

\bibitem{AndPas19}
J.~Anderson and F.~Pasqualotto.
\newblock Global stability for nonlinear wave equations with multi-localized
  initial data, 2019.

\bibitem{AndZba20}
J.~Anderson and S.~Zbarsky.
\newblock Stability and instability of traveling wave solutions to nonlinear
  wave equations, 2020.

\bibitem{AndBlu15}
L.~Andersson and P.~Blue.
\newblock Hidden symmetries and decay for the wave equation on the {K}err
  spacetime.
\newblock {\em Ann. of Math. (2)}, 182(3):787--853, 2015.

\bibitem{Chr86}
D.~Christodoulou.
\newblock Global solutions of nonlinear hyperbolic equations for small initial
  data.
\newblock {\em Comm. Pure Appl. Math.}, 39(2):267--282, 1986.

\bibitem{Chr98}
D.~Christodoulou.
\newblock On the geometry and dynamics of crystalline continua.
\newblock {\em Ann. Inst. H. Poincar\'{e} Phys. Th\'{e}or.}, 69(3):335--358,
  1998.

\bibitem{ChrKl93}
D.~Christodoulou and S.~Klainerman.
\newblock {\em {The global nonlinear stability of the {M}inkowski space}},
  volume~41 of {\em {Princeton Mathematical Series}}.
\newblock Princeton University Press, Princeton, NJ, 1993.

\bibitem{CouHil62}
R.~Courant and D.~Hilbert.
\newblock {\em Methods of mathematical physics. {V}ol. {II}: {P}artial
  differential equations}.
\newblock (Vol. II by R. Courant.). Interscience Publishers (a division of John
  Wiley \& Sons), New York-Lon don, 1962.

\bibitem{DafHolRod19}
M.~Dafermos, G.~Holzegel, and I.~Rodnianski.
\newblock The linear stability of the {S}chwarzschild solution to gravitational
  perturbations.
\newblock {\em Acta Math.}, 222(1):1--214, 2019.

\bibitem{DafRodShl16}
M.~Dafermos, I.~Rodnianski, and Y.~Shlapentokh-Rothman.
\newblock Decay for solutions of the wave equation on {K}err exterior
  spacetimes {III}: {T}he full subextremal case {$|a|<M$}.
\newblock {\em Ann. of Math. (2)}, 183(3):787--913, 2016.

\bibitem{DenPus20}
Y.~Deng and F.~Pusateri.
\newblock On the global behavior of weak null quasilinear wave equations.
\newblock {\em Comm. Pure Appl. Math.}, 73(5):1035--1099, 2020.

\bibitem{GerMasSha09}
P.~Germain, N.~Masmoudi, and J.~Shatah.
\newblock Global solutions for 3{D} quadratic {S}chr\"{o}dinger equations.
\newblock {\em Int. Math. Res. Not. IMRN}, (3):414--432, 2009.

\bibitem{Hor97}
L.~H\"{o}rmander.
\newblock {\em Lectures on nonlinear hyperbolic differential equations},
  volume~26 of {\em Math\'{e}matiques \& Applications (Berlin) [Mathematics \&
  Applications]}.
\newblock Springer-Verlag, Berlin, 1997.

\bibitem{IfrTat15}
M.~Ifrim and D.~Tataru.
\newblock Global bounds for the cubic nonlinear {S}chr\"{o}dinger equation
  ({NLS}) in one space dimension.
\newblock {\em Nonlinearity}, 28(8):2661--2675, 2015.

\bibitem{JeoOh19}
I.-J. Jeong and S.-J. Oh.
\newblock On the cauchy problem for the hall and electron magnetohydrodynamic
  equations without resistivity i: illposedness near degenerate stationary
  solutions, 2019.

\bibitem{Joh81}
F.~John.
\newblock Blow-up for quasilinear wave equations in three space dimensions.
\newblock {\em Comm. Pure Appl. Math.}, 34(1):29--51, 1981.

\bibitem{JohKla84}
F.~John and S.~Klainerman.
\newblock Almost global existence to nonlinear wave equations in three space
  dimensions.
\newblock {\em Comm. Pure Appl. Math.}, 37(4):443--455, 1984.

\bibitem{Kei18}
J.~Keir.
\newblock The weak null condition and global existence using the p-weighted
  energy method, 2018.

\bibitem{Kla80}
S.~Klainerman.
\newblock Global existence for nonlinear wave equations.
\newblock {\em Comm. Pure Appl. Math.}, 33(1):43--101, 1980.

\bibitem{Kla82}
S.~Klainerman.
\newblock Long time behavior of solutions to nonlinear wave equations.
\newblock {\em Proceedings of the International Congress of Mathematicians,
  Warsaw, (1983)}, pages 1209--1215, 1983.

\bibitem{Kl85}
S.~Klainerman.
\newblock Uniform decay estimates and the {L}orentz invariance of the classical
  wave equation.
\newblock {\em Comm. Pure Appl. Math.}, 38(3):321--332, 1985.

\bibitem{Kla86}
S.~Klainerman.
\newblock The null condition and global existence to nonlinear wave equations.
\newblock In {\em Nonlinear systems of partial differential equations in
  applied mathematics, {P}art 1 ({S}anta {F}e, {N}.{M}., 1984)}, volume~23 of
  {\em Lectures in Appl. Math.}, pages 293--326. Amer. Math. Soc., Providence,
  RI, 1986.

\bibitem{KlaMac93}
S.~Klainerman and M.~Machedon.
\newblock Space-time estimates for null forms and the local existence theorem.
\newblock {\em Comm. Pure Appl. Math.}, 46(9):1221--1268, 1993.

\bibitem{KlaPon83}
S.~Klainerman and G.~Ponce.
\newblock Global, small amplitude solutions to nonlinear evolution equations.
\newblock {\em Comm. Pure Appl. Math.}, 36(1):133--141, 1983.

\bibitem{KlaRodTao02}
S.~Klainerman, I.~Rodnianski, and T.~Tao.
\newblock A physical space approach to wave equation bilinear estimates.
\newblock volume~87, pages 299--336. 2002.
\newblock Dedicated to the memory of Thomas H. Wolff.

\bibitem{KlaSid96}
S.~Klainerman and T.~C. Sideris.
\newblock On almost global existence for nonrelativistic wave equations in
  {$3$}{D}.
\newblock {\em Comm. Pure Appl. Math.}, 49(3):307--321, 1996.

\bibitem{KlaSze20}
S.~Klainerman and J.~Szeftel.
\newblock {\em {Global Nonlinear Stability of Schwarzschild Spacetime under
  Polarized Perturbations}}, volume 210 of {\em {Annals of Mathematics
  Studies}}.
\newblock Princeton University Press, Princeton, NJ, 2020.

\bibitem{Lie89}
O.~Liess.
\newblock Global existence for the nonlinear equations of crystal optics.
\newblock In {\em Journ\'{e}es ``\'{E}quations aux {D}\'{e}riv\'{e}es
  {P}artielles'' ({S}aint {J}ean de {M}onts, 1989)}, pages Exp. No. V, 11.
  \'{E}cole Polytech., Palaiseau, 1989.

\bibitem{Lie91}
O.~Liess.
\newblock Decay estimates for the solutions of the system of crystal optics.
\newblock {\em Asymptotic Anal.}, 4(1):61--95, 1991.

\bibitem{LinRod03}
H.~Lindblad and I.~Rodnianski.
\newblock The weak null condition for {E}instein's equations.
\newblock {\em C. R. Math. Acad. Sci. Paris}, 336(11):901--906, 2003.

\bibitem{LinRod05}
H.~Lindblad and I.~Rodnianski.
\newblock Global existence for the {E}instein vacuum equations in wave
  coordinates.
\newblock {\em Comm. Math. Phys.}, 256(1):43--110, 2005.

\bibitem{LinRod10}
H.~Lindblad and I.~Rodnianski.
\newblock The global stability of {M}inkowski space-time in harmonic gauge.
\newblock {\em Ann. of Math. (2)}, 171(3):1401--1477, 2010.

\bibitem{PusSha13}
F.~Pusateri and J.~Shatah.
\newblock Space-time resonances and the null condition for first-order systems
  of wave equations.
\newblock {\em Comm. Pure Appl. Math.}, 66(10):1495--1540, 2013.

\bibitem{SmiTat05}
H.~F. Smith and D.~Tataru.
\newblock Sharp local well-posedness results for the nonlinear wave equation.
\newblock {\em Ann. of Math. (2)}, 162(1):291--366, 2005.

\bibitem{Sogge08}
C.~D. Sogge.
\newblock {\em Lectures on non-linear wave equations}.
\newblock International Press, Boston, MA, second edition, 2008.

\end{thebibliography}

\end{document}